\documentclass[oneside,english]{amsart}
\usepackage[T1]{fontenc}
\usepackage[latin9]{luainputenc}
\usepackage{geometry}
\usepackage{array}
\usepackage{longtable}
\usepackage{float}
\usepackage{multirow}
\usepackage{amstext}
\usepackage{amsthm}
\usepackage{amssymb}
\usepackage{stmaryrd}
\usepackage{graphicx}
\usepackage{setspace}
\usepackage{wasysym}

\makeatletter

\providecommand{\tabularnewline}{\\}

\numberwithin{equation}{section}
\numberwithin{figure}{section}
\theoremstyle{plain}
\newtheorem{thm}{\protect\theoremname}
\theoremstyle{remark}
\newtheorem{rem}[thm]{\protect\remarkname}
\theoremstyle{definition}
\newtheorem{defn}[thm]{\protect\definitionname}
\theoremstyle{definition}
\newtheorem{example}[thm]{\protect\examplename}
\theoremstyle{plain}
\newtheorem{prop}[thm]{\protect\propositionname}
\theoremstyle{plain}
\newtheorem{cor}[thm]{\protect\corollaryname}

\makeatother

\usepackage{babel}
\providecommand{\corollaryname}{Corollary}
\providecommand{\definitionname}{Definition}
\providecommand{\examplename}{Example}
\providecommand{\propositionname}{Proposition}
\providecommand{\remarkname}{Remark}
\providecommand{\theoremname}{Theorem}

\begin{document}
\title{Centres of centralizers of nilpotent elements in Lie superalgebras
$\mathfrak{sl}(m|n)$ or $\mathfrak{osp}(m|2n)$}
\author{Leyu Han}
\begin{abstract}
\singlespacing{}\noindent Let $\bar{G}$ be the simple algebraic supergroup $\mathrm{SL}(m|n)$
or $\mathrm{OSp}(m|2n)$ over $\mathbb{C}$. Let $\mathfrak{g}=\mathrm{Lie}(\bar{G})=\mathfrak{g}_{\bar{0}}\oplus\mathfrak{g}_{\bar{1}}$
and let $G=\bar{G}(\mathbb{C})$ where $\mathbb{C}$ is considered
as a superalgebra concentrated in even degree. Suppose $e\in\mathfrak{g}_{\bar{0}}$
is nilpotent. We describe the centralizer $\mathfrak{g}^{e}$ of $e$
in{\normalsize{} }$\mathfrak{g}$ and its centre $\mathfrak{z}(\mathfrak{g}^{e})$.
In particular, we give bases for $\mathfrak{g}^{e}$, $\mathfrak{z}(\mathfrak{g}^{e})$
and $\left(\mathfrak{z}(\mathfrak{g}^{e})\right)^{G^{e}}$. We also
determine the labelled Dynkin diagram $\varDelta$ with respect to
$e$ and subsequently describe the relation between $\left(\mathfrak{z}(\mathfrak{g}^{e})\right)^{G^{e}}$
and $\varDelta$.
\end{abstract}

\maketitle
\begin{singlespace}

\section{Introduction}
\end{singlespace}

\begin{singlespace}
\noindent Let $\bar{G}$ be the simple algebraic supergroup $\mathrm{SL}(m|n)$
or $\mathrm{OSp}(m|2n)$ over $\mathbb{C}$. Let $\mathfrak{g}=\mathrm{Lie}(\bar{G})=\mathfrak{g}_{\bar{0}}\oplus\mathfrak{g}_{\bar{1}}$
and let $G=\bar{G}(\mathbb{C})$ where $\mathbb{C}$ is considered
as a superalgebra concentrated in even degree. Let $e\in\mathfrak{g}_{\bar{0}}$
be nilpotent and write $\mathfrak{g}^{e}=\{x\in\mathfrak{g}:[e,x]=0\}$
(resp. $G^{e}=\{g\in G:geg^{-1}=e\}$) for the centralizer of $e$
in $\mathfrak{g}$ (resp. $G$). The main result of this paper is
an explicit description of the centre of centralizer $\mathfrak{z}(\mathfrak{g}^{e})=\{x\in\mathfrak{g}^{e}:[x,y]=0\text{ for all }y\in\mathfrak{g}^{e}\}$
of $e$ in $\mathfrak{g}$ and its structure $\left(\mathfrak{z}(\mathfrak{g}^{e})\right)^{G^{e}}$
under the adjoint action of $G^{e}$ where $\left(\mathfrak{z}(\mathfrak{g}^{e})\right)^{G^{e}}=\{x\in\mathfrak{z}(\mathfrak{g}^{e}):gxg^{-1}=x\text{ for all }g\in G^{e}\}$.
The present paper is one of two papers in which we calculate bases
for $\mathfrak{g}^{e}$, $\mathfrak{z}(\mathfrak{g}^{e})$ and $\left(\mathfrak{z}(\mathfrak{g}^{e})\right)^{G^{e}}$
for $\mathfrak{g}$ a basic classical Lie superalgebra except $\mathfrak{psl}(n|n)$.
In this paper, we deal with Lie superalgebras $\mathfrak{sl}(m|n)$
and $\mathfrak{osp}(m|2n)$, while the other deals with exceptional
Lie superalgebras $D(2,1;\alpha)$, $G(3)$ and $F(4)$.
\end{singlespace}

Research on the centralizer and the centre of centralizer of nilpotent
elements in simple Lie algebras has been developed over the years
since Springer \cite{Springer1966} considered the centralizer $G^{u}$
of a unipotent element $u$ in a simple algebraic group $G$. For
$G$ is of exceptional type, further study of $G^{u}$ was carried
out in \cite{Chang1968}, \cite{Stuhler1971}, \cite{Shinoda1974},
\cite{Shoji1974}, \cite{Mizuno1977} and\cite{Mizuno1980}. A description
of the structure of the centralizer $\mathfrak{g}^{e}$ of nilpotent
elements in classical Lie algebras can be found in Jantzen's monograph,
\cite{Jantzen2004}. In \cite{Yakimova2009}, Yakimova identified
the centre of $\mathfrak{g}^{e}$ for classical Lie algebra $\mathfrak{g}$
over a field of characteristic zero. Note that the structure of $\left(\mathfrak{z}(\mathfrak{g}^{e})\right)^{G^{e}}$
described in this paper is a new result not only in the case of Lie
superalgebras, but also in the case of Lie algebras. In \cite{Seitz2000}
and \cite{Seitz2004}, Seitz further considered $Z(G^{u})$. Lawther--Testerman
studied the centralizer $G^{u}$ and its centre $Z(G^{u})$ over a
field of characteristic $0$ or a good prime in \cite{Lawther2008}.
They subsequently determined the dimension of the Lie algebra of $Z(G^{u})$.
In particular, Lawther--Testerman \cite{Lawther2008} used work of
Yakimova in \cite{Yakimova2009} to deal with classical cases. In
recent work, Liebeck and Seitz developed a new approach to classify
unipotent and nilpotent orbits in \cite{Liebeck2012}. 

To our best knowledge, there is a lot less study in this direction
in the case of Lie superalgebras. Finite-dimensional simple Lie superalgebras
over an algebraically closed field of characteristic zero were classified
by V. G. Kac in \cite{Kac1977}. After this classification, a wide
range of relevant problems have drawn the attention of mathematicians.
Determining the centralizers $\mathfrak{g}^{e}$ of nilpotent even
elements in $\mathfrak{g}=\mathfrak{gl}(m|n)$ and $\mathfrak{osp}(m|2n)$
for a field of zero or prime characteristic formed part of the results
of recent work of Wang and Zhao \cite{Wang2009} and Hoyt \cite{Hoyt2012}.
However, more than forty years after Kac's classification, the dimension
of the centre of $\mathfrak{g}^{e}$ where $\mathfrak{g}$ is a simple
Lie superalgebra still remains a mystery. We attempt to discover this
mystery in the present paper so that our knowledge about $\mathfrak{z}(\mathfrak{g}^{e})$
can be broaden to a similar level to that in the area of Lie algebras.

In the remaining part of this introduction, we give a more detailed
survey of our results.

\begin{singlespace}
In this paper, note that $\mathfrak{g}=\mathfrak{sl}(m|n)$ and $G=\{(A,B):A\in\mathrm{GL}_{m}(\mathbb{C}),B\in\mathrm{GL}_{n}(\mathbb{C})\text{ and }\mathrm{det}(A)=\mathrm{det}(B)\}$,
or $\mathfrak{g}=\mathfrak{osp}(m|2n)$ and $G=\mathrm{O}_{m}(\mathbb{C})\times\mathrm{Sp}_{2n}(\mathbb{C})$.
More generally, if $\mathfrak{g}$ is a direct sum of Lie superalgebras
of the form $\mathfrak{sl}(m_{i}|n_{i})$ or $\mathfrak{osp}(m_{i}|n_{i})$,
we define $G$ similarly. 
\end{singlespace}

In Subsection \ref{subsec:Lablled-Dynkin-diagrams}, we give a full
definition of the labelled Dynkin diagram with respect to $e$. We
classify the labelled Dynkin diagrams using Dynkin pyramids. Pyramids
are defined in \cite{Elashvili2005} for non-super case and generalized
in \cite{Hoyt2012} for super case. Each pyramid determines a nilpotent
element $e$ and a semisimple element $h$ such that $\{e,h\}$ can
be extended to an $\mathfrak{sl}(2)$-triple in $\mathfrak{g}_{\bar{0}}$.
We use $h$ to determine the labelled Dynkin diagram with respect
to $e$. In contrast to the case of non-super Lie algebras, in general
$e$ determines more than one labelled Dynkin diagram. 

\begin{singlespace}
Our results can be viewed as Lie superalgebra versions of those obtained
by Lawther and Testerman in \cite{Lawther2008}. They obtain four
theorems as direct results of their findings. In this paper, we obtain
analogues of Theorems 2--4 in \cite{Lawther2008} for Lie superalgebras
$\mathfrak{sl}(m|n)$ and $\mathfrak{osp}(m|2n)$. We take $\left(\mathfrak{z}(\mathfrak{g}^{e})\right)^{G^{e}}$
as the correct replacement for $Z(G^{e})$ since $\mathrm{Lie}(Z(G^{e}))=\left(\mathfrak{z}(\mathfrak{g}^{e})\right)^{G^{e}}$
for a field of characteristic $0$. Note that for $\mathfrak{g}=\mathfrak{sl}(m|n)$,
we have $\left(\mathfrak{z}(\mathfrak{g}^{e})\right)^{G^{e}}=\mathfrak{z}(\mathfrak{g}^{e})$.
\end{singlespace}

Fix $\varDelta$ to be a labelled Dynkin diagram with respect to $e$.
Denote by $n_{i}(\varDelta)$ the number of nodes which have labels
equal to $i$ in $\varDelta$. As a consequence of our calculations,
we observe that the choice of $\varDelta$ does not affect the following
theorems and labels in $\varDelta$ can only be $0,1$ or $2$. 

Our first theorem concerns the case where labels in $\varDelta$ are
even. 
\begin{thm}
\begin{singlespace}
\noindent Let $\mathfrak{g}=\mathfrak{g}_{\bar{0}}\oplus\mathfrak{g}_{\bar{1}}=\mathfrak{sl}(m|n)$
or $\mathfrak{osp}(m|2n)$ and $e\in\mathfrak{g}_{\bar{0}}$ be nilpotent.
Assume $\varDelta$ has no label equal to $1$, then 
\end{singlespace}

\noindent (1) $\dim\left(\mathfrak{z}(\mathfrak{g}^{e})\right)^{G^{e}}=n_{2}(\varDelta)=\dim\mathfrak{z}(\mathfrak{g}^{h})$
for \textup{$\mathfrak{g}=\mathfrak{sl}(m|n),m\neq n$ or }$\mathfrak{osp}(m|2n)$;

\begin{singlespace}
\noindent (2)\textup{ $\dim\left(\mathfrak{z}(\mathfrak{g}^{e})\right)^{G^{e}}-1=n_{2}(\varDelta)=\dim\mathfrak{z}(\mathfrak{g}^{h})$
}for \textup{$\mathfrak{g}=\mathfrak{sl}(n|n)$.}
\end{singlespace}
\end{thm}

\begin{singlespace}
\noindent Our next result gives a more general result relating $\dim\left(\mathfrak{z}(\mathfrak{g}^{e})\right)^{G^{e}}$
and $\varDelta$. 
\end{singlespace}
\begin{thm}
\begin{singlespace}
\noindent Let $\mathfrak{g}=\mathfrak{g}_{\bar{0}}\oplus\mathfrak{g}_{\bar{1}}=\mathfrak{sl}(m|n)$
or $\mathfrak{osp}(m|2n)$ and $e\in\mathfrak{g}_{\bar{0}}$ be nilpotent.
Let $a_{1},\dots,a_{l}$ be the labels in $\varDelta$. Then 
\[
\dim\left(\mathfrak{z}(\mathfrak{g}^{e})\right)^{G^{e}}=\left\lceil \frac{1}{2}\sum_{i=1}^{l}a_{i}\right\rceil +\varepsilon
\]
where the value of $\varepsilon$ is equal to $0$ except when $\mathfrak{g}=\mathfrak{sl}(n|n)$,
we have $\varepsilon=1$.
\end{singlespace}
\end{thm}

Theorem 1 is subsumed by a more general result as stated in Theorem
3 and the proof of which involves more techniques. In order to state
Theorem 3, we require some notations first. Define the sub-labelled
Dynkin diagram $\varDelta_{0}$ to be the \textit{$2$-free core of
$\varDelta$} such that $\varDelta_{0}$ is obtained by removing all
nodes with labels equal to $2$ from $\varDelta$. For $\mathfrak{g}=\mathfrak{sl}(m|n)$
(resp. $\mathfrak{g}=\mathfrak{osp}(m|2n)$), let $\lambda$ be a
partition of $(m|n)$ (resp. $(m|2n)$) and let $P$ be the Dynkin
pyramid (resp. ortho-symplectic Dynkin pyramid) of shape $\lambda$
which will be defined in Subsection \ref{subsec:Dynkin-pyramid-A(m,n)}
(resp. Subsection \ref{subsec:Hoyt's-ortho-symplectic-Dynkin}). Let
$r_{i}$ (resp. $s_{i}$) be the number of boxes on the $i$th column
with parity $\bar{0}$ (resp. $\bar{1}$) in $P$ and $k\geq0$ be
minimal such that the $k$th column in $P$ contains no boxes. Then
we define $\tau$ to be:

\begin{singlespace}
\noindent (1) the number of $i$ such that $r_{i}=s_{i}\neq0$ for
all $i>k$ or $i<-k$ when $\mathfrak{g}=\mathfrak{sl}(m|n)$;

\noindent (2) the number of $i$ such that $r_{i}=s_{i}\neq0$ for
all $i>k$ when $\mathfrak{g}=\mathfrak{osp}(m|2n)$.

\noindent We also define 
\[
\nu_{0}=\begin{cases}
0 & \text{if }\mathfrak{g}=\mathfrak{sl}(m|n),\sum_{-k<i<k}r_{i}\neq\sum_{-k<i<k}s_{i}\text{ or }\mathfrak{g}=\mathfrak{osp}(m|2n);\\
1 & \text{if }\mathfrak{g}=\mathfrak{sl}(m|n),\sum_{-k<i<k}r_{i}=\sum_{-k<i<k}s_{i}.
\end{cases}
\]

\end{singlespace}
\begin{thm}
\begin{singlespace}
\noindent Let $\mathfrak{g}=\mathfrak{g}_{\bar{0}}\oplus\mathfrak{g}_{\bar{1}}=\mathfrak{sl}(m|n)$
or $\mathfrak{osp}(m|2n)$ and $e\in\mathfrak{g}_{\bar{0}}$ be nilpotent.
Let $\varDelta_{0}$ be the 2-free core of $\varDelta$. Let $\mathfrak{g}_{0}$
be the subalgebra of $\mathfrak{g}$ generated by the root vectors
corresponding to the simple roots in $\varDelta_{0}$, then $\mathfrak{g}_{0}$
is a direct sum of simple Lie superalgebras. Let $G_{0}$ be the subgroup
of $G$ defined as above. There exists a nilpotent $G_{0}$-orbit
in $(\mathfrak{g}_{0})_{\bar{0}}$ having labelled Dynkin diagram
$\varDelta_{0}$. Suppose $e_{0}\in(\mathfrak{g}_{0})_{\bar{0}}$
is a representative of this orbit, then 

\noindent (1). $\dim\mathfrak{g}^{e}-\dim\mathfrak{g}_{0}^{e_{0}}=n_{2}(\varDelta)$
for all $\mathfrak{g}$; 

\noindent (2). When $\mathfrak{g}=\mathfrak{sl}(m|n),m\neq n$ or
$\mathfrak{osp}(m|2n)$, then $\dim\left(\mathfrak{z}(\mathfrak{g}^{e})\right)^{G^{e}}-\dim\left(\mathfrak{z}(\mathfrak{g}_{0}^{e_{0}})\right)^{G_{0}^{e_{0}}}=n_{2}(\varDelta)-\tau-\nu_{0}$;

\noindent (3). When $\mathfrak{g}=\mathfrak{sl}(n|n)$, then $\dim\left(\mathfrak{z}(\mathfrak{g}^{e})\right)^{G^{e}}-\dim\left(\mathfrak{z}(\mathfrak{g}_{0}^{e_{0}})\right)^{G_{0}^{e_{0}}}=n_{2}(\varDelta)+1-\tau-\nu_{0}$.
\end{singlespace}
\end{thm}

\begin{singlespace}
This paper is organized as follows. In Section \ref{sec:Preliminaries},
we recall some fundamental concepts of Lie superalgebras such as basic
classical Lie superalgebras, root system and labelled Dynkin diagrams.
In Sections \ref{sec:gl(m,n)}--\ref{sec:osp}, we recall the concept
of Dynkin pyramid and ortho-symplectic Dynkin pyramid. We use this
to determine a nilpotent element $e\in\mathfrak{g}_{\bar{0}}$ and
the labelled Dynkin diagram with respect to $e$ explicitly, and to
calculate bases of the centralizers $\mathfrak{g}^{e}$ and centres
$\mathfrak{z}(\mathfrak{g}^{e})$. In Subsection \ref{subsec:Adjoint-action-of-osp},
we find a basis for $\left(\mathfrak{z}(\mathfrak{g}^{e})\right)^{G^{e}}$
where $\mathfrak{g}=\mathfrak{osp}(m|2n)$ and $G=\mathrm{O}_{m}(\mathbb{C})\times\mathrm{Sp}_{2n}(\mathbb{C})$. 
\end{singlespace}
\begin{singlespace}

\section{Preliminaries\label{sec:Preliminaries}}
\end{singlespace}

\subsection{Basic classical Lie superalgebras}

Let $\mathfrak{g}=\mathfrak{g}_{\bar{0}}\oplus\mathfrak{g}_{\bar{1}}$
be a finite-dimensional simple Lie superalgebra over $\mathbb{C}$.
Recall that $\mathfrak{g}$ is called a basic classical Lie superalgebra
if $\mathfrak{g}_{\bar{0}}$ is a reductive Lie algebra and $\mathfrak{g}$
has a non-degenerate even invariant supersymmetric bilinear form $(\cdotp,\cdotp)$.
Finite-dimensional complex simple Lie superalgebras were classified
by V. G. Kac in \cite[Theorem 5]{Kac1977}. The simple basic classical
Lie superalgebras that are not Lie algebras consist of classical types
which are infinite families and three exceptional types. The infinite
series are $\mathfrak{sl}(m|n)$, $\mathfrak{psl}(n|n)$ and $\mathfrak{osp}(m|2n)$.
In this paper, we consider Lie superalgebras $\mathfrak{sl}(m|n)$
and $\mathfrak{osp}(m|2n)$.

\begin{singlespace}
The following definitions can be found in \cite[Chapter 1]{Cheng2012}.
Let $\text{\ensuremath{\mathfrak{h}}}$ be a Cartan subalgebra of
$\mathfrak{g}$. There exists a root space decomposition $\mathfrak{g}=\mathfrak{h}\oplus\bigoplus_{\alpha\in\mathfrak{h}^{*}}\mathfrak{g_{\alpha}}$
where $\mathfrak{g_{\alpha}}:=\{x\in\mathfrak{g}:[h,x]=\alpha(h)x\ \text{for\ all}\ h\in\mathfrak{\mathfrak{h}}\}$
is the root space corresponding to $\alpha$ and we have $\mathfrak{h}=\mathfrak{g}_{0}$.
The set $\Phi=\{\alpha\in\mathfrak{h^{*}:\alpha\neq}0,\mathfrak{g}_{\alpha}\neq0\}$
is called the root system of $\mathfrak{g}$. The set of even (resp.
odd) roots is defined to be $\Phi_{\bar{0}}=:\{\alpha\in\Phi:\mathfrak{g}_{\alpha}\subseteq\mathfrak{g}_{\bar{0}}\}$
(resp. $\Phi_{\bar{1}}=:\{\alpha\in\Phi:\mathfrak{g}_{\alpha}\subseteq\mathfrak{g}_{\bar{1}}\}$).
A set of positive roots for a root system is a set $\Phi^{+}\subseteq\Phi$
such that all root $\alpha\in\Phi$ there is exactly one of $\alpha$,$-\alpha$
contained in $\Phi^{+}$; and for any two distinct roots $\alpha,\beta\in\Phi^{+}$,
$\alpha+\beta\in\Phi$ implies that $\alpha+\beta\in\Phi^{+}$. Note
that $\Phi^{+}=\Phi_{\bar{0}}^{+}\cup\Phi_{\bar{1}}^{+}$. The set
of simple roots $\varPi=\{\alpha_{1},...,\alpha_{l}\}\subseteq\Phi^{+}$
consists of the positive roots that cannot be written as sums of positive
roots. Note that $l$ does not depend on choice of $\varPi$ and we
call it the rank of $\mathfrak{g}$.

A Borel subalgebra $\mathfrak{b}$ of $\mathfrak{g}$ is defined to
be a solvable subalgebra such that $\mathfrak{b}=\mathfrak{h}\oplus\bigoplus_{\alpha\in\Phi^{+}}\mathfrak{g}_{\alpha}$.
Note that there are in general many inequivalent conjugacy classes
of Borel subalgebras and every Borel subalgebra containing $\mathfrak{h}$
determines a corresponding system of positive roots $\Phi^{+}$. Consequently
$\mathfrak{b}$ determines a system of simple roots $\varPi$. Recall
that the Weyl group $W$ of $\mathfrak{g}$ is generated by the Weyl
reflections $w_{\alpha}(\beta)=\beta-2\frac{(\alpha,\beta)}{(\alpha,\alpha)}\alpha$
with $\alpha\in\Phi_{\bar{0}},\beta\in\Phi$. For each conjugacy class
of Borel subalgebras of $\mathfrak{g}$, a simple root system can
be transformed into an equivalent one under the transformation of
the Weyl group $W$ of $\mathfrak{g}$, see \cite[Subsection 2.3]{Frappat1989}. 
\end{singlespace}

\subsection{Dynkin diagrams\label{subsec:Dynkin-diagrams}}

\begin{singlespace}
In this subsection, we continue working with $\mathfrak{g}$ as in
Subsection 2.1. We next recall the concept of the Dynkin diagram as
defined for example in \cite[Section 2.2]{Frappat1989}. 

We know that there exists a non-degenerate even invariant supersymmetric
bilinear form $(\cdotp,\cdotp)$ on $\mathfrak{g}$. One can check
that $(\cdotp,\cdotp)$ restricts to a non-degenerate symmetric bilinear
form on $\mathfrak{h}$. Therefore, there exists an isomorphism from
$\mathfrak{h}$ to $\mathfrak{h}^{*}$ which provides a symmetric
bilinear form on $\mathfrak{h}^{*}$. Then the \textit{Dynkin diagram}
of a Lie superalgebra $\mathfrak{g}$ with a simple root system $\varPi$
is a graph where the vertices are labelled by $\varPi$ and there
are $\mu_{\alpha\beta}$ lines between the vertices labelled by simple
roots $\alpha_{i}$ and $\alpha_{j}$ such that:
\begin{equation}
\mu_{\alpha\beta}=\begin{cases}
\ensuremath{\vert(\alpha_{i},\alpha_{j})\ensuremath{\vert}} & \text{if }(\alpha_{i},\alpha_{i})=(\alpha_{j},\alpha_{j})=0,\\
\frac{2\ensuremath{\vert}(\alpha_{i},\alpha_{j})\ensuremath{\vert}}{min\{\vert(\alpha_{i},\alpha_{i})\vert,\ensuremath{\vert}(\alpha_{j},\alpha_{j})\ensuremath{\vert}\}} & \text{if }(\alpha_{i},\alpha_{i})(\alpha_{j},\alpha_{j})\neq0,\\
\frac{2\ensuremath{\vert}(\alpha_{i},\alpha_{j})\ensuremath{\vert}}{min_{(\alpha_{k},\alpha_{k})\neq0}\ensuremath{\vert}(\alpha_{k},\alpha_{k})\ensuremath{\vert}} & \text{if }(\alpha_{i},\alpha_{i})\neq0,\ (\alpha_{j},\alpha_{j})=0\ \text{\text{and}\ }\alpha_{k}\in\Phi.
\end{cases}\label{eq:lines-=0003BC}
\end{equation}
 We say a root $\alpha\in\Phi$ is \textit{isotropic} if $(\alpha,\alpha)=0$
and is \textit{non-isotropic} if $(\alpha,\alpha)\neq0$. We associate
a white node $\ocircle$ to each even root, a grey node $\varotimes$
to each odd isotropic root and a black node $\newmoon$ to each odd
non-isotropic root. Moreover, when $\mu_{\alpha\beta}>1$, we put
an arrow pointing from the vertex labelled by $\alpha_{i}$ to the
vertex labelled by $\alpha_{j}$ if $(\alpha_{i},\alpha_{i})(\alpha_{j},\alpha_{j})\neq0$
and $(\alpha_{i},\alpha_{i})>(\alpha_{j},\alpha_{j})$ or if $(\alpha_{i},\alpha_{i})=0,(\alpha_{j},\alpha_{j})\neq0$
and $\vert(\alpha_{j},\alpha_{j})\vert<2$, or pointing from the vertex
labelled by $\alpha_{j}$ to the vertex labelled by $\alpha_{i}$
if $(\alpha_{i},\alpha_{i})=0,(\alpha_{j},\alpha_{j})\neq0$ and $\vert(\alpha_{j},\alpha_{j})\vert>2$.
If the value of $\mu_{\alpha\beta}$ is not a natural number, then
we label the edge between vertices corresponding to roots $\alpha$
and $\beta$ with $\mu_{\alpha\beta}$ instead of drawing multiple
lines between them. Note that the Dynkin diagram depends on $\varPi$
up to conjugacy by $W$, thus Dynkin diagrams of $\mathfrak{g}$ for
different choices of simple roots can be different.
\end{singlespace}
\begin{rem}
In this paper, the Dynkin diagram for $\mathfrak{sl}(n|n)$ is by
convention the one for $\mathfrak{gl}(n|n)$.
\end{rem}

\subsection{Labelled Dynkin diagrams\label{subsec:Lablled-Dynkin-diagrams}}

\noindent Let $e\in\mathfrak{g}_{\bar{0}}$ be nilpotent. There exists
an $\mathfrak{sl}(2)$-triple $\{e,h,f\}\subseteq\mathfrak{g}_{\bar{0}}$
by the Jacobson--Morozov Theorem, see for example \cite[Theorem 3.3.1]{Chllingwood1993}.
An $\mathfrak{sl}(2)$-triple determines a grading on $\mathfrak{g}$
according to the eigenvalues of ad$h,$ thus we can decompose $\mathfrak{g}$
into its ad$h$-eigenspaces $\mathfrak{g}=\bigoplus_{j\in\mathbb{Z}}\mathfrak{g}(j)$
where $\mathfrak{g}(j)=\{x\in\mathfrak{g}:[h,x]=jx\}$. In order to
construct the labelled Dynkin diagram of $e$, we first fix $\mathfrak{h}\subseteq\mathfrak{g}(0)$
to be a Cartan subalgebra of $\mathfrak{g}$ containing $h$. Then
we choose a system $\Phi^{+}(0)$ of positive roots for $\mathfrak{g}(0)$
to get a Borel subalgebra $\mathfrak{b}(0)$ of $\mathfrak{g}(0)$.
Let $\mathfrak{b}$ be the Borel subalgebra such that $\mathfrak{b}=\mathfrak{b}(0)\oplus\mathfrak{g}(j>0)$,
then we obtain the corresponding system of positive roots $\Phi^{+}$
and a system of simple roots $\varPi=\{\alpha_{1},...,\alpha_{l}\}$
which will give a Dynkin diagram of $\mathfrak{g}$. Furthermore,
for each $i=1,...,l$, note that $\mathfrak{g}_{\alpha_{i}}$ is the
root space corresponding to $\alpha_{i}$ and $\mathfrak{g}_{\alpha_{i}}\subseteq\mathfrak{g}(j_{i})$
for some $j_{i}\geq0$. Hence, we have $\alpha_{i}(h)\geq0$ for $i=1,...,l$. 
\begin{defn}
\begin{singlespace}
\noindent The \textit{labelled Dynkin diagram} $\varDelta$ of $e$
determined by $\varPi$ is given by taking the Dynkin diagram of $\mathfrak{g}$
and labelling each node $\alpha$ with $\alpha(h)$.
\end{singlespace}
\end{defn}

\section{Lie superalgebras $\mathfrak{sl}(m|n)$\label{sec:gl(m,n)}}

\subsection{Dynkin pyramids and labelled Dynkin diagram for $\mathfrak{sl}(m|n)$\textit{\label{subsec:Dynkin-pyramid-A(m,n)}}}

\begin{singlespace}
\noindent Let $V=V_{\bar{0}}\oplus V_{\bar{1}}$ be a finite-dimensional
vector superspace such that $\dim V_{\bar{0}}=m$ and $\dim V_{1}=n$.
Let $\mathfrak{g}=\mathfrak{g}_{\bar{0}}\oplus\mathfrak{g}_{\bar{1}}=\mathfrak{sl}(m|n)=\mathfrak{sl}(V)$.
Note that the nilpotent orbits in $\mathfrak{g}_{\bar{0}}$ are parametrized
by the partitions of $(m|n)$. Let $\lambda$ be a partition of $(m|n)$
such that 
\begin{equation}
\lambda=(p|q)=(p_{1},\dots,p_{r}|q_{1},\dots,q_{s})\label{eq:gl(m,n)-partition 0}
\end{equation}
where $p$ (resp. $q$) is a partition of $m$ (resp. $n$) and $p_{1}\geq\dots\geq p_{r},q_{1}\geq\dots\geq q_{s}$.
By rearranging the order of numbers in $(p|q)$, we write 
\begin{equation}
\lambda=(\lambda_{1},\dots,\lambda_{r+s})\label{eq:gl(m,n)-partition 1}
\end{equation}
 such that $\lambda_{1}\geq\dots\geq\lambda_{r+s}$. For $i=1,\dots,r+s$,
we define $\ensuremath{\left|i\right|}$ such that for $c\in\mathbb{Z}$,
we have that $\ensuremath{\left|\{i:\lambda_{i}=c,\ensuremath{\left|i\right|}=\bar{0}\}\right|}=\ensuremath{\left|\{j:p_{j}=c\}\right|}$,
$\left|\{i:\lambda_{i}=c,\ensuremath{\left|i\right|}=\bar{1}\}\right|=\ensuremath{\left|\{j:q_{j}=c\}\right|}$
and if $\lambda_{i}=\lambda_{j}$, $\ensuremath{\left|i\right|}=\bar{0},\ensuremath{\left|j\right|}=\bar{1}$,
then $i<j$. i.e. $\sum_{\ensuremath{\left|i\right|}=\bar{0}}\lambda_{i}=m$
and $\sum_{\ensuremath{\left|i\right|}=\bar{1}}\lambda_{i}=n$.
\end{singlespace}

\begin{singlespace}
For the purpose of describing the dimension of $\mathfrak{g}^{e}$
and the labelled Dynkin diagram for each nilpotent orbit $e\in\mathfrak{g}_{\bar{0}}$,
we recall that a Dynkin pyramid $P$ of shape $\lambda$ is defined
to be a finite collection of boxes of size $2\times2$ in the $xy$-plane
which are centred at integer coordinates, see \cite[Section 4]{Elashvili2005}
and \cite[Section 7]{Hoyt2012}. Write $\lambda=(\lambda_{1},\dots,\lambda_{r+s})$
as defined in (\ref{eq:gl(m,n)-partition 1}), we number rows of the
Dynkin pyramid $P$ from $1,\dots,r+s$ from bottom to top, then the
$j$th row of $P$ has length $\lambda_{j}$ and we mark boxes in
the row $j$ with parity $\bar{0}$ or $\bar{1}$ according to $\vert j\vert$.
We associate a numbering $1,2,...,m+n$ for the boxes of the Dynkin
pyramid from top to bottom and from left to right. The row number
of the $i$th box is denoted by $\mathrm{row}(i)$ and the parity
of $\mathrm{row}(i)$ is denoted by $\ensuremath{\left|\mathrm{row}(i)\right|}$.
We say that the column number $\mathrm{col}(i)$ of $i$ is the $x$-coordinate
of the centre of the $i$th box. 

Define a basis $\{v_{1},\dots,v_{m+n}\}$ of $V=V_{\bar{0}}\oplus V_{\bar{1}}$
where $v_{i}\in V_{\ensuremath{\left|\mathrm{row}(i)\right|}}$. According
to \cite[Section 7]{Hoyt2012}, $P$ determines a nilpotent element
$e\in\mathfrak{g}_{\bar{0}}$ such that $e$ sends $v_{i}$ to $v_{j}$
if the box labelled by $j$ is the left neighbour of $i$ and sends
$v_{i}$ to zero if the box labelled by $v_{i}$ has no left neighbour.
Write $e_{ij}$ for the $ij$-matrix unit, then 
\[
e=\sum_{\mathrm{row}(i)=\mathrm{row}(j),\mathrm{col}(j)=\mathrm{col}(i)-2}e_{ij}
\]
 for all $1\leq i,j\leq m+n$. We know that $P$ also determines a
semisimple element $h\in\mathfrak{g}_{\bar{0}}$ such that $h=\sum_{i=1}^{m+n}-\text{col}(i)e_{ii}$.
Note that $\{e,h\}$ can be extended to an $\mathfrak{sl}(2)$-triple
$\{e,h,f\}$ in $\mathfrak{g}_{\bar{0}}$ according to \cite[Section 7]{Hoyt2012}.

Next we recall a construction of the root system for $\mathfrak{g}$
according to \cite[Section 2.2]{Musson2012}. Let $\mathfrak{h}$
be the Cartan subalgebra of $\mathfrak{g}$ consisting of all diagonal
matrices in $\mathfrak{g}$ and let $a=\mathrm{diag}(a_{1},\dots,a_{m+n})\in\mathfrak{h}$.
For the basis $\{v_{1},\dots,v_{m+n}\}$ of $V$, we define $\{\varepsilon_{1},\dots,\varepsilon_{m+n}\}\subseteq\mathfrak{h}^{*}$
such that $\varepsilon_{i}(a)=a_{i}$ for $i=1,\dots,m+n$ and the
parity of $\varepsilon_{i}$ is equal to $\ensuremath{\left|\mathrm{row}(i)\right|}$.
Then the root system of $\mathfrak{g}$ with respect to $\mathfrak{h}$
is $\Phi=\Phi_{\bar{0}}\cup\Phi_{\bar{1}}$ where $\Phi_{\bar{0}}=\{\varepsilon_{i}-\varepsilon_{j}:i\neq j,|i|=|j|\}$
and $\Phi_{\bar{1}}=\{\varepsilon_{i}-\varepsilon_{j}:i\neq j,|i|\neq|j|\}$
and $(\varepsilon_{i},\varepsilon_{j})=(-1)^{\ensuremath{|i|}}\delta_{ij}$.
By computing $(\varepsilon_{i}-\varepsilon_{j},\varepsilon_{i}-\varepsilon_{j})$
for all $i\neq j,|i|\neq|j|$, we have that all odd roots are isotropic.

The labelled Dynkin diagram for a nilpotent orbit $e\in\mathfrak{g}_{\bar{0}}$
is constructed as follows. First, draw the Dynkin pyramid of shape
$\lambda$. For $i=1,...,m+n-1,$ we associate a white node $\ocircle$
(resp. a grey node $\varotimes$) for root $\alpha_{i}=\varepsilon_{i}-\varepsilon_{i+1}$
if $\ensuremath{\left|\mathrm{row}(i+1)\right|}=\ensuremath{\left|\mathrm{row}(i)\right|}$
(resp. $\ensuremath{\left|\mathrm{row}(i+1)\right|}\neq\ensuremath{\left|\mathrm{row}(i)\right|}$).
We label the $i$th node with the value $\mathrm{col}(i+1)-\mathrm{col}(i)$. 
\end{singlespace}
\begin{example}
\begin{singlespace}
\noindent For a partition $\lambda=(5,1|3)$, the Dynkin pyramid of
shape $\lambda$ is

\noindent \ \ \ \ \ \ \ \ \ \ \ \ \ \ \ \ \ \ \ \ \ \ \ \ \ \ \ \ \ \ \ \ \ \ \ \ \ \ \ \ \ \ \ \includegraphics[scale=0.9]{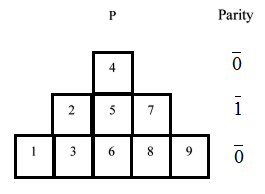}

\noindent We calculate that $h=\text{diag}(4,2,2,0,0,0,-2,-2,-4)$.
Then the corresponding labelled Dynkin diagram is 

\noindent \ \ \ \ \ \ \ \ \ \ \ \ \ \ \ \ \ \ \ \ \ \ \ \ \ \ \ \ \ \ \ \includegraphics{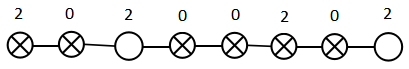}
\end{singlespace}
\end{example}

\begin{rem}
\begin{singlespace}
\noindent Different numberings within columns and different choices
of parities of rows are possible and would lead to different labelled
Dynkin diagrams. In this way one can get all possible labelled Dynkin
diagrams. In this paper, the way we define the partition $\lambda$
in (\ref{eq:gl(m,n)-partition 1}) leads to a unique labelled Dynkin
diagram. 
\end{singlespace}
\end{rem}

\begin{singlespace}
The following example shows that for different choices of parities
of rows of a given pyramid, we can get different Dynkin diagrams.
\end{singlespace}
\begin{example}
\begin{singlespace}
\noindent \label{exa:(3,2,2,1)}For a partition $\lambda=(3,2|2,1)$,
the corresponding Dynkin pyramid and labelled Dynkin diagram are shown
below:

\noindent \ \ \ \ \ \ \ \ \ \ \ \ \ \ \ \ \ \ \ \ \ \ \ \ \ \ \ \ \ \ \ \ \ \ \ \ \ \ \ \ \ \ \ \ \ \ \ \ \ \ \ \ \ \ \ \ \includegraphics[scale=0.9]{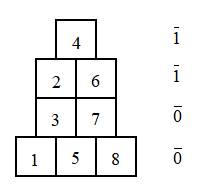}

\noindent \ \ \ \ \ \ \ \ \ \ \ \ \ \ \ \ \ \ \ \ \ \ \ \ \ \ \ \ \includegraphics{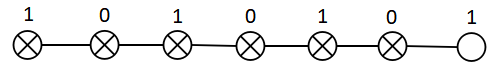}
\end{singlespace}

\noindent However, if we allow different parities for rows, then there
exist another Dynkin pyramid and therefore a different labelled Dynkin
diagram:

\noindent \ \ \ \ \ \ \ \ \ \ \ \ \ \ \ \ \ \ \ \ \ \ \ \ \ \ \ \ \ \ \ \ \ \ \ \ \ \ \ \ \ \ \ \ \ \ \ \ \ \ \ \ \ \ \ \ \includegraphics[scale=0.9]{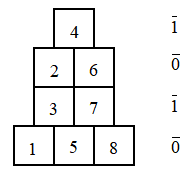}
\end{example}

\begin{singlespace}
\noindent \ \ \ \ \ \ \ \ \ \ \ \ \ \ \ \ \ \ \ \ \ \ \ \ \ \ \ \ \includegraphics{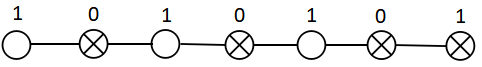}
\end{singlespace}

\subsection{Centralizer of nilpotent elements $e\in\mathfrak{g}_{\bar{0}}$\label{subsec:c-gl(m,n)}}

\begin{singlespace}
\noindent Let $e\in\mathfrak{g}_{\bar{0}}$ be a nilpotent element
with Jordan type $\lambda=(\lambda_{1},\dots,\lambda_{r+s})$ which
is defined as in (\ref{eq:gl(m,n)-partition 1}). In order to calculate
the dimension of $\mathfrak{\mathfrak{g}}^{e}$, we first recall a
basis for $\mathfrak{gl}(m|n)^{e}$ based on \cite[Section 1]{Yakimova2009}
and \cite[Section 3.2]{Hoyt2012}.
\end{singlespace}

\begin{singlespace}
Let $\mathfrak{\bar{g}}=\bar{\mathfrak{g}}_{\bar{0}}\oplus\bar{\mathfrak{g}}_{\bar{1}}=\mathfrak{gl}(m|n)$.
Let $u_{1},\dots,u_{r+s}\in V$ such that $u_{i}=v_{k}$ for $\mathrm{\mathrm{row}}(k)=i$
and $\mathrm{col}(k)=\lambda_{k}$, then the vectors $e^{j}u_{i}$
with $0\leq j\leq\lambda_{i}-1,\ \ensuremath{\left|i\right|}=\bar{0}$
form a basis for $V_{\bar{0}}$ and the vectors $e^{j}u_{i}$ with
$0\leq j\leq\lambda_{i}-1,\ \ensuremath{\left|i\right|}=\bar{1}$
form a basis for $V_{\bar{1}}$. Note that $e^{\lambda_{i}}u_{i}=0$
for $1\leq i\leq r+s$. 

With the above notation, Hoyt worked out a basis of $\mathfrak{\bar{g}}^{e}$
in \cite[Section 3.2.1]{Hoyt2012} which we recall below. For $\xi\in\mathfrak{\bar{g}}^{e}$,
we have $\xi(e^{j}u_{i})=e^{j}\xi(u_{i})$. Hence, each $\xi$ is
determined by $\xi(u_{i})$ for $i=1,\dots,r+s$ and we can write
\begin{equation}
\xi(u_{i})=\sum_{j=1}^{r+s}\sum_{k=\max\{\lambda_{j}-\lambda_{i},0\}}^{\lambda_{j}-1}c_{i}^{j,k}e^{k}u_{j}\label{eq:=00FF08i,j,k=00FF09}
\end{equation}
where $c_{i}^{j,k}\in\mathbb{C}$ are coefficients. Then $\mathfrak{\bar{g}}^{e}$
has a basis 
\begin{equation}
\{\xi_{i}^{j,k}:1\leq i,j\leq r+s\text{ and }\max\{\lambda_{j}-\lambda_{i},0\}\leq k\leq\lambda_{j}-1\}\label{eq:g^e(i,j,k)}
\end{equation}
 such that $\xi_{i}^{j,k}(u_{t})=\delta_{it}e^{k}u_{j}$.

Instead of using the formula for $\dim\mathfrak{\bar{g}}^{e}$ in
\cite[Section 3.2.1]{Hoyt2012}, we obtain an alternative formula
for $\dim\mathfrak{\bar{g}}^{e}$ below. Note that formulas in \cite[Section 3.2.1]{Hoyt2012}
and Proposition \ref{prop:dim gl(m,n)^e} are equivalent, but the
formula in Proposition \ref{prop:dim gl(m,n)^e} is more convenient
to use in Subsection \ref{subsec:Explanation-of-Theorems-A(m,n)}.
\end{singlespace}
\begin{prop}
\begin{singlespace}
\noindent \textup{\label{prop:dim gl(m,n)^e}}Let $\lambda$ be a
partition of $(m|n)$ denoted as in (\ref{eq:gl(m,n)-partition 1}).
Denote by $P$ the Dynkin pyramid of $\lambda$ and $e\in\mathfrak{g}_{\bar{0}}$
be a nilpotent element determined by $P$. Let $c_{i}$ be the number
of boxes in the $i$th column of $P$. Then \textup{$\dim\mathfrak{\bar{g}}^{e}=\sum_{i\in\mathbb{Z}}c_{i}^{2}+\sum_{i\in\mathbb{Z}}c_{i}c_{i+1}$.}
\end{singlespace}
\end{prop}

\begin{singlespace}
\noindent \begin{proof} Let $\mathfrak{\bar{g}}=\bigoplus_{j\in\mathbb{Z}}\mathfrak{\bar{g}}(j)$
as defined in Subsection \ref{subsec:Lablled-Dynkin-diagrams}. According
to \cite[Definition 4.1 and Theorem 7.2]{Hoyt2012}, we have that
the map $\text{ad}e:\mathfrak{\bar{g}}(j\geq-1)\rightarrow\mathfrak{\bar{g}}(j\geq1)$
is surjective and $\ker(\text{ad}e)=\mathfrak{\bar{g}}^{e}\subseteq\mathfrak{\bar{g}}(j\geq-1)$,
$\dim\ker(\text{ad}e)+\dim\text{im}(\text{ad}e)=\dim\mathfrak{\bar{g}}(j\geq-1)$,
so we have that $\dim\mathfrak{\bar{g}}^{e}=\dim\mathfrak{\bar{g}}(j\geq-1)-\dim\mathfrak{\bar{g}}(j\geq1)=\dim\mathfrak{\bar{g}}(0)+\dim\mathfrak{\bar{g}}(-1)$.
We also calculate that $[h,e_{kl}]=\left(\mathrm{col}(l)-\mathrm{col}(k)\right)e_{kl}$,
which means $e_{kl}\in\mathfrak{\bar{g}}(j)$ if $j=\mathrm{col}(l)-\mathrm{col}(k)$.
This implies that $\mathfrak{\bar{g}}(0)=\langle e_{kl}:\mathrm{col}(l)=\mathrm{col}(k)\rangle\cong\bigoplus_{i\in\mathbb{Z}}\mathrm{End}(\mathbb{C}^{c_{i}})$
and $\mathfrak{\bar{g}}(-1)=\langle e_{kl}:\mathrm{col}(l)-\mathrm{col}(k)=-1\rangle=\bigoplus\mathrm{Hom}(\mathbb{C}^{c_{i}},\mathbb{C}^{c_{i+1}})$.
Therefore, we obtain that $\dim\mathfrak{\bar{g}}^{e}=\sum_{i\in\mathbb{Z}}c_{i}^{2}+\sum_{i\in\mathbb{Z}}c_{i}c_{i+1}$.\end{proof}
\end{singlespace}

\begin{singlespace}
We also use an alternative notation for $\lambda$ such that 
\begin{equation}
\lambda=(c^{m_{c}+n_{c}},\dots,1^{m_{1}+n_{1}})\label{eq:gl(m,n)-partition 2}
\end{equation}
 where $m_{i}=\ensuremath{\left|\{j:\lambda_{j}=i,\ensuremath{\left|j\right|}=\bar{0}\}\right|}$
and $n_{i}=\ensuremath{\left|\{j:\lambda_{j}=i,\ensuremath{\left|j\right|}=\bar{1}\}\right|}$.
We define $M^{j}=\langle u_{i}:\lambda_{i}=j\rangle$ for $1\leq j\leq c$.
By \cite[Theorem 3.4]{Hoyt2012}, we have that $\mathfrak{\bar{g}}^{e}(0)\cong\mathfrak{gl}(M^{1})\oplus\mathfrak{gl}(M^{2})\oplus\dots\oplus\mathfrak{gl}(M^{c})$.
In addition, we have $M^{j}=M_{\bar{0}}^{j}\oplus M_{\bar{1}}^{j}$
where $\dim M_{\bar{0}}^{j}=m_{j}$, $\dim M_{\bar{1}}^{j}=n_{j}$
and thus $\mathfrak{gl}(M^{j})\cong\mathfrak{gl}(m_{j}|n_{j})$ for
each $j$. Therefore, we obtain that $\mathfrak{\bar{g}}^{e}(0)\cong\bigoplus_{j=1}^{c}\mathfrak{gl}(m_{j}|n_{j})$.

Next we move on to calculate $\dim\mathfrak{g}^{e}$. 
\end{singlespace}
\begin{thm}
\begin{singlespace}
\noindent \label{thm:I(M,0)}Let $e\in\mathfrak{g}_{\bar{0}}$ be
nilpotent, we have that $\dim\mathfrak{g}^{e}=\dim\mathfrak{\bar{g}}^{e}-1$.
\end{singlespace}
\end{thm}

\begin{proof}We first look at a Lie superalgebra homomorphism
$\text{str}:\mathfrak{\bar{g}}\rightarrow\mathbb{C}.$ Note that $\ker(\text{str})=\mathfrak{g}$
and $\dim\text{im}(\text{str})=1$. By restricting we get $\text{str}^{e}:\mathfrak{\bar{g}}^{e}\rightarrow\mathbb{C}$
and $\ker(\text{str}^{e})=\mathfrak{g}^{e}$. Then by the rank-nullity
theorem, we get $\dim\ker(\text{str}^{e})+\dim\text{im}(\text{str}^{e})=\dim\mathfrak{\bar{g}}^{e}$. 

\noindent Let
\[
e=\begin{pmatrix}e_{1} & 0\\
0 & e_{2}
\end{pmatrix}\in\mathfrak{g}_{\bar{0}}\subseteq\mathfrak{\bar{g}}_{\bar{0}}
\]
 where $e_{1}$ and $e_{2}$ are $m\times m$ and $n\times n$ matrices
respectively.

\noindent Let $I_{m}\in\mathfrak{gl}(m)$ be the $m\times m$ identity
matrix and $0_{n\times n}\in\mathfrak{gl}(n)$ be the $n\times n$
zero matrix. Let $I(m|0)$ be the $(m+n)\times(m+n)$ matrix 
\begin{equation}
I(m|0)=\begin{pmatrix}I_{m} & 0_{m\times n}\\
0_{n\times m} & 0_{n}
\end{pmatrix}.\label{eq:I(m,0)}
\end{equation}
Then we calculate $[I(m|0),e]=0$. Hence, we have $I(m|0)\in\mathfrak{\bar{g}}^{e}$.

\noindent Therefore, we have that $\mathrm{str}^{e}(I(m|0))=\mathrm{trace}(I_{m})=m$,
thus $\mathrm{str}^{e}$ is non-zero and $\dim\text{im}(\text{str}^{e})=1$.
Therefore, we have $\dim\mathfrak{g}^{e}=\dim\mathfrak{\bar{g}}^{e}-1$.\end{proof}

\subsection{Centre of centralizer of nilpotent elements $e\in\mathfrak{g}_{\bar{0}}$\label{subsec:cc-gl(m,n)}}

\begin{singlespace}
\noindent In order to give a basis for $\mathfrak{z}(\mathfrak{g}^{e})$,
we determine a basis for $\mathfrak{z}(\mathfrak{\bar{g}}^{e})$ first.
\end{singlespace}
\begin{prop}
\begin{singlespace}
\noindent \textup{\label{prop:I(m,0)}}Let $e\in\mathfrak{\bar{g}}_{\bar{0}}$
be nilpotent, we have that $\mathfrak{z}(\mathfrak{\bar{g}}^{e})\subseteq\mathfrak{\bar{g}}_{\bar{0}}$.
\end{singlespace}
\end{prop}

\begin{singlespace}
\noindent \begin{proof}Suppose $x\in\mathfrak{z}(\mathfrak{\bar{g}}^{e})$,
then $x\in\mathfrak{\bar{g}}^{e}$ and $[x,y]=0$ for all $y\in\mathfrak{\bar{g}}^{e}$.
Since $I(m|0)\in\mathfrak{\bar{g}}^{e}$ where $I(m\vert0)$ is defined
in (\ref{eq:I(m,0)}), we have that $[x,I(m\vert0)]=0$ . Therefore,
we have that $\mathfrak{z}(\mathfrak{\bar{g}}^{e})\subseteq(\mathfrak{\bar{g}}^{e})^{I(m|0)}\subseteq\mathfrak{\bar{g}}^{I(m|0)}$.

\noindent Let $x=\begin{pmatrix}A & B\\
C & D
\end{pmatrix}\in\mathfrak{\bar{g}}$. Notice that $[I(m|0),x]=0$ if and only if $B=C=0$. This implies
that $x\in\mathfrak{\bar{g}}_{\bar{0}}$. Hence, we have $\mathfrak{\bar{g}}^{I(m|0)}=\mathfrak{\bar{g}}_{\bar{0}}$.
Therefore, we deduce that $\mathfrak{z}(\mathfrak{\bar{g}}^{e})\subseteq\mathfrak{\bar{g}}_{\bar{0}}$.\end{proof}
\end{singlespace}

\begin{singlespace}
Yakimova shows that $\mathfrak{z}(\mathfrak{gl}(m+n)^{e})=\langle I,e,...,e^{l}\rangle$
where $l=\lambda_{1}-1$ in \cite[Theorem 2]{Yakimova2009}. Now we
use the above result to work out a basis of $\mathfrak{z}(\mathfrak{\bar{g}}^{e})$. 
\end{singlespace}
\begin{thm}
\begin{singlespace}
\noindent \textup{$\mathfrak{z}(\mathfrak{\bar{g}}^{e})$} has a basis
$\{I,e,...,e^{l}:l=\lambda_{1}-1\}$.
\end{singlespace}
\end{thm}

\begin{singlespace}
\noindent \begin{proof}Firstly, we denote the Lie algebra $\mathfrak{gl}(m+n)$
by $\mathfrak{g}'$. Note that $\mathfrak{gl}(m+n)$ is isomorphic
to $\mathfrak{\bar{g}}$ as a vector space. We also have that $\mathfrak{\bar{g}}^{e}=(\mathfrak{g}')^{e}$
as a vector space because $[e,x]=[e,x']$ for $x\in\mathfrak{\bar{g}}$
and $x'\in\mathfrak{g}'$. Denote by 
\[
\mathfrak{g}'_{0}=\left\{ \begin{pmatrix}A' & 0\\
0 & D'
\end{pmatrix}\in\mathfrak{g}':A'\in\mathfrak{gl}(m)\text{ and }D'\in\mathfrak{gl}(n)\right\} 
\]
and 
\[
\mathfrak{g}'_{1}=\left\{ \begin{pmatrix}0 & B'\\
C' & 0
\end{pmatrix}\in\mathfrak{g}':B'\text{is a }m\times n\text{ matrix and }C'\text{ is a }n\times m\text{ matrix}\right\} .
\]
 Then we let $\mathfrak{z}((\mathfrak{g}')^{e})_{0}=\mathfrak{z}((\mathfrak{g}')^{e})\cap\mathfrak{g}_{0}'$
and $\mathfrak{z}((\mathfrak{g}')^{e})_{1}=\mathfrak{z}((\mathfrak{g}')^{e})\cap\mathfrak{g}'_{1}$.
We further observe that $\mathfrak{z}(\mathfrak{\bar{g}}^{e})_{\bar{0}}=\mathfrak{z}((\mathfrak{g}')^{e})_{0}$
by definition. Since we already found that $\mathfrak{z}(\mathfrak{\bar{g}}^{e})\subseteq\mathfrak{\bar{g}}_{\bar{0}}$,
thus $\mathfrak{z}(\mathfrak{\bar{g}}^{e})_{\bar{1}}=0$. Using the
same argument as in Proposition \ref{prop:I(m,0)}, we have that $\mathfrak{z}((\mathfrak{g}')^{e})_{1}=0$.
Hence, $\mathfrak{z}(\mathfrak{\bar{g}}^{e})=\mathfrak{z}(\mathfrak{\bar{g}}^{e})_{\bar{0}}=\mathfrak{z}((\mathfrak{g}')^{e})_{0}=\mathfrak{z}((\mathfrak{g}')^{e})$.
Therefore, $\mathfrak{z}(\mathfrak{\bar{g}}^{e})=\langle I,e,...,e^{l}:l=\lambda_{1}-1\rangle$.\end{proof}
\end{singlespace}

\begin{singlespace}
Now we give a basis for $\mathfrak{z}(\mathfrak{g}^{e})$. We divide
our analysis into two cases: $m\neq n$ and $m=n>1$. 
\end{singlespace}
\begin{thm}
\begin{singlespace}
\noindent Let $\mathfrak{g}=\mathfrak{g}_{\bar{0}}\oplus\mathfrak{g}_{\bar{1}}=\mathfrak{sl}(m|n)$
and $\lambda$ be a partition of $(m|n)$ denoted as \textup{(\ref{eq:gl(m,n)-partition 1}).}
Let $e\in\mathfrak{g}_{\bar{0}}$ be a nilpotent element determined
by the Dynkin pyramid of $\lambda$, then $\mathfrak{z}(\mathfrak{g}^{e})=\langle e,...,e^{l}:l=\lambda_{1}-1\rangle$
except for $m=n,n>1$, in which case $\mathfrak{z}(\mathfrak{g}^{e})=\langle I,e,...,e^{l}:l=\lambda_{1}-1\rangle$.
\end{singlespace}
\end{thm}

\begin{proof}When $m\neq n$, we know that $\mathfrak{\bar{g}}^{e}=\mathfrak{g}^{e}\oplus\mathbb{C}I$
and thus $\mathfrak{z}(\mathfrak{\bar{g}}^{e})\subseteq\mathfrak{z}(\mathfrak{g}^{e})\oplus\mathbb{C}I$.
Let $x\in\mathfrak{z}(\mathfrak{g}^{e})$, then $[x,y]=0$ for all
$y\in\mathfrak{g}^{e}$. We also know that $x\in\mathfrak{\bar{g}}^{e}$
since $\mathfrak{g}^{e}\subseteq\mathfrak{\bar{g}}^{e}$. Moreover,
we have that $[x,I]=0$ and $[x,y]=0$ for all $y\in\mathfrak{g}^{e}\oplus\mathbb{C}I$.
Thus we have that $\mathfrak{z}(\mathfrak{g}^{e})\subseteq\mathfrak{z}(\mathfrak{\bar{g}}^{e})$.
Hence, a basis of $\mathfrak{z}(\mathfrak{g}^{e})$ consists of all
basis vectors of $\mathfrak{z}(\mathfrak{\bar{g}}^{e})$ except the
identity matrix $I$. Therefore, we have that $\mathfrak{z}(\mathfrak{g}^{e})=\langle e,...,e^{l}:l=\lambda_{1}-1\rangle$
and $\dim\mathfrak{z}(\mathfrak{g}^{e})=\lambda_{1}-1$.

\noindent When $m=n,n>1$, for a partition $\lambda=(\lambda_{1},\dots,\lambda_{r+s})$,
we have elements $\lambda_{j}\xi_{i}^{i,0}-(-1)^{\bar{i}}\lambda_{i}\xi_{j}^{j,0}$
with $i\neq j$ lie in the Cartan subalgebra $\mathfrak{h}\subseteq\mathfrak{g}$.

\noindent Let $S=\langle I,e,...,e^{l}:l=\lambda_{1}-1\rangle$. Clearly
$S\subseteq\mathfrak{z}(\mathfrak{g}^{e})$. We know that $e^{k}=\sum_{i=1}^{r+s}\xi_{i}^{i,k}$
and $e^{k}\in\mathfrak{g}$ for all $0\leq k\leq\lambda_{1}-1$. Suppose
$x\in\mathfrak{z}(\mathfrak{g}^{e})$ is of the form
\[
x=\sum_{1\leq t,h\leq r+s,\max\{\lambda_{t}-\lambda_{h}\}\leq k\leq\lambda_{t}-1}c_{h}^{t,k}\xi_{h}^{t,k}
\]
 where $c_{h}^{t,k}\in\mathbb{C}$ are coefficients. If $r+s\geq3$,
then $x$ commutes with $\lambda_{j}\xi_{i}^{i,0}-(-1)^{\bar{i}}\lambda_{i}\xi_{j}^{j,0}$
for all $i,j$. By computing $[\lambda_{j}\xi_{i}^{i,0}-\left(-1\right)^{\bar{i}}\lambda_{i}\xi_{j}^{j,0},x]$
for $i\neq j$ we have
\[
[\lambda_{j}\xi_{i}^{i,0}-(-1)^{\bar{i}}\lambda_{i}\xi_{j}^{j,0},x]=\lambda_{j}(\sum_{h,k}c_{h}^{i,k}\xi_{h}^{i,k}-\sum_{t,k}c_{i}^{t,k}\xi_{i}^{t,k})+(-1)^{\bar{i}}\lambda_{i}(\sum_{t,k}c_{j}^{t,k}\xi_{j}^{t,k}-\sum_{h,k}c_{h}^{j,k}\xi_{h}^{j,k}).
\]
This is equal to $0$, which forces $\sum_{h,k}c_{h}^{i,k}\xi_{h}^{i,k}=\sum_{t,k}c_{i}^{t,k}\xi_{i}^{t,k}$
and $\sum_{h,k}c_{h}^{j,k}\xi_{h}^{j,k}=\sum_{t,k}c_{j}^{t,k}\xi_{j}^{t,k}$.
This implies that $c_{h}^{t,k}=0$ for all $h\neq t$ and thus $x\in\langle\xi_{t}^{t,k}\rangle$.

\noindent If $r+s=2$, we have $\lambda_{1}=\lambda_{2}=n>1$. In
this case we only deal with $\xi_{h}^{t,k}$ for $h,t\in\{1,2\}$.
Note that a basis of $\mathfrak{g}^{e}$ is $\{\xi_{1}^{1,0}+\xi_{2}^{2,0},\xi_{1}^{1,j},\xi_{2}^{2,j},\xi_{1}^{2,k},\xi_{2}^{1,k}:j=1,\dots,n-1,k=0,1,\dots,n-1\}$.
Hence, an element $y\in\mathfrak{g}^{e}$ is of the form $y=\sum_{1\leq t,h\leq2,0\leq k\leq n-1}c_{h}^{t,k}\xi_{h}^{t,k}+c(\xi_{1}^{1,0}+\xi_{2}^{2,0})$.
Now suppose $y\in\mathfrak{z}(\mathfrak{g}^{e})$, by computing $[\xi_{1}^{1,1},y]=\sum_{k=0}^{n-1}c_{2}^{1,k}\xi_{2}^{1,k+1}\pm\sum_{k=0}^{n-1}c_{1}^{2,k}\xi_{1}^{2,k+1}$,
we obtain that $c_{1}^{2,k}=0$ and $c_{2}^{1,k}=0$ for all $k=0,\dots,n-2$.
Then we calculate $[\xi_{2}^{1,0},y]=c_{1}^{2,n-1}(\xi_{1}^{1,n-1}\pm\xi_{2}^{2,n-1})$,
which implies that $c_{1}^{2,n-1}=0$. Similarly we have that $c_{2}^{1,n-1}=0$.
Therefore, we obtain that $x\in\langle\xi_{1}^{1,0}+\xi_{2}^{2,0},\xi_{t}^{t,k}:k>0\rangle$.

\noindent From above we have that $x=\sum_{t,k}c_{t}^{t,k}\xi_{t}^{t,k}$.
Adding an element of $S$ we may assume that $c_{1}^{1,k}(x)=0$ for
all $k$. Suppose $x\notin S$, then there exist some $c_{i}^{i,k}\neq0$.
Next considering $\xi_{1}^{i,0}\in\mathfrak{g}^{e}$, we have $[x,\xi_{1}^{i,0}]=\sum_{t,k}c_{t}^{t,k}(x)[\xi_{t}^{t,k},\xi_{1}^{i,0}]=\sum_{k}c_{i}^{i,k}(x)\xi_{1}^{i,k}\neq0$,
thus $x\notin\mathfrak{z}(\mathfrak{g}^{e})$. Hence $\mathfrak{z}(\mathfrak{g}^{e})\subseteq S$.
Therefore, we deduce that $\mathfrak{z}(\mathfrak{g}^{e})=S$ as required.
This implies that $\dim\mathfrak{z}(\mathfrak{g}^{e})=\lambda_{1}$.
\end{proof}

\subsection{Proof of theorems\label{subsec:Explanation-of-Theorems-A(m,n)}}

\begin{singlespace}
\noindent Let $r_{i}$ (resp. $s_{i}$) be the number of boxes with
parity $\bar{0}$ (resp. $\bar{1}$) in the $i$th column of the Dynkin
pyramid $P$ and denote $c_{i}=r_{i}+s_{i}$. 
\end{singlespace}

\begin{singlespace}
In order to prove Theorem 1 for $\mathfrak{g}=\mathfrak{sl}(m|n)$,
we first look at the case that the corresponding $\varDelta$ only
has even labels. Note that labels in labelled Dynkin diagram $\varDelta$
are the horizontal difference between consecutive boxes in the pyramid.
Hence, there is no label equal to $1$ in $\varDelta$ if and only
if $\lambda_{i}-\lambda_{i+1}$ are even for all $i=1,\dots,r+s$,
i.e. $\lambda_{i}$ are all even or all odd. Based on the way that
the labelled Dynkin diagram is constructed, we have that $n_{2}(\varDelta)=\lambda_{1}-1$.
Therefore, we have that $\dim\text{\ensuremath{\mathfrak{z}}}(\mathfrak{g}^{e})=n_{2}(\varDelta)$
for $m\neq n$ and $\dim\text{\ensuremath{\mathfrak{z}}}(\mathfrak{g}^{e})=n_{2}(\varDelta)+1$
for $m=n>1$.

Next we turn to look for $\text{\ensuremath{\mathfrak{z}}}(\mathfrak{g}^{h})$.
Note that an element in $\mathfrak{g}^{h}$ is of the form
\[
\begin{pmatrix}x_{-\lambda_{1}+1} & \cdots & 0\\
\vdots & \ddots & \vdots\\
0 & \cdots & x_{\lambda_{1}-1}
\end{pmatrix}
\]
 where $x_{i}\in\mathfrak{gl}(r_{i}|s_{i})$ are block matrices for
$i=-\lambda_{1}+1,-\lambda_{1}+3\dots,\lambda_{1}-3,\lambda_{1}-1$
such that $\sum_{i}\text{str}(x_{i})=0$. Thus we have that an element
in $\text{\ensuremath{\mathfrak{z}}}(\mathfrak{g}^{h})$ is of the
form 
\[
\begin{pmatrix}d_{-\lambda_{1}+1}I(r_{-\lambda_{1}+1}|s_{-\lambda_{1}+1}) & \cdots & 0\\
\vdots & \ddots & \vdots\\
0 & \cdots & d_{\lambda_{1}-1}I(r_{\lambda_{1}-1}|s_{\lambda_{1}-1})
\end{pmatrix}
\]
 for some $d_{i}\in\mathbb{Z}$ such that $\sum d_{i}(r_{i}-s_{i})=0$.
Hence, we deduce that $\dim\text{\ensuremath{\mathfrak{z}}}(\mathfrak{g}^{h})=\lambda_{1}-1=n_{2}(\varDelta)$
for $m\neq n$ and $\dim\text{\ensuremath{\mathfrak{z}}}(\mathfrak{g}^{h})=n_{2}(\varDelta)+1$
for $m=n>1$.

Next we prove Theorem 2 for Lie superalgebras $\mathfrak{sl}(m|n)$.
Based on the way that the labelled Dynkin diagram is constructed,
any labelled Dynkin diagram for $e\in\mathfrak{g}_{\bar{0}}$ is the
same as the labelled Dynkin diagram for $e\in\mathfrak{gl}(m+n)$
except some of the vertices are $\otimes$, i.e. given a nilpotent
element $e\in\mathfrak{g}_{\bar{0}}$, all the labels $a_{i}$ in
the labelled Dynkin diagram with respect to $\mathfrak{sl}(m|n)$
are the same as that in the labelled Dynkin diagram with respect to
$\mathfrak{gl}(m+n)$ so that $\sum a_{i}$ is also the same. We also
have $a_{i}=\text{col}(i+1)-\text{col}(i)$, thus $\sum a_{i}=\sum_{i=1}^{m+n-1}\left(\text{col}(i+1)-\text{col}(i)\right)=\text{col}(m+n)-\text{col}(1)=2\lambda_{1}-2=2\dim\text{\ensuremath{\mathfrak{z}}}(\mathfrak{g}^{e})-2$.
Therefore, we have that $\dim\text{\ensuremath{\mathfrak{z}}}(\mathfrak{g}^{e})=\left\lceil \frac{1}{2}\sum a_{i}\right\rceil +\varepsilon$
where $\varepsilon=0$ for $m\neq n$ and $\varepsilon=1$ for $m=n>1$.

To prove Theorem 3 for $\mathfrak{sl}(m|n)$, we define $\mathfrak{g}_{0}$
to be the subalgebra generated by the root spaces $\mathfrak{g}_{-\alpha}$
and $\mathfrak{g}_{\alpha}$ for $\alpha$ a simple root with label
$0$ or $1$ in $\varDelta$. We consider two general cases: the labelled
Dynkin diagram $\varDelta$ has no label equal to $1$ and $\varDelta$
has some labels equal to $1$.

When $\varDelta$ has no label equal to $1$. Note that $e_{0}=0$
since $\varDelta_{0}$ has all labels equal to $0$. We also have
$\mathfrak{g}_{0}^{e_{0}}=\mathfrak{g}_{0}=\bigoplus_{i\in\mathbb{Z}}\mathfrak{sl}(r_{i}|s_{i})$.
Then $\dim\mathfrak{g}_{0}^{e_{0}}=\dim\mathfrak{g}_{0}=\sum_{i\in\mathbb{Z}}\dim\mathfrak{sl}(r_{i}|s_{i})=\sum_{i\in\mathbb{Z}:c_{i}>0}(c_{i}^{2}-1)$.
Hence, $\dim\mathfrak{g}^{e}-\dim\mathfrak{g}_{0}^{e_{0}}=(\sum_{i\in\mathbb{Z}:c_{i}>0}c_{i}^{2}-1)-\sum_{i\in\mathbb{Z}:c_{i}>0}(c_{i}^{2}-1)=-1+\sum_{i\in\mathbb{Z}:c_{i}>0}1=\lambda_{1}-1=n_{2}(\varDelta)$
since there are in total $\lambda_{1}$ columns in $P$ with non-zero
boxes. Moreover, we have that $\dim\mathfrak{z}(\mathfrak{g}_{0}^{e_{0}})=\tau$
where $\tau$ is the number of $i$ such that $r_{i}=s_{i}$. Therefore,
we obtain that $\dim\mathfrak{z}(\mathfrak{g}^{e})-\dim\mathfrak{z}(\mathfrak{g}_{0}^{e_{0}})=n_{2}(\varDelta)-\tau$
for $m\neq n$ and $\dim\mathfrak{z}(\mathfrak{g}^{e})-\dim\mathfrak{z}(\mathfrak{g}_{0}^{e_{0}})=n_{2}(\varDelta)+1-\tau$
for $m=n>1$.

When $\varDelta$ has some labels equal to $1$. There are in total
$2\lambda_{1}+1$ columns with labels from $-\lambda_{1}$ to $\lambda_{1}$
in the Dynkin pyramid $P$. Let $k>0$ be minimal such that $c_{k}=0$
and thus we know that $n_{2}(\varDelta)=\lambda_{1}-k$. Then we have
that 
\begin{align*}
\mathfrak{g}_{0}\cong\mathfrak{sl}(r_{-\lambda_{1}+1} & |s_{-\lambda_{1}+1})\oplus\dots\oplus\mathfrak{sl}(r_{-k-1}|s_{-k-1})\oplus\mathfrak{sl}\left(\sum_{i=-k+1}^{k-1}r_{i}\vert\sum_{i=-k+1}^{k-1}s_{i}\right)\\
 & \oplus\mathfrak{sl}(r_{k+1}|s_{k+1})\dots\oplus\mathfrak{sl}(r_{\lambda_{1}-1}|s_{\lambda_{1}-1}).
\end{align*}
Note that the projection of $e_{0}$ in each $\mathfrak{sl}(r_{i}|s_{i})$
is equal to $0$ for $\ensuremath{i}>k$ and $i<-k$, thus $e_{0}\in\mathfrak{sl}\left(\sum_{i=-k+1}^{k-1}r_{i}\vert\sum_{i=-k+1}^{k-1}s_{i}\right)$.
We know that 
\[
\dim\mathfrak{sl}\left(\sum_{i=-k+1}^{k-1}r_{i}\vert\sum_{i=-k+1}^{k-1}s_{i}\right)^{^{e_{0}}}=\sum_{i=-k+1}^{k-1}c_{i}^{2}+\sum_{i=-k+1}^{k-1}c_{i}c_{i+1}-1
\]
and $\dim\mathfrak{sl}(r_{i}|s_{i})=c_{i}^{2}-1$. We also know that
$P$ is symmetric, thus 
\[
\dim\mathfrak{g}_{0}^{e_{0}}=\sum_{i=-k+1}^{k-1}c_{i}^{2}+\sum_{i=-k+1}^{k-1}c_{i}c_{i+1}-1+2(c_{k+1}^{2}-1)+\dots+2(c_{\lambda_{1}-1}^{2}-1).
\]
Therefore, we have $\dim\mathfrak{g}^{e}-\dim\mathfrak{g}_{0}^{e_{0}}=\lambda_{1}-k=n_{2}(\varDelta).$

Observe that when $r_{i}\neq s_{i}$ for all $i>k$ and $i<-k$ and
$\sum_{i=-k+1}^{k-1}r_{i}\neq\sum_{i=-k+1}^{k-1}s_{i}$, then $\dim\mathfrak{z}(\mathfrak{g}_{0}^{e_{0}})=k-1$
as $\mathfrak{z}(\mathfrak{sl}(r_{i}|s_{i}))=0$. However, when there
exist some $i$ for $i>k$ or $i<-k$ such that $r_{i}=s_{i}$, then
$\dim\mathfrak{z}(\mathfrak{sl}(r_{i}|r_{i}))=1$. Moreover, $\dim\mathfrak{z}\left(\mathfrak{sl}\left(\sum_{i=-k+1}^{k-1}r_{i}\mid\sum_{i=-k+1}^{k-1}s_{i}\right)^{e_{0}}\right)=k$
if $\sum_{i=-k+1}^{k-1}r_{i}=\sum_{i=-k+1}^{k-1}s_{i}$ and $\dim\mathfrak{z}\left(\mathfrak{sl}\left(\sum_{i=-k+1}^{k-1}r_{i}\mid\sum_{i=-k+1}^{k-1}s_{i}\right)^{e_{0}}\right)=k-1$
if $\sum_{i=-k+1}^{k-1}r_{i}\neq\sum_{i=-k+1}^{k-1}s_{i}$ by Subsection
\ref{subsec:cc-gl(m,n)}. Let 
\[
\nu_{0}=\begin{cases}
0 & \text{if }\sum_{i=-k+1}^{k-1}r_{i}\neq\sum_{i=-k+1}^{k-1}s_{i};\\
1 & \text{if }\sum_{i=-k+1}^{k-1}r_{i}=\sum_{i=-k+1}^{k-1}s_{i}.
\end{cases}
\]
Then we have that $\dim\mathfrak{z}(\mathfrak{g}_{0}^{e_{0}})=k-1+\tau+\nu_{0}$
where $\tau$ is the number of $i$ such that $i>k$ or $i<-k$ and
$r_{i}=s_{i}$. Therefore, we have $\dim\mathfrak{z}(\mathfrak{g}^{e})-\dim\mathfrak{z}(\mathfrak{g}_{0}^{e_{0}})=n_{2}(\varDelta)-\tau-\nu_{0}$
for $m\neq n$ and $\dim\mathfrak{z}(\mathfrak{g}^{e})-\dim\mathfrak{z}(\mathfrak{g}_{0}^{e_{0}})=n_{2}(\varDelta)+1-\tau-\nu_{0}$
for $m=n>1$.
\end{singlespace}
\begin{singlespace}

\section{The ortho-symplectic Lie superalgebras\label{sec:osp}}
\end{singlespace}

\subsection{Matrix expression of ortho-sympletic Lie superalgebras\label{subsec:Construction-of-osp}}

\begin{singlespace}
\noindent Suppose $V=V_{\bar{0}}\oplus V_{\bar{1}}$ is a $\mathbb{Z}_{2}$-graded
vector space over $\mathbb{C}$. Let $B:V\times V\rightarrow V$ be
a non-degenerate even supersymmetric bilinear form on $V$, i.e. $B(V_{\bar{i}},V_{\bar{j}})=0$
unless $\bar{i}+\bar{j}=\bar{0}$, the restriction of $B$ to $V_{\bar{0}}$
is symmetric and the restriction of $B$ to $V_{\bar{1}}$ is skew-symmetric.
Recall that the ortho-symplectic Lie superalgebra $\mathfrak{osp}(V)$
is defined to be $\mathfrak{osp}(V)=\mathfrak{osp}(V)_{\bar{0}}\oplus\mathfrak{osp}(V)_{\bar{1}}$
where 
\begin{equation}
\mathfrak{osp}(V)_{\bar{i}}:=\{x\in\mathfrak{gl}(V)_{\bar{i}}:B(x(v),w)=-(-1)^{\bar{i}\bar{v}}B(v,x(w))\text{ for homogeneous }v,w\in V\}\label{eq:osp(V)}
\end{equation}
for $\bar{i}\in\{\bar{0},\bar{1}\}$ and $\bar{v}$ is the parity
of $v$. We write $\mathfrak{osp}(m|2n)$ for $\mathfrak{osp}(V)$
when $\dim V_{\bar{0}}=m$ and $\dim V_{\bar{1}}=2n$. Note that the
even part $\mathfrak{g}_{\bar{0}}=\mathfrak{o}(m)\oplus\mathfrak{sp}(2n)$.
\end{singlespace}

\begin{singlespace}
We next explain how to represent $\mathfrak{osp}(m|2n)$ using matrices
with respect to certain choices of basis of $V$. Let $l=\left\lfloor \frac{m}{2}\right\rfloor $.
Given any sequence $\eta=(\eta_{1},\dots,\eta_{l+n})\in\{\bar{0},\bar{1}\}$
such that $\ensuremath{\left|\{i:\eta_{i}=\bar{0}\}\right|}=l$ and
$\ensuremath{\left|\{i:\eta_{i}=\bar{1}\}\right|}=n$. Then when $m$
is odd, we define the standard basis $\mathfrak{B}$ of $V$ with
respect to $\eta$ to be $\mathfrak{B}=\{v_{1}^{\eta_{1}},\dots,v_{l+n}^{\eta_{l+n}},v_{0}^{\bar{0}},v_{-(l+n)}^{\eta_{l+n}},\dots,v_{-1}^{\eta_{1}}\}$,
where $v_{i}^{\eta_{i}},v_{-i}^{\eta_{i}}\in V_{\bar{0}}$ if $\eta_{i}=\bar{0}$
and $v_{i}^{\eta_{i}},v_{-i}^{\eta_{i}}\in V_{\bar{1}}$ if $\eta_{i}=\bar{1}$
for each $i$. When $m$ is even, $\mathfrak{B}=\{v_{1}^{\eta_{1}},\dots,v_{l+n}^{\eta_{l+n}},v_{-(l+n)}^{\eta_{l+n}},\dots,v_{-1}^{\eta_{1}}\}$
is the standard basis of $V$. With the above basis, the non-degenerate
even supersymmetric bilinear form $B$ on $V$ is given by
\end{singlespace}

\begin{equation}
B(v_{i}^{\eta_{i}},v_{j}^{\eta_{j}})=\begin{cases}
0 & \text{\text{if}\ }i\neq-j;\\
1 & \text{\text{if}\ }i=-j,\ \eta_{i}=\bar{0}\text{ or }\eta_{i}=\bar{1}\text{ and }i>0\\
-1 & \text{\text{if}\ }i=-j,\ \eta_{i}=\bar{1},\ i<0.
\end{cases}\label{eq:bilinear B}
\end{equation}
Thus the matrix of $B$ with respect to basis $\mathfrak{B}$ is 
\[
\begin{pmatrix} &  &  &  &  & 1\\
 &  &  &  & \diagup\\
 &  &  & 1\\
 &  & \left(-1\right)^{\eta_{1}}\\
 & \diagup\\
\left(-1\right)^{\eta_{k+n}}
\end{pmatrix}
\]

Next we explain a basis of $\mathfrak{g}$ corresponding to the above
basis $\mathfrak{B}$ of $V$. Define $E_{j,k}$ to be the linear
transformation sending $v_{k}^{\eta_{k}}$ to $v_{j}^{\eta_{j}}$.
Note that $\mathfrak{osp}(V)_{\bar{0}}$ is spanned by elements of
the form $E_{j,-j}$ for $\eta_{j}=\bar{1}$ and $E_{j,k}+\gamma_{-k,-j}E_{-k,-j}$
for $\eta_{j}=\eta_{k}$ and $j\neq-k$ and $\mathfrak{osp}(V)_{\bar{1}}$
is spanned by elements of the form $E_{j,k}+\gamma_{-k,-j}E_{-k,-j}$
for $\eta_{j}\neq\eta_{k}$ where $\gamma_{-k,-j}=\pm1$ as specified
below. We use equation (\ref{eq:osp(V)}) to determine $\gamma_{-k,-j}$.
For $i\in\{\bar{0},\bar{1}\},j,k\neq0$, we have 
\begin{equation}
B((E_{j,k}+\gamma_{-k,-j}E_{-k,-j})v_{k}^{\eta_{k}},v_{-j}^{\eta_{j}})=-(-1)^{i\eta_{k}}B(v_{k}^{\eta_{k}},(E_{j,k}+\gamma_{-k,-j}E_{-k,-j})v_{-j}^{\eta_{j}}).\label{eq:B()}
\end{equation}
according to (\ref{eq:osp(V)}). Then by (\ref{eq:bilinear B}) we
have that 
\[
\text{LHS of \eqref{eq:B()}}=B(v_{j}^{\eta_{j}},v_{-j}^{\eta_{j}})=\begin{cases}
1 & \text{\text{if}\ }\eta_{j}=\bar{0}\text{ or }\eta_{j}=\bar{1},\ j>0;\\
-1 & \text{\text{if}\ }\eta_{j}=\bar{1},\ j<0.
\end{cases}
\]
 and 
\begin{align*}
\text{RHS of \eqref{eq:B()}} & =-(-1)^{i\eta_{k}}\gamma_{-k,-j}B(v_{k}^{\eta_{k}},v_{-k}^{\eta_{k}})\\
 & =\begin{cases}
-(-1)^{i\eta_{k}}\gamma_{-k,-j} & \text{\text{if}\ }\eta_{k}=\bar{0}\text{ or }\eta_{k}=\bar{1},\ k>0;\\
(-1)^{i\eta_{k}}\gamma_{-k,-j} & \text{\text{if}\ }\eta_{k}=\bar{1},\ k<0.
\end{cases}
\end{align*}
Hence, for $E_{j,k}+\gamma_{-k,-j}E_{-k,-j}\in\mathfrak{osp}(V)_{\bar{0}}$,
we have that 
\[
\gamma_{-k,-j}=\begin{cases}
1 & \text{if }\eta_{j}=\eta_{k}=\bar{1}\text{ and }jk<0,\\
-1 & \text{if }\eta_{j}=\eta_{k}=\bar{0}\text{ or }\eta_{j}=\eta_{k}=\bar{1}\text{ and }jk>0.
\end{cases}
\]
 For $E_{j,k}+\gamma_{-k,-j}E_{-k,-j}\in\mathfrak{osp}(V)_{\bar{1}}$,
we have that $\eta_{j}\neq\eta_{k}$. Let $\text{sign}(j)=1$ for
$j>0$ and $\text{sign}(j)=-1$ for $j<0$. We deduce that 
\begin{equation}
\gamma_{-k,-j}=-\eta_{j}\text{sign}(j)+\eta_{k}\text{sign}(k)\text{ for all }\eta_{j}\neq\eta_{k}.\label{eq:sign}
\end{equation}
 Note that $\eta_{j},\eta_{k}$ in equation (\ref{eq:sign}) are viewed
as $0,1\in\mathbb{Z}$.

\begin{singlespace}
We further calculate signs in basis elements $e_{0,k}+\gamma_{-k,0}e_{-k,0}$
for $k>0$ and $e_{k,0}+\gamma_{0,-k}e_{0,-k}$ for $k>0$ if they
exist. Applying a similar argument we get that for $k>0$,
\begin{equation}
\gamma_{-k,0}=\begin{cases}
-1 & \text{\text{if}\ }\eta_{k}=\bar{0},\\
1 & \text{\text{if}\ }\eta_{k}=\bar{1}.
\end{cases}\text{ and }\gamma_{0,-k}=-1.\label{eq:sign in osp -3}
\end{equation}

Therefore, for the matrix expression of $\mathfrak{osp}(m|2n)$, we
fix all entries above the skew diagonal to have positive sign and
then use the above rules to determine signs of entries below the skew
diagonal. More precisely, a basis of $\mathfrak{osp}(m|2n)$ is 
\begin{align}
\{e_{i,-i},e_{j,k}+\gamma_{-k,-j}e_{-k,-j} & :\eta_{i}=\bar{1},0<j,k\leq l+n,\text{ or }j=0,k>0,\label{eq:basis for osp}\\
 & \text{ or }j>0,k=0\text{ or }jk<0,j+k<0\}\nonumber 
\end{align}
 where $\gamma_{-k,-j}$ is determined as above.
\end{singlespace}
\begin{example}
\begin{singlespace}
For $\mathfrak{osp}(3|2)$, take the basis to be $\{v_{1}^{\bar{0}},v_{2}^{\bar{1}},v_{0}^{\bar{0}},v_{-2}^{\bar{1}},v_{-1}^{\bar{0}}\}$,
then we have
\[
\mathfrak{osp}(3|2)=\left\{ \begin{pmatrix}a & b & c & d & 0\\
e & f & g & h & -d\\
k & l & 0 & -g & -c\\
r & s & l & -f & b\\
0 & r & -k & -e & -a
\end{pmatrix}:a,b,c,d,e,f,g,h,k,l,r,s\in\mathbb{C}\right\} .
\]
\end{singlespace}
\end{example}

\subsection{The ortho-symplectic Dynkin pyramid\label{subsec:Hoyt's-ortho-symplectic-Dynkin}}

\begin{singlespace}
\noindent Note that the nilpotent $G$-orbits in $\mathfrak{g}_{\bar{0}}$
are parametrized by the partitions of $(m|2n)$. Let $\lambda$ be
a partition of $(m|2n)$ such that 
\begin{equation}
\lambda=(p|q)=(p_{1},\dots,p_{r}|q_{1},\dots,q_{s})\label{eq:osp(m,2n)-partition 0}
\end{equation}
where $p$ (resp. $q$) is a partition of $m$ (resp. $2n$), $p_{1}\geq\dots\geq p_{r},q_{1}\geq\dots\geq q_{s}$
and all even parts of $p$ and all odd parts of $q$ have even multiplicity.
Write $\lambda=(c^{m_{c}+n_{c}},\dots,1^{m_{1}+n_{1}})$ such that
$m_{i}=\vert\{j:p_{j}=i\}\vert$ and $n_{i}=\vert\{j:q_{j}=i\}\vert$.
In this subsection, we recall the ortho-symplectic Dynkin pyramid
$P$ for $\lambda$ which is given in \cite[Section 8]{Hoyt2012}.
We use the ortho-symplectic Dynkin pyramid to give a nice representative
of the nilpotent orbit and determine the labelled Dynkin diagram with
respect to this orbit. Note that we use a numbering which is different
from that in \cite[Section 8]{Hoyt2012}.
\end{singlespace}

\begin{singlespace}
Recall that $P$ consists of $(m+2n)$ boxes with size $2\times2$
in the $xy$-plane and is centrally symmetric about $(0,0)$. Define
the row number (resp. column number) of a box to be the $y$-coordinate
(resp. $x$-coordinate) of the centre of the box. Below we describe
the rule to place boxes in the upper half plane and the rest of the
boxes are added to the lower half plane in a centrally symmetric way. 

Firstly, we set the zeroth row to be empty if $m$ is even. If $m$
is odd, there exist some odd parts appearing with odd multiplicity
and we let $a_{1}$ be the largest such part in $\lambda$ for all
$p_{i}$. Then we put $a_{1}$ boxes into the zeroth row in the columns
$1-a_{1},3-a_{1},\dots,a_{1}-1$. Next we remove one part of $a_{1}$
from $\lambda$.
\end{singlespace}

Then $p$ becomes a partition that contains an even number of odd
parts with odd multiplicity. Denote these representatives by $c_{1}>b_{1}>\dots>c_{N}>b_{N}$.
In the upper half plane, the rest of the boxes are added inductively
to the next row following the rules below. 

Suppose $a_{2}$ is the largest part remaining in $\lambda$. When
$m_{a_{2}}$ is odd, then $a_{2}=c_{k}$ for some $k\in\{1,\dots,N\}$.
We add an even skew row of length $\frac{c_{k}+b_{k}}{2}$ and put
boxes in this row in the columns $1-b_{k},3-b_{k},\dots,c_{k}-1$
with even parity $\bar{0}$. After that we remove $c_{k}$ and $b_{k}$
from the partition. Then add $\left\lfloor \frac{m_{a_{2}}}{2}\right\rfloor $
rows of length $a_{2}$ and boxes in these rows are placed in the
columns $1-a_{2},3-a_{2},\dots,a_{2}-1$ with parity $\bar{0}$. When
$m_{a_{2}}$ is even, we only add $\left\lfloor \frac{m_{a_{2}}}{2}\right\rfloor $
rows in columns $1-a_{2},3-a_{2},\dots,a_{2}-1$ with parity $\bar{0}$.
When $n_{a_{2}}$ is odd, then an odd skew row of length $\frac{a_{2}}{2}$
is added in the columns $1,\dots,a_{2}-1$ with boxes labelled by
parity $\bar{1}$. We draw a box with a cross through it to represent
each missing box in skew rows. Then we add $\left\lfloor \frac{n_{a_{2}}}{2}\right\rfloor $
rows of length $a_{2}$ and put boxes in the columns $1-a_{2},3-a_{2},\dots,a_{2}-1$
with parity $\bar{1}$. When $n_{a_{2}}$ is even, we only add $\left\lfloor \frac{n_{a_{2}}}{2}\right\rfloor $
rows in the columns $1-a_{2},3-a_{2},\dots,a_{2}-1$ with parity $\bar{1}$.
Then remove $a_{2}^{m_{a_{2}}+n_{a_{2}}}$ from $\lambda$.

\begin{singlespace}
Let $l=\left\lfloor \frac{m}{2}\right\rfloor $. In this paper, we
label boxes in $P$ down columns from left to right with numbers $1,\dots,l+n,-(l+n),\dots,-1$
such that boxes labelled by $i$ and $-i$ are central symmetrically.
For the case where $m$ is odd we have an additional central box which
is labelled by $0$. 
\end{singlespace}
\begin{example}
\begin{singlespace}
\noindent \label{exa:(5,3,1|3,3)=000026(3,3|4)}The ortho-symplectic
Dynkin pyramids for the partitions $(5,3,1|3,3)$ and $(3,3|4)$ are:
\end{singlespace}

\begin{doublespace}
\noindent \ \ \ \ \ \ \ \ \ \ \ \ \ \ \ \ \ \ \ \ \ \ \ \ \ \includegraphics[scale=0.9]{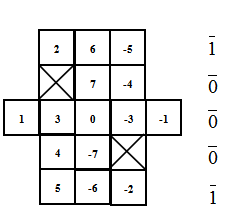}\ \ \ \ \ \ \ \includegraphics[scale=0.9]{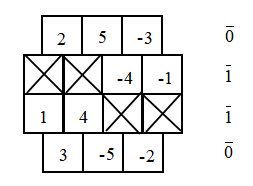}
\end{doublespace}
\end{example}

\begin{singlespace}
Let $\mathrm{row}(i)$ (resp. $\mathrm{col}(i)$) be the row number
(resp. column number) of the $i$th box. Let $\ensuremath{\left|\mathrm{row}(i)\right|}\in\{\bar{0},\bar{1}\}$
be the parity of $\mathrm{row}(i)$ for $i=\pm1,\dots,\pm(l+n)$ (resp.
$i=\pm1,\dots,\pm(l+n),0$) if $m$ is even (resp. $m$ is odd). Note
that the above numbering of the Dynkin pyramid determines a basis
$\mathfrak{B}$ of $V$ which is given in Subsection \ref{subsec:Construction-of-osp}
and $\eta_{i}=\ensuremath{\left|\mathrm{row}(i)\right|}$. 
\end{singlespace}

According to \cite[Section 8]{Hoyt2012}, the ortho-symplectic pyramid
$P$ determines a nilpotent element $e\in\mathfrak{g}_{\bar{0}}$
and $e=\sum\gamma_{i,j}E_{i,j}$ where the sum is over all $i$ and
$j$ such that 

\begin{singlespace}
1. $\mathrm{row}(i)=\mathrm{row}(j)$ and $\mathrm{col}(j)=\mathrm{col}(i)-2;$

2. $\mathrm{row}(i)=-\mathrm{row}(j),\left|\mathrm{row}(i)\right|=\bar{0}$
and the box labelled by $i$ is in an skew-row in the upper half plane,
$\mathrm{col}(i)=2$ and $\mathrm{col}(j)=0$ or $\mathrm{col}(i)=0$
and $\mathrm{col}(j)=-2$;

3. $\mathrm{row}(i)=-\mathrm{row}(j),\left|\mathrm{row}(i)\right|=\bar{1}$
and the box labelled by $i$ is in an skew-row in the upper half plane,
$\mathrm{col}(i)=1$ and $\mathrm{col}(j)=-1$.
\end{singlespace}

\begin{singlespace}
\noindent The pyramid $P$ also defines a semisimple element $h\in\mathfrak{g}_{\bar{0}}$
such that $h$ is the $(m+2n)$-diagonal matrix where the $i$th entry
is $-\mathrm{col}(i)$. Note that $\{e,h\}$ can be extended to an
$\mathfrak{sl}(2)$-triple $\{e,h,f\}$ in $\mathfrak{g}_{\bar{0}}$
according to \cite[Section 8]{Hoyt2012}.
\end{singlespace}

\subsection{Root system and and labelled Dynkin diagram for $\mathfrak{osp}(m|2n)$}

\begin{singlespace}
\noindent Let $\mathfrak{g}=\mathfrak{osp}(m|2n)=\mathfrak{g}_{\bar{0}}\oplus\mathfrak{g}_{\bar{1}}$
and $\mathfrak{h}$ be the set consisting of all diagonal matrices
in $\mathfrak{g}$. A basis of $\mathfrak{h}^{*}$ is given by \{$\varepsilon_{1}^{\eta_{1}}$,
$\dots$, $\varepsilon_{l+n}^{\eta_{l+n}}$\} where $\eta_{i}$ are
the parities as defined in Subsection \ref{subsec:Construction-of-osp}
and $(\varepsilon_{i}^{\eta_{i}},\varepsilon_{j}^{\eta_{j}})=(-1)^{\eta_{i}}\delta_{ij}$. 
\end{singlespace}

\begin{singlespace}
According to \cite[Section 2.3]{Musson2012}, for odd $m\geq1$, $n\geq1$,
the root system for $\mathfrak{g}$ is given by $\Phi=\Phi_{\bar{0}}\cup\Phi_{\bar{1}}$
such that 
\[
\Phi_{\bar{0}}=\{\pm\varepsilon_{i}^{\eta_{i}}\pm\varepsilon_{j}^{\eta_{j}}:\eta_{i}=\eta_{j}\}\cup\{\pm\varepsilon_{i}^{\bar{0}}\}\cup\{\pm2\varepsilon_{i}^{\bar{1}}\},
\]
\[
\Phi_{\bar{1}}=\{\pm\varepsilon_{i}^{\bar{1}}\}\cup\{\pm\varepsilon_{i}^{\eta_{i}}\pm\varepsilon_{j}^{\eta_{j}}:\eta_{i}\neq\eta_{j}\};
\]
and a choice of positive roots is $\Phi^{+}=\{\varepsilon_{i}^{\eta_{i}}\pm\varepsilon_{j}^{\eta_{j}},\varepsilon_{k}^{\eta_{k}},2\varepsilon_{t}^{\bar{1}}\}$. 
\end{singlespace}

For $m=2$, $n\geq1$, the root system for $\mathfrak{g}$ is given
by $\Phi=\Phi_{\bar{0}}\cup\Phi_{\bar{1}}$ such that 
\[
\Phi_{\bar{0}}=\{\pm\varepsilon_{i}^{\bar{1}}\pm\varepsilon_{j}^{\bar{1}}\}\cup\{\pm2\varepsilon_{i}^{\bar{1}}\},\Phi_{\bar{1}}=\{\pm\varepsilon_{1}^{\eta_{1}}\pm\varepsilon_{j}^{\eta_{j}}:\eta_{1}\neq\eta_{j}\};
\]
and a choice of positive roots is $\Phi^{+}=\{\varepsilon_{i}^{\eta_{i}}\pm\varepsilon_{j}^{\eta_{j}},2\varepsilon_{t}^{\bar{1}}\}$. 

For even $m>2$, $n\geq1$, the root system for $\mathfrak{g}$ is
given by $\Phi=\Phi_{\bar{0}}\cup\Phi_{\bar{1}}$ such that

\begin{singlespace}
\[
\Phi_{\bar{0}}=\{\pm\varepsilon_{i}^{\eta_{i}}\pm\varepsilon_{j}^{\eta_{j}}:\eta_{i}=\eta_{j}\}\cup\{\pm2\varepsilon_{i}^{\bar{1}}\},\Phi_{\bar{1}}=\{\pm\varepsilon_{i}^{\eta_{i}}\pm\varepsilon_{j}^{\eta_{j}}:\eta_{i}\neq\eta_{j}\};
\]
 and a choice of positive roots is $\Phi^{+}=\{\varepsilon_{i}^{\eta_{i}}\pm\varepsilon_{j}^{\eta_{j}},2\varepsilon_{t}^{\bar{1}}\}$.
\end{singlespace}

We have that odd roots $\pm\varepsilon_{i}^{\bar{1}}$ are non-isotropic
and all other odd roots are isotropic.

The labelled Dynkin diagram with respect to $e\in\mathfrak{g}_{\bar{0}}$
is constructed as follows: firstly, draw the ortho-symplectic Dynkin
pyramid $P$ of $\lambda$ following Subsection \ref{subsec:Hoyt's-ortho-symplectic-Dynkin}.
For boxes labelled by $i=1,\dots,l+n-1$, we associate a white node
$\fullmoon$ (resp. a grey node $\otimes$) to the root $\alpha_{i}$
if $\ensuremath{\left|\mathrm{row}(i+1)\right|}=\ensuremath{\left|\mathrm{row}(i)\right|}$
(resp. $\ensuremath{\left|\mathrm{row}(i+1)\right|}\neq\ensuremath{\left|\mathrm{row}(i)\right|}$)
and connect the $(i-1)$th and $i$th node with a single line. We
label the $i$th node with $a_{i}=\mathrm{col}(i+1)-\mathrm{col}(i)$.
For $i=l+n$, we need to consider different cases. 

\begin{singlespace}
When $m$ is odd, we associate a white node $\fullmoon$ (resp. a
black node $\newmoon$) to root $\alpha_{l+n}$ if $\ensuremath{\left|\mathrm{row}(l+n)\right|}=\ensuremath{\left|\mathrm{row}(0)\right|}$
(resp. $\ensuremath{\left|\mathrm{row}(l+n)\right|}\neq\ensuremath{\left|\mathrm{row}(0)\right|}$).
We connect the $(l+n-1)$th and $(l+n)$th node with $2$ lines and
put an arrow pointing from the $(l+n-1)$th node to the $(l+n)$th
node. The $(l+n)$th node is labelled by $a_{l+n}=\mathrm{col}(0)-\mathrm{col}(l+n)$. 

When $m=2$.
\end{singlespace}
\begin{itemize}
\item If $\vert\mathrm{row}(n+1)\vert=\bar{1}$, we associate a grey node
$\otimes$ to root $\alpha_{n+1}$. We connect the $n$th and $(n+1)$th
node with 2 lines and connect the $(n-1)$th and $(n+1)$th node with
a single line. The $(n+1)$th node is labelled by $a_{n+1}=-\mathrm{col}(n+1)-\mathrm{col}(n)$. 
\begin{singlespace}
\item If $\vert\mathrm{row}(n+1)\vert=\bar{0}$, we associate a white node
$\fullmoon$ to root $\alpha_{n+1}$. We connect the $n$th and $(n+1)$th
node with $2$ lines and put an arrow pointing from the $(n+1)$th
node to the $n$th node. The $(n+1)$th node is labelled by $a_{n+1}=-2\mathrm{col}(n+1)$. 
\end{singlespace}
\end{itemize}
When $m>2$ is even.
\begin{itemize}
\item If $\vert\mathrm{row}(l+n)\vert=\bar{0}$ and $\vert\mathrm{row}(l+n-1)\vert=\bar{1}$,
we associate a grey node $\otimes$ to root $\alpha_{l+n}$. We connect
the $(l+n)$th and the $(l+n-1)$th node with $2$ lines and connect
the $(l+n-2)$th and $(l+n)$th node with a single line. The $(l+n)$th
node is labelled by $-\mathrm{col}(l+n)-\mathrm{col}(l+n-1)$. 
\begin{singlespace}
\item If $\vert\mathrm{row}(l+n)\vert=\vert\mathrm{row}(l+n-1)\vert=\bar{0}$,
we associate a white node $\fullmoon$ to root $\alpha_{l+n}$ and
put a single line between the $(l+n)$th and the $(l+n-2)$th node.
The $(l+n)$th node is labelled by $a_{l+n}=-2\mathrm{col}(l+n)$. 
\item If $\vert\mathrm{row}(l+n)\vert=\bar{1}$, we associate a white node
$\fullmoon$ to root $\alpha_{l+n}$ and connect the $(l+n)$th and
the $(l+n-1)$th node with $2$ lines. An arrow is pointing from $(l+n)$th
node to the $(l+n-1)$th node. The $(l+n)$th node is labelled by
$-\mathrm{col}(l+n)-\mathrm{col}(l+n-1)$.
\end{singlespace}
\end{itemize}
We describe the corresponding simple roots and draw the labelled Dynkin
diagram for each case in Table \ref{tab:LDD-osp}. Note that the symbol
\includegraphics[width=0.5cm,height=0.5cm]{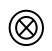}
in Table \ref{tab:LDD-osp} represents either a white or grey node
can appear.

\begin{singlespace}
\noindent %
\begin{longtable}[c]{|>{\raggedright}m{2cm}||>{\centering}m{5cm}||>{\centering}m{7cm}|}
\caption{\label{tab:LDD-osp}Labelled Dynkin diagrams for $\mathfrak{osp}(m|2n)$}
\tabularnewline
\endfirsthead
\hline 
 & \begin{singlespace}
\noindent sets of simple roots
\end{singlespace}
 & labelled Dynkin diagram\tabularnewline
\hline 
\hline 
\multirow{1}{2cm}{$m\geq1$ is odd} & \multirow{1}{5cm}{\begin{singlespace}
\noindent $\{\varepsilon_{1}^{\eta_{1}}-\varepsilon_{2}^{\eta_{2}},\dots,\varepsilon_{l+n-1}^{\eta_{l+n-1}}-\varepsilon_{l+n}^{\eta_{l+n}},\varepsilon_{l+n}^{\eta_{l+n}}\}$
\end{singlespace}
} & \includegraphics[viewport=0bp 0bp 268bp 54.2045bp,scale=0.6]{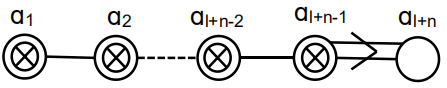}

\includegraphics[scale=0.6]{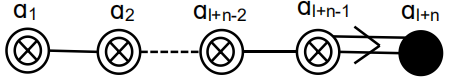}\vspace{0.5cm}
\tabularnewline
\hline 
\hline 
$m=2$ & \begin{singlespace}
\noindent $\{\varepsilon_{1}^{\eta_{1}}-\varepsilon_{2}^{\eta_{2}},\dots,\varepsilon_{n}^{\eta_{n}}-\varepsilon_{n+1}^{\eta_{n+1}},2\varepsilon_{n+1}^{\bar{1}}\}$
\end{singlespace}
 & \includegraphics[viewport=0bp 0bp 270.6bp 71.5152bp,scale=0.6]{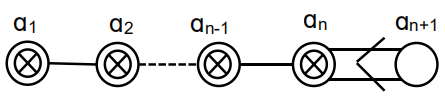}\vspace{0.5cm}
\tabularnewline
\hline 
\hline 
$m=2$ & \begin{singlespace}
$\{\varepsilon_{1}^{\eta_{1}}-\varepsilon_{2}^{\eta_{2}},\dots,\varepsilon_{n}^{\bar{1}}-\varepsilon_{n+1}^{\bar{0}},\varepsilon_{n}^{\bar{1}}+\varepsilon_{n+1}^{\bar{0}}\}$
\end{singlespace}
 & \includegraphics[viewport=0bp 0bp 267.6bp 142.8bp,scale=0.6]{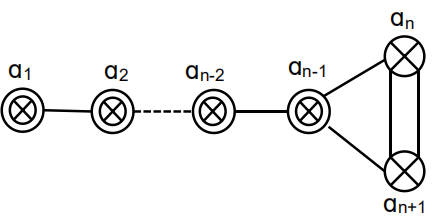}\vspace{0.5cm}
\tabularnewline
\hline 
\hline 
\multirow{3}{2cm}{$m>2$ is even} & \begin{singlespace}
\noindent $\{\varepsilon_{1}^{\eta_{1}}-\varepsilon_{2}^{\eta_{2}},\dots,\varepsilon_{l+n-1}^{\eta_{l+n-1}}-\varepsilon_{l+n}^{\bar{1}},2\varepsilon_{l+n}^{\bar{1}}\}$
\end{singlespace}
 & \includegraphics[viewport=0bp 0bp 271bp 53.1954bp,scale=0.6]{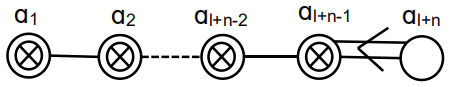}\vspace{0.5cm}
\tabularnewline
\cline{2-3} \cline{3-3} 
 & \begin{singlespace}
\noindent $\{\varepsilon_{1}^{\eta_{1}}-\varepsilon_{2}^{\eta_{2}},\dots,\varepsilon_{l+n-1}^{\bar{1}}-\varepsilon_{l+n}^{\bar{0}},\varepsilon_{l+n-1}^{\bar{1}}+\varepsilon_{l+n}^{\bar{0}}\}$
\end{singlespace}
 & \includegraphics[viewport=0bp 0bp 268bp 159.1453bp,scale=0.6]{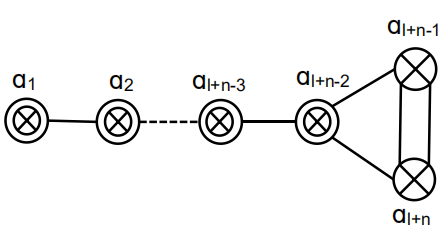}\vspace{0.5cm}
\tabularnewline
\cline{2-3} \cline{3-3} 
 & \begin{singlespace}
\noindent $\{\varepsilon_{1}^{\eta_{1}}-\varepsilon_{2}^{\eta_{2}},\dots,\varepsilon_{l+n-1}^{\bar{0}}-\varepsilon_{l+n}^{\bar{0}},\varepsilon_{l+n-1}^{\bar{0}}+\varepsilon_{l+n}^{\bar{0}}\}$
\end{singlespace}
 & \includegraphics[viewport=0bp 0bp 266bp 167.7419bp,scale=0.6]{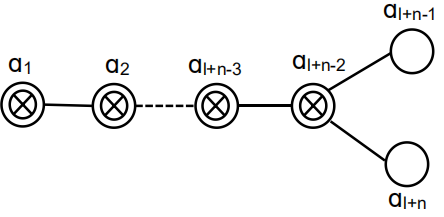}\vspace{0.5cm}
\tabularnewline
\hline 
\end{longtable}
\end{singlespace}
\begin{example}
\begin{singlespace}
\noindent For the partitions $(5,3,1|3,3)$ and $(3,3|4)$, the corresponding
Dynkin pyramids are shown in Example \ref{exa:(5,3,1|3,3)=000026(3,3|4)}.
Then the corresponding labelled Dynkin diagrams are:
\end{singlespace}

\noindent \ \ \ \ \includegraphics{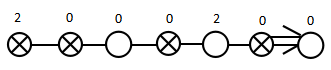} and
\includegraphics[scale=0.7]{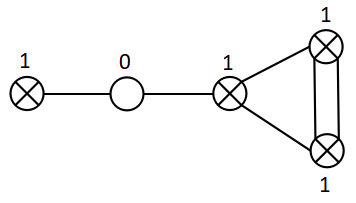} 

\noindent respectively.
\end{example}

\begin{rem}
\begin{singlespace}
\noindent Similar to Dynkin pyramid, different numberings within columns
for a ortho-symplectic Dynkin pyramid are possible and would lead
to different labelled Dynkin diagrams. In this way one can get all
labelled Dynkin diagram. The following example shows we get different
labelled Dynkin diagram if we allow different numberings within columns.
\end{singlespace}
\end{rem}

\begin{example}
For the partition $(5,3,1|3,3)$, if we follow the principle of numbering
within columns given in this subsection, the corresponding labelled
Dynkin diagram is given in Example 15. However, if we choose a different
numbering as shown below, the corresponding labelled Dynkin diagram
shown below is different from that is in Example 15.

\includegraphics[scale=0.8]{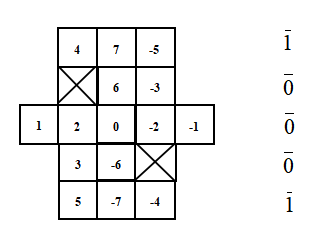}\includegraphics[scale=0.8]{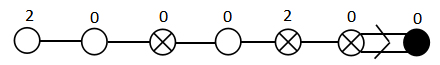}
\end{example}

\subsection{Alternative Dynkin pyramid for $\lambda$\label{subsec:Alternative-Dynkin-pyramid}}

\begin{singlespace}
\noindent In this subsection, we use an alternative notation for $\lambda$
and rewrite
\begin{equation}
\lambda=(\lambda_{1},\dots,\lambda_{a},\lambda_{a+1},\lambda_{-(a+1)},\dots,\lambda_{b},\lambda_{-b})\label{eq:jordan type group}
\end{equation}
 where $\lambda_{1},\dots,\lambda_{a}$ are the parts with odd multiplicity,
$\lambda_{1}>\lambda_{2}>\dots>\lambda_{a}$ and $\lambda_{a+1}=\lambda_{-(a+1)}\geq\dots\geq\lambda_{b}=\lambda_{-b}$.
We define $\ensuremath{\left|i\right|}\in\{\bar{0},\bar{1}\}$ such
that for $c\in\mathbb{Z}$, we have $\ensuremath{\left|\{i:\lambda_{i}=c,\ensuremath{\left|i\right|}=\bar{0}\}\right|}=\ensuremath{\left|\{j:p_{j}=c\}\right|}$
and $\left|\{i:\lambda_{i}=c,\ensuremath{\left|i\right|}=\bar{1}\}\right|=\ensuremath{\left|\{j:q_{j}=c\}\right|}$
for some $j$. Next we establish a Dynkin pyramid $\tilde{P}$ which
is different from that we used in the previous section. We use $\tilde{P}$
to determine a basis for $\mathfrak{g}^{e}$ in Subsection \ref{subsec:c-osp}.
\end{singlespace}

\begin{singlespace}
The new version of Dynkin pyramid $\tilde{P}$ consists of $(m+2n)$
boxes with size $2\times2$ in the $xy$-plane and is centred on $(0,0)$.
We label rows from $1$ to $b$ in the upper half plane which is different
from the way we used in Subsection \ref{subsec:Hoyt's-ortho-symplectic-Dynkin}.
For each $\lambda_{i}$ with $i>0$, we put $\lambda_{i}$ boxes both
into the $i$th row and $-i$th row in the columns $1-\lambda_{i},3-\lambda_{i},\dots,\lambda_{i}-1$.
We start with the parts $\lambda_{1},\dots,\lambda_{a}$. For $1\leq i\leq a$,
we cross out $\lfloor\frac{\lambda_{i}}{2}\rfloor$ boxes in the $i$th
row from left to right and cross out $\lceil\frac{\lambda_{i}}{2}\rceil$
boxes in the $-i$th row from right to left. If $\lambda_{i}$ is
odd (resp. even), we label boxes without cross in the $i$th row from
left to right with $i_{0},i_{2},\dots,i_{\lambda_{i}-1}$ (resp. $i_{1},i_{3},\dots,i_{\lambda_{i}-1}$)
and boxes without cross in the $-i$th row from left to right with
$i_{-(\lambda_{i}-1)},\dots,i_{-2}$ (resp. $i_{-(\lambda_{i}-1)},\dots,i_{-1}$).
Then we deal with the parts $\lambda_{a+1},\lambda_{-(a+1)},\dots,\lambda_{b},\lambda_{-b}$.
For $a+1\leq i\leq b$, we label boxes in the $i$th row with $i_{1-\lambda_{i}},i_{3-\lambda_{i}},\dots,i_{\lambda_{i}-1}$
and boxes in the $-i$th row are labelled by $-i_{1-\lambda_{i}},-i_{3-\lambda_{i}},\dots,-i_{\lambda_{i}-1}$.
Let $\ensuremath{\left|\mathrm{row}(i)\right|}$ be the parity of
the $i$th row such that $\ensuremath{\left|\mathrm{row}(i)\right|}=\ensuremath{\left|i\right|}$
and $\left|i\right|$ is defined in (\ref{eq:jordan type group}).

Note that from the above Dynkin pyramid $\tilde{P}$ we get a basis
\begin{equation}
\{v_{i_{j}}:i_{j}\text{ is a box in }\tilde{P}\}\label{eq:vij-alternative dynkin}
\end{equation}
 of $V$. More precisely, basis elements $\{v_{i_{j}}:1\leq i\leq a,0\leq j\leq\lambda_{i}-1\text{ for odd }\lambda_{i}\text{ and }1\leq j\leq\lambda_{i}\text{ for even }\lambda_{i}\}\cup\{v_{i_{j}}:a+1\leq i\leq b\}$
(resp. $\{v_{i_{j}}:1\leq i\leq a,1-\lambda_{i}\leq j<0\text{ for odd }\lambda_{i}\text{ and }-\lambda_{i}\leq j\leq-1\text{ for even }\lambda_{i}\}\cup\{v_{-i_{j}}:a+1\leq i\leq b\}$)
correspond to the boxes in the upper (lower) half of $\tilde{P}$.
The bilinear form $B(\ldotp,\ldotp)$ on $V$ is given by 
\[
B(v_{i_{j}},v_{k_{l}})=\begin{cases}
\pm1 & \text{if }i=\pm k,j=-l\\
0 & \text{otherwise,}
\end{cases}
\]
 and signs will be given explicitly later in (\ref{eq:c-osp-bilinear-form-1})
and (\ref{eq:c-osp-bilinear-form-2}). Note that $\tilde{P}$ also
gives a nilpotent element $e\in\mathfrak{g}_{\bar{0}}$ such that
$e$ sends $v_{i_{j}}$ to $v_{i_{j-2}}$ if there exists a box labelled
by $v_{i_{j-2}}$ and sends $v_{i_{j}}$ to zero if there no such
box.
\end{singlespace}
\begin{example}
\begin{singlespace}
\noindent The alternative Dynkin pyramid $\tilde{P}$ for the partition
$\lambda=(5,3|2^{2})$ is

\noindent 
\begin{figure}[H]
\begin{singlespace}
\noindent \begin{centering}
\includegraphics[scale=0.8]{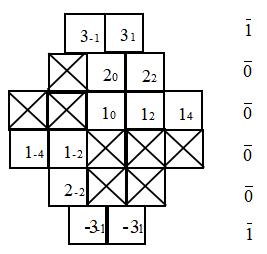}
\par\end{centering}
\end{singlespace}
\caption{Dynkin pyramid for the partition $\lambda=(5,3|2^{2})$}
\end{figure}
\end{singlespace}

\begin{singlespace}
Define $u_{i}=v_{i_{\lambda_{i}-1}}\in V$ and thus $e^{j}u_{i}=v_{i_{\lambda_{i}-2j-1}}$.
Then the vectors $e^{j}u_{i}$ with $\ensuremath{\left|i\right|}=\bar{0},\ 0\leq j\leq\lambda_{i}-1$
form a basis for $V_{\bar{0}}$ and $e^{j}u_{i}$ with $\ensuremath{\left|i\right|}=\bar{1},\ 0\leq j\leq\lambda_{i}-1$
form a basis for $V_{\bar{1}}$. Moreover, we have $e^{\lambda_{i}}u_{i}=0$
for all $i$ and they satisfy the following conditions:

(1) For $i=1,\dots,a$, we have
\begin{equation}
B(e^{k}u_{i},e^{h}u_{j})=\begin{cases}
(-1)^{k} & \text{if }i=j\text{ and }k+h=\lambda_{i}-1,\\
0, & \text{otherwise},
\end{cases}\label{eq:c-osp-bilinear-form-1}
\end{equation}
\end{singlespace}
\end{example}

\begin{singlespace}
(2) For $i=\pm(a+1),\dots,\pm b$, then there exists $\theta_{i}\in\{-1,1\}$
such that 
\begin{equation}
B(e^{k}u_{i},e^{h}u_{j})=\begin{cases}
(-1)^{k}\theta_{i} & \text{if }-i=j\text{ and }k+h=\lambda_{i}-1,\\
0, & \text{otherwise}.
\end{cases}\label{eq:c-osp-bilinear-form-2}
\end{equation}

\end{singlespace}

\subsection{Centralizer of nilpotent element $e\in\mathfrak{g}_{\bar{0}}$\label{subsec:c-osp}}

\begin{singlespace}
\noindent In this subsection, we give a basis for $\mathfrak{g}^{e}$
based on \cite[Chapter 1]{Yakimova2009} and then state an alternative
formula to the formula given in \cite[Subsection 3.2.2]{Hoyt2012}
for $\dim\mathfrak{g}^{e}$. In order to describe a basis for $\mathfrak{g}^{e}$,
we define $i^{*}=i$ for $i=1,\dots,a$ and $i^{*}=-i$ for $i=\pm(a+1),\dots,\pm b$.
\end{singlespace}

\begin{singlespace}
Recall that a basis for $\mathfrak{gl}(m|n)^{e}$ is known in terms
of $\xi_{i}^{j,k}$ such that $\xi_{i}^{j,k}$ sends $u_{i}$ to $e^{k}u_{j}$
and all other $u_{t}$ to $0$. We know that $\mathfrak{g}^{e}=\mathfrak{g}\cap\mathfrak{gl}(m|2n)^{e}$.
Therefore, the elements in a basis of $\mathfrak{g}_{\bar{0}}^{e}$
are of the form:
\end{singlespace}

\begin{singlespace}
\noindent 
\[
\xi_{i}^{i^{*},\lambda_{i}-1-k}\text{ for }0\leq k\leq\lambda_{i}-1,\ k\text{ is odd if }\ensuremath{\left|i\right|}=\bar{0}\text{ and }k\text{ is even if }\ensuremath{\left|i\right|}=\bar{1};
\]

\noindent 
\[
\xi_{i}^{j,\lambda_{j}-1-k}+\varepsilon_{i}^{j,\lambda_{j}-1-k}\xi_{j^{*}}^{i^{*},\lambda_{i}-1-k}\text{ for all }0\leq k\leq\min\{\lambda_{i},\lambda_{j}\}-1,\ensuremath{\left|i\right|}=\ensuremath{\left|j\right|}
\]
where $\varepsilon_{i}^{j,\lambda_{j}-1-k}\in\{\pm1\}$ can be determined
by using \cite[Section 3.2]{Jantzen2004}. More precisely, we have
that 
\[
\varepsilon_{i}^{j,\lambda_{j}-1-k}=(-1)^{\lambda_{j}-k}\theta_{j}\theta_{i}\text{ for all }0\leq k\leq\min\{\lambda_{i},\lambda_{j}\}-1,\ensuremath{\left|i\right|}=\ensuremath{\left|j\right|}.
\]
Elements in a basis of $\mathfrak{g}_{\bar{1}}^{e}$ are of the form
\[
\xi_{i}^{j,\lambda_{j}-1-k}\pm\xi_{j^{*}}^{i^{*},\lambda_{i}-1-k}
\]
 with appropriate choices of signs for $0\leq k\leq\min\{\lambda_{i},\lambda_{j}\}-1$,
$\ensuremath{\left|i\right|}\neq\ensuremath{\left|j\right|}$.
\end{singlespace}

\begin{singlespace}
Write $e=e_{\mathfrak{o}}+e_{\mathfrak{sp}}$ where $e_{\mathfrak{o}}\in\mathfrak{o}(m)$
and $e_{\mathfrak{sp}}\in\mathfrak{sp}(2n)$, we know that $\dim\mathfrak{g}_{\bar{0}}^{e}=\dim\mathfrak{o}(m)^{e_{\mathfrak{o}}}+\dim\mathfrak{sp}(2n)^{e_{\mathfrak{sp}}}$.
By \cite[Section 3.2]{Jantzen2004}, we have that $\dim\mathfrak{o}(m)^{e_{\mathfrak{o}}}=\frac{1}{2}\dim\mathfrak{gl}(m)^{e_{\mathfrak{o}}}-\frac{1}{2}\ensuremath{\left|\{i:\lambda_{i}\text{ is odd},\ensuremath{\left|i\right|}=\bar{0}\}\right|}$
and $\dim\mathfrak{sp}(2n)^{e_{\mathfrak{sp}}}=\frac{1}{2}\dim\mathfrak{gl}(2n)^{e_{\mathfrak{sp}}}+\frac{1}{2}\ensuremath{\left|\{i:\lambda_{i}\text{ is odd},\ensuremath{\left|i\right|}=\bar{1}\}\right|}$.
Hence, we obtain that 
\[
\dim\mathfrak{g}_{\bar{0}}^{e}=\frac{1}{2}\dim\mathfrak{gl}(m|2n)_{\bar{0}}^{e}-\frac{1}{2}\ensuremath{\left|\{i:\lambda_{i}\text{ is odd},\ensuremath{\left|i\right|}=\bar{0}\}\right|}+\frac{1}{2}\ensuremath{\left|\{i:\lambda_{i}\text{ is odd},\ensuremath{\left|i\right|}=\bar{1}\}\right|}.
\]
In addition, Hoyt argues that $\dim\mathfrak{g}_{\bar{1}}^{e}=\frac{1}{2}\dim\mathfrak{gl}(m|2n)_{\bar{1}}^{e}$
in \cite[Section 3.2.2]{Hoyt2012}. Therefore, we have that 
\begin{equation}
\dim\mathfrak{g}^{e}=\frac{1}{2}\dim\mathfrak{gl}(m|2n)^{e}-\frac{1}{2}\ensuremath{\left|\{i:\lambda_{i}\text{ is odd},\ensuremath{\left|i\right|}=\bar{0}\}\right|}+\frac{1}{2}\ensuremath{\left|\{i:\lambda_{i}\text{ is odd},\ensuremath{\left|i\right|}=\bar{1}\}\right|}.\label{eq:osp(m,2n)^e}
\end{equation}

\end{singlespace}

We obtain an alternative formula for $\dim\mathfrak{g}^{e}$ below.
\begin{prop}
Let $\mathfrak{g}=\mathfrak{g}_{\bar{0}}\oplus\mathfrak{g}_{\bar{1}}=\mathfrak{osp}(m|2n)$
and $\lambda$ be a partition of $(m|2n)$. Denote by $P$ the ortho-symplectic
Dynkin pyramid of $\lambda$. Then $P$ determines an $\mathfrak{sl}(2)$-triple
$\{e,h,f\}$ in $\mathfrak{g}_{\bar{0}}$. Let $c_{i}$ be the number
of boxes in the $i$th column of $P$ and $r_{i}$ (resp. $s_{i}$)
be the number of boxes with parity $\bar{0}$ (resp. $\bar{1}$) in
the $i$th column of $P$. We have that $\dim\mathfrak{g}^{e}=\frac{1}{2}(\sum c_{i}^{2}+\sum c_{i}c_{i+1})-\frac{r_{0}}{2}+\frac{s_{0}}{2}$.
\end{prop}

\noindent \begin{proof}Let $\mathfrak{g}=\bigoplus_{j\in\mathbb{Z}}\mathfrak{g}(j)$.
Using the same argument as in the proof of Proposition \ref{prop:dim gl(m,n)^e},
we get $\dim\mathfrak{g}^{e}=\dim\mathfrak{g}(0)+\dim\mathfrak{g}(-1)$.
For an element $E_{k,l}+\gamma_{l,k}E_{-l,-k}\in\mathfrak{g}$, we
calculate that $[h,E_{k,l}+\gamma_{l,k}E_{-l,-k}]=\left(\mathrm{col}(l)-\mathrm{col}(k)\right)(E_{k,l}+\gamma_{l,k}E_{-l,-k})$,
this implies that $E_{k,l}+\gamma_{l,k}E_{-l,-k}\in\mathfrak{g}(j)$
if $j=\mathrm{col}(l)-\mathrm{col}(k)$. Hence, we have that 
\[
\mathfrak{g}(0)=\langle E_{k,l}+\gamma_{l,k}E_{-l,-k}:\mathrm{col}(l)=\mathrm{col}(k)\rangle\cong(\bigoplus_{i<0}\mathfrak{gl}(r_{i}|s_{i}))\oplus\mathfrak{osp}(r_{0}|s_{0}).
\]
 Moreover, we know that 
\[
\dim\mathfrak{osp}(r_{0}|s_{0})=\frac{(r_{0}+s_{0})^{2}-(r_{0}+s_{0})}{2}+s_{0}=\frac{c_{0}^{2}-r_{0}+s_{0}}{2}.
\]
 Thus we have that $\dim\mathfrak{g}(0)=\sum_{i<0}c_{i}^{2}+\frac{c_{0}^{2}}{2}-\frac{r_{0}}{2}+\frac{s_{0}}{2}=\frac{1}{2}\sum c_{i}^{2}-\frac{r_{0}}{2}+\frac{s_{0}}{2}$.
Observe that 
\[
\mathfrak{g}(-1)=\langle E_{k,l}+\gamma_{l,k}E_{-l,-k}:\mathrm{col}(l)-\mathrm{col}(k)=-1\rangle\cong\bigoplus_{i<0}\mathrm{Hom}(\mathbb{C}^{c_{i}},\mathbb{C}^{c_{i+1}}).
\]
Hence, we have that $\dim\mathfrak{g}(-1)=\sum_{i<0}c_{i}c_{i+1}$. 

\noindent For each row which corresponds to an odd $\lambda_{i}$,
there must exist a box in the $0$th column. Thus we have that $r_{0}=\ensuremath{\left|\{i:\lambda_{i}\text{ is odd},\ensuremath{\left|i\right|}=\bar{0}\}\right|}$
and $s_{0}=\ensuremath{\left|\{i:\lambda_{i}\text{ is odd},\ensuremath{\left|i\right|}=\bar{1}\}\right|}$.
Therefore, we deduce that $\dim\mathfrak{g}^{e}=\frac{1}{2}(\sum c_{i}^{2}+\sum c_{i}c_{i+1})-\frac{r_{0}}{2}+\frac{s_{0}}{2}$.
\end{proof}

\subsection{Centre of centralizer of nilpotent element $e\in\mathfrak{g}{}_{\bar{0}}$
with Jordan type $\lambda$ such that all parts of $\lambda$ have
even multiplicity\label{subsec:cc-osp}}

\begin{singlespace}
\noindent We know that there is a root space decomposition $\mathfrak{g}=\mathfrak{h}\oplus\bigoplus_{\alpha\in\Phi}\mathfrak{g}_{\alpha}$
and thus $\mathfrak{z}(\mathfrak{g})\subseteq\mathfrak{g}^{\mathfrak{h}}=\mathfrak{h}$.
Construct a Dynkin pyramid $\tilde{P}$ following the rules described
in Subsection \ref{subsec:Alternative-Dynkin-pyramid}. We know that
$\tilde{P}$ determines a nilpotent element $e\in\mathfrak{g}_{\bar{0}}$
and we can embed $e$ into an $\mathfrak{sl}(2)$-triple $\{e,h,f\}\subseteq\mathfrak{g}_{\bar{0}}$
by the Jacobson--Morozov theorem. Then we have that the centralizer
$\mathfrak{h}^{e}$ of $e$ in $\mathfrak{h}$ is a Cartan subalgebra
of $\mathfrak{g}^{e}$ according to \cite[Section 3]{Brundan2006}.
Moreover, \cite[Lemma 13]{Brundan2006} shows that the set of weights
of $\mathfrak{h}^{e}$ on $\mathfrak{g}^{e}$ is equivalent to the
set of weights of $\mathfrak{h}^{e}$ on $\mathfrak{g}$. Then we
can obtain the following decomposition for $\mathfrak{g}^{e}$:
\[
\mathfrak{g}^{e}=(\mathfrak{g}^{e})^{\mathfrak{h}^{e}}\oplus\bigoplus_{\alpha\in\Phi^{e}}\mathfrak{g}_{\alpha}^{e}
\]
where $\Phi^{e}\subseteq(\mathfrak{h}^{e})^{*}$ is defined as the
set of non-zero weights of $\mathfrak{h}^{e}$ on $\mathfrak{g}^{e}$
and $\mathfrak{g}_{\alpha}^{e}=\{x\in\mathfrak{g}^{e}:[h,x]=\alpha(h)x\ \text{for all }h\in\mathfrak{h}^{e}\}$.
Hence, we have $\mathfrak{z}(\mathfrak{g}^{e})\subseteq(\mathfrak{g}^{e})^{\mathfrak{h}^{e}}.$
\end{singlespace}

\begin{singlespace}
We first consider the case where all parts of the Jordan type $\lambda=(\lambda_{1},\lambda_{-1},\dots,\lambda_{b},\lambda_{-b})$
with respect to $e$ have even multiplicity. Then we know that there
are $2b$ rows in the ortho-symplectic Dynkin pyramid $P$ and we
label rows in the upper half plane from bottom to top by $1,2,\dots,b$
and rows in the lower half plane in a symmetric way. Let $S$ be the
set spanned by the odd powers of $e$, i.e. $S=\langle e,e^{3},\dots,e^{t}:t=2\lfloor\frac{\lambda_{1}}{2}\rfloor-1\rangle$.
\end{singlespace}
\begin{thm}
\textup{\label{thm:osp even multiplicity}}Let $\mathfrak{g}=\mathfrak{osp}(m|2n)=\mathfrak{g}_{\bar{0}}\oplus\mathfrak{g}_{\bar{1}}$
and the Jordan type with respect to $e\in\mathfrak{g}_{\bar{0}}$
is \textup{$\lambda=(\lambda_{1},\lambda_{-1},\dots,\lambda_{b},\lambda_{-b})$}
such that \textup{$\lambda_{1}\geq\dots\geq\lambda_{b}$} and \textup{$\lambda_{i}=\lambda_{-i}$}.
Then \textup{$\mathfrak{z}(\mathfrak{g}^{e})=S$} except when \textup{$\lambda_{1}$}
is odd and \textup{$\lambda_{1}>\lambda_{i}$} for \textup{$i\neq\pm1$}
and \textup{$\vert1\vert=\bar{0}$}. In which case, we have that \textup{$\mathfrak{z}(\mathfrak{g}^{e})=S\oplus\langle\xi_{1}^{1,\lambda_{1}-1}-\xi_{-1}^{-1,\lambda_{1}-1}\rangle$.}
\end{thm}

\noindent \begin{proof}This proof will proceed in steps. It is clear
that $S\subseteq\mathfrak{z}(\mathfrak{g}^{e})$. We know that $e^{l}=\sum_{j=1}^{b}(\xi_{j}^{j,l}\pm\xi_{-j}^{-j,l})$
and $e^{l}\in\mathfrak{g}_{\bar{0}}$ for all odd $l$ with $0\leq l<\lambda_{1}$.

\noindent \textbf{Step 1}: Deduce that $\mathfrak{z}(\mathfrak{g}^{e})\subseteq\langle\xi_{j}^{j,\lambda_{j}-1-k}+\varepsilon_{j}^{j,\lambda_{j}-1-k}\xi_{-j}^{-j,\lambda_{j}-1-k}:1\leq j\leq b,0\leq k\leq\lambda_{j}-1\rangle$.

\noindent Define $h_{i}=\xi_{i}^{i,0}-\xi_{-i}^{-i,0}$ for all $0\leq i\leq b$.
Then we have $\mathfrak{h}^{e}=\langle h_{i}:0\leq i\leq b\rangle$.
We define $\beta_{i}\in(\mathfrak{h}^{e})^{*}$ by 
\[
\beta_{i}(h_{j})=\begin{cases}
1 & \text{if }i=j;\\
0 & \text{if }i\neq j,
\end{cases}
\]
and $\beta_{-i}=-\beta_{i}$ for all $i$. As in the proof of \cite[Theorem 2]{Yakimova2009},
we can calculate that $[\xi_{i}^{i,0},\xi_{j}^{t,s}]=\delta_{it}\xi_{j}^{i,s}-\delta_{ij}\xi_{i}^{t,s}.$
Thus for $h\in\mathfrak{h}^{e}$, computing the commutator between
$h$ and the basis element $\xi_{l}^{j,\lambda_{j}-1-k}+\varepsilon_{l}^{j,\lambda_{j}-1-k}\xi_{-j}^{-l,\lambda_{l}-1-k}$
( $l,j\geq0,$ $l\neq j$ ) of $\mathfrak{g}^{e}$ we obtain:
\begin{equation}
[h,\xi_{l}^{j,\lambda_{j}-1-k}+\varepsilon_{l}^{j,\lambda_{j}-1-k}\xi_{-j}^{-l,\lambda_{l}-1-k}]=(\beta_{j}-\beta_{l})(h)(\xi_{l}^{j,\lambda_{j}-1-k}+\varepsilon_{l}^{j,\lambda_{j}-1-k}\xi_{-j}^{-l,\lambda_{l}-1-k}).\label{eq:=00005Bh,1=00005D}
\end{equation}
 We also compute the commutator between $h$ and the basis element
$\xi_{l}^{-j,\lambda_{j}-1-k}+\varepsilon_{l}^{-j,\lambda_{j}-1-k}\xi_{j}^{-l,\lambda_{l}-1-k}$
( $l,j\geq0,$ $l\neq j$ ) of $\mathfrak{g}^{e}$:
\begin{equation}
[h,\xi_{l}^{-j,\lambda_{j}-1-k}+\varepsilon_{l}^{-j,\lambda_{j}-1-k}\xi_{j}^{-l,\lambda_{l}-1-k}]=(-\beta_{j}-\beta_{l})(h)(\xi_{l}^{-j,\lambda_{j}-1-k}+\varepsilon_{l}^{-j,\lambda_{j}-1-k}\xi_{j}^{-l,\lambda_{l}-1-k})\label{eq:=00005Bh,2=00005D}
\end{equation}
 and the commutator between $h$ and the basis element $\xi_{-l}^{j,\lambda_{j}-1-k}+\varepsilon_{-l}^{j,\lambda_{j}-1-k}\xi_{-j}^{l,\lambda_{l}-1-k}$
( $l,j\geq0,$ $l\neq j$ ) of $\mathfrak{g}^{e}$:
\[
[h,\xi_{-l}^{j,\lambda_{j}-1-k}+\varepsilon_{-l}^{j,\lambda_{j}-1-k}\xi_{-j}^{l,\lambda_{l}-1-k}]=(\beta_{j}+\beta_{l})(h)(\xi_{-l}^{j,\lambda_{j}-1-k}+\varepsilon_{-l}^{j,\lambda_{j}-1-k}\xi_{-j}^{l,\lambda_{l}-1-k}).
\]
Hence, the coefficient of $\xi_{l}^{j,\lambda_{j}-1-k}+\varepsilon_{l}^{j,\lambda_{j}-1-k}\xi_{-j}^{-l,\lambda_{l}-1-k}$
in an element of $(\mathfrak{g}^{e})^{\mathfrak{h}^{e}}$ can be nonzero
if and only if $(\beta_{j}-\beta_{l})(h)=0$ for all $h\in\mathfrak{h}^{e}$,
the coefficient of $\xi_{l}^{-j,\lambda_{j}-1-k}+\varepsilon_{l}^{-j,\lambda_{j}-1-k}\xi_{j}^{-l,\lambda_{l}-1-k}$
in an element of $(\mathfrak{g}^{e})^{\mathfrak{h}^{e}}$ can be nonzero
if and only if $(-\beta_{j}-\beta_{l})(h)=0$ for all $h\in\mathfrak{h}^{e}$
and the coefficient of $\xi_{-l}^{j,\lambda_{j}-1-k}+\varepsilon_{-l}^{j,\lambda_{j}-1-k}\xi_{-j}^{l,\lambda_{l}-1-k}$
in an element of $(\mathfrak{g}^{e})^{\mathfrak{h}^{e}}$ can be nonzero
if and only if $(\beta_{j}+\beta_{l})(h)=0$. Take $h=h_{j}$, we
obtain that $(\beta_{j}-\beta_{l})(h)=1$ and $(-\beta_{j}-\beta_{l})(h)=-1$
for $l\neq j$. Therefore, we deduce that 
\[
(\mathfrak{g}^{e})^{\mathfrak{h}^{e}}=\langle\xi_{j}^{j,\lambda_{j}-1-k}+\varepsilon_{j}^{j,\lambda_{j}-1-k}\xi_{-j}^{-j,\lambda_{j}-1-k}\rangle
\]
 and thus $\mathfrak{z}(\mathfrak{g}^{e})\subseteq\langle\xi_{j}^{j,\lambda_{j}-1-k}+\varepsilon_{j}^{j,\lambda_{j}-1-k}\xi_{-j}^{-j,\lambda_{j}-1-k}\rangle$.

\noindent We now have that an element $x$ in $\mathfrak{z}(\mathfrak{g}^{e})$
is of the form $\sum_{j,k}c_{j}^{j,k}(\xi_{j}^{j,\lambda_{j}-1-k}+\varepsilon_{j}^{j,\lambda_{j}-1-k}\xi_{-j}^{-j,\lambda_{j}-1-k})$
from Step 1. Fix $j$ and $k$ and let $l=\lambda_{j}-1-k$.

\noindent \textbf{Step 2}: Show that $c_{j}^{j,l}=0$ whenever $l$
is even except in one special case. 

\noindent According to Equation (1) in \cite[Section 3.2]{Jantzen2004},
we have that 
\[
\varepsilon_{j}^{j,l}=\begin{cases}
1 & \text{if }l\text{ is odd;}\\
-1 & \text{if }l\text{ is even.}
\end{cases}
\]

\noindent Next we consider an element $\xi_{i}^{-i,0}$. Note that
$\xi_{i}^{-i,0}\in\mathfrak{g}^{e}$ if $\lambda_{i}$ is even and
$\ensuremath{\left|i\right|}=\bar{0}$, or $\lambda_{i}$ is odd and
$\ensuremath{\left|i\right|}=\bar{1}$. Hence, when $\xi_{i}^{-i,0}\in\mathfrak{g}^{e}$,
the commutator between $x$ and $\xi_{i}^{-i,0}$ is:
\[
[\xi_{i}^{-i,0},\sum_{j,l}c_{j}^{j,l}(\xi_{j}^{j,l}+\varepsilon_{j}^{j,l}\xi_{-j}^{-j,l})]=\sum c_{j}^{j,l}(\xi_{j}^{-j,l}-\varepsilon_{j}^{j,l}\xi_{j}^{-j,l}).
\]
 Hence, we deduce that $c_{j}^{j,l}=0$ whenever $l$ is even. When
$\xi_{i}^{-i,0}\notin\mathfrak{g}^{e}$, i.e. $\lambda_{i}$ is odd
and $\ensuremath{\left|i\right|}=\bar{0}$, or $\lambda_{i}$ is even
and $\ensuremath{\left|i\right|}=\bar{1}$. We take commutator between
$x$ and $\xi_{i}^{-i,1}$:
\[
[\xi_{i}^{-i,1},\sum_{j,l}c_{j}^{j,l}(\xi_{j}^{j,l}+\varepsilon_{j}^{j,l}\xi_{-j}^{-j,l})]=\sum c_{j}^{j,l}(\xi_{j}^{-j,l+1}-\varepsilon_{j}^{j,l}\xi_{j}^{-j,l+1}).
\]
 Hence, we deduce that $c_{j}^{j,l}=0$ for all $l$ is even, except
when $l=\lambda_{j}-1$ and $\ensuremath{\left|i\right|}=\bar{0}$.
For $l$ even and $l=\lambda_{j}-1$, suppose that there exist some
$\lambda_{i}$ with $i>0$ such that $\lambda_{i}\geq\lambda_{j}$,
then we can compute 
\[
[\xi_{i}^{j,0}+\varepsilon_{i}^{j,0}\xi_{-j}^{-i,0},\sum_{j}c_{j}^{j,\lambda_{j}-1}(\xi_{j}^{j,\lambda_{j}-1}-\xi_{-j}^{-j,\lambda_{j}-1})]=-\sum_{j}c_{j}^{j,\lambda_{j}-1}(\varepsilon_{i}^{j,0}\xi_{-j}^{-i,\lambda_{j}-1}+\xi_{i}^{j,\lambda_{j}-1}),
\]
which implies that $c_{j}^{j,\lambda_{j}-1}=0$. Thus $\mathfrak{z}(\mathfrak{g}^{e})\subseteq\langle\xi_{j}^{j,l}+\xi_{-j}^{-j,l}:l\text{ is odd}\rangle$
except for $\lambda_{1}$ is odd with $\ensuremath{\left|1\right|}=\bar{0}$
such that $\lambda_{1}>\lambda_{i}$ for $i\neq\pm1$, in which case
we cannot show that $c_{1}^{1,\lambda_{1}-1}=0$ and $\mathfrak{z}(\mathfrak{g}^{e})\subseteq\langle\xi_{j}^{j,l}+\xi_{-j}^{-j,l}:l\text{ is odd}\rangle\oplus\langle\xi_{1}^{1,\lambda_{1}-1}-\xi_{-1}^{-1,\lambda_{1}-1}\rangle$.

\noindent \textbf{Step 3}: Show that $c_{i}^{i,l}=c_{t}^{t,l}$ for
all $i,t$ whenever $l$ is odd.

\noindent Now consider $\xi_{i}^{t,0}+\varepsilon_{i}^{t,0}\xi_{-t}^{-i,\lambda_{i}-\lambda_{t}}\in\mathfrak{g}^{e}$
with $i<t$ and $i,t>0$. Then we compute that
\begin{align*}
\left[\xi_{i}^{t,0}+\varepsilon_{i}^{t,0}\xi_{-t}^{-i,\lambda_{i}-\lambda_{t}},x\right] & =\left[\xi_{i}^{t,0}+\varepsilon_{i}^{t,0}\xi_{-t}^{-i,\lambda_{i}-\lambda_{t}},\sum_{j,l}c_{j}^{j,l}(\xi_{j}^{j,l}+\xi_{-j}^{-j,l})\right]\\
 & =\sum_{l}(c_{i}^{i,l}-c_{t}^{t,l})\xi_{i}^{t,l}+\sum_{l}(c_{t}^{t,l}-c_{i}^{i,l})\varepsilon_{i}^{t,0}\xi_{-t}^{-i,\lambda_{i}-\lambda_{t}+l}.
\end{align*}
 This equals to zero if and only if $c_{i}^{i,l}=c_{t}^{t,l}$ for
all $i$ and $t$. Hence, $\mathfrak{z}(\mathfrak{g}^{e})\subseteq S$
and therefore $\mathfrak{z}(\mathfrak{g}^{e})=S$ except for $\lambda_{1}>\lambda_{i}$
for $i\neq\pm1$ and $\ensuremath{\left|1\right|}=\bar{0}$.

\noindent \textbf{Step 4}: Show that $\xi_{1}^{1,\lambda_{1}-1}-\xi_{1^{*}}^{1^{*},\lambda_{1}-1}\in\mathfrak{z}(\mathfrak{g}^{e})$
when $\lambda_{1}$ is odd with $\ensuremath{\left|1\right|}=\bar{0}$
such that $\lambda_{1}>\lambda_{i}$ for $i\neq\pm1$.

\noindent Suppose $\lambda_{1}$ is odd with $\ensuremath{\left|1\right|}=\bar{0}$
and $l=\lambda_{1}-1$. Suppose $\lambda_{1}>\lambda_{i}$ for $i\neq\pm1$.
Clearly $\xi_{1}^{1,\lambda_{1}-1}-\xi_{-1}^{-1,\lambda_{1}-1}$ commutes
with all basis elements in $\mathfrak{g}^{e}$ of the form $\xi_{i}^{-i,\lambda_{i}-1-k}$
for $i\neq\pm1$ and $\xi_{i}^{j,\lambda_{j}-1-k}\pm\xi_{-j}^{-i,\lambda_{i}-1-k}$
for $i,j\neq\pm1$.

\noindent It remains to check whether $\xi_{1}^{1,\lambda_{1}-1}-\xi_{-1}^{-1,\lambda_{1}-1}$
commutes with $\xi_{1}^{-1,\lambda_{1}-1-k}$, $\xi_{-1}^{1,\lambda_{1}-1-k}$
for $0\leq k\leq\lambda_{1}-1$, $k$ is odd, $\xi_{1}^{j,\lambda_{j}-1-k}\pm\xi_{-j}^{-1,\lambda_{1}-1-k}$,
$\xi_{-1}^{j,\lambda_{j}-1-k}\pm\xi_{-j}^{1,\lambda_{1}-1-k}$ for
$0\leq k\leq\lambda_{j}-1$ and $\xi_{i}^{1,\lambda_{1}-1-k}\pm\xi_{-1}^{-i,\lambda_{i}-1-k}$,
$\xi_{i}^{-1,\lambda_{1}-1-k}\pm\xi_{1}^{-i,\lambda_{i}-1-k}$ for
$0\leq k\leq\lambda_{i}-1$. Note that 
\[
[\xi_{1}^{1,\lambda_{1}-1}-\xi_{-1}^{-1,\lambda_{1}-1},\xi_{1}^{-1,\lambda_{1}-1-k}]=-2\xi_{1}^{-1,2\lambda_{1}-2-k}
\]
and 
\[
[\xi_{1}^{1,\lambda_{1}-1}-\xi_{-1}^{-1,\lambda_{1}-1},\xi_{1}^{-1,\lambda_{1}-1-k}]=2\xi-_{1}^{1,2\lambda_{1}-2-k}.
\]
 We know that $\xi_{1}^{-1,2\lambda_{1}-2-k}=0$ and $\xi-_{1}^{1,2\lambda_{1}-2-k}=0$
because $2\lambda_{1}-2-k\geq\lambda_{1}$. Similarly we can compute
that $[\xi_{1}^{1,\lambda_{1}-1}-\xi_{-1}^{-1,\lambda_{1}-1},\xi_{1}^{j,\lambda_{j}-1-k}\pm\xi_{-j}^{-1,\lambda_{1}-1-k}]=0$,
$[\xi_{1}^{1,\lambda_{1}-1}-\xi_{-1}^{-1,\lambda_{1}-1},\xi_{-1}^{j,\lambda_{j}-1-k}\pm\xi_{-j}^{1,\lambda_{1}-1-k}]=0$,
$[\xi_{1}^{1,\lambda_{1}-1}-\xi_{-1}^{-1,\lambda_{1}-1},\xi_{i}^{1,\lambda_{1}-1-k}\pm\xi_{-1}^{-i,\lambda_{i}-1-k}]=0$
and $[\xi_{1}^{1,\lambda_{1}-1}-\xi_{-1}^{-1,\lambda_{1}-1},\xi_{i}^{-1,\lambda_{1}-1-k}\pm\xi_{1}^{-i,\lambda_{i}-1-k}]=0$.
Hence, we have that $\xi_{1}^{1,\lambda_{1}-1}-\xi_{-1}^{-1,\lambda_{1}-1}$
commutes with all basis elements in $\mathfrak{g}^{e}$. Therefore,
we have that $\xi_{1}^{1,\lambda_{1}-1}-\xi_{-1}^{-1,\lambda_{1}-1}\in\mathfrak{z}(\mathfrak{g}^{e})$
in this case and $\mathfrak{z}(\mathfrak{g}^{e})=S\oplus\langle\xi_{1}^{1,\lambda_{1}-1}-\xi_{-1}^{-1,\lambda_{1}-1}\rangle$.
\end{proof}

\subsection{Centre of centralizer of nilpotent element $e\in\mathfrak{g}{}_{\bar{0}}$
with Jordan type $\lambda$ such that all parts of $\lambda$ have
multiplicity one}

\begin{singlespace}
Next we consider the case where all parts of the Jordan type $\lambda=(\lambda_{1},\dots,\lambda_{a})$
of $e$ have multiplicity one, which implies $\lambda_{i}$ is odd
for $\ensuremath{\left|i\right|}=\bar{0}$ and $\lambda_{i}$ is even
for $\ensuremath{\left|i\right|}=\bar{1}$. Note that when $m=0$
or $n=0$, then $\mathfrak{g}$ is either an orthogonal or symplectic
Lie algebra, a basis of $\mathfrak{z}(\mathfrak{g}^{e})$ has been
given in \cite[Theorem 4]{Yakimova2009}. Recall that $S=\langle e,e^{3},\dots,e^{t}:t=2\lfloor\frac{\lambda_{1}}{2}\rfloor-1\rangle$.
Below we give a general result for $\mathfrak{z}(\mathfrak{g}^{e})$.
\end{singlespace}
\begin{thm}
\begin{singlespace}
\textup{\label{Thm: osp multiplicity 1}}Let $\mathfrak{g}=\mathfrak{g}_{\bar{0}}\oplus\mathfrak{g}_{\bar{1}}=\mathfrak{osp}(m|2n)$
and let $e\in\mathfrak{g}_{\bar{0}}$ be nilpotent with a partition
$\lambda=(\lambda_{1},\dots,\lambda_{a})$ which is defined in \textup{(\ref{eq:jordan type group})}
and $\lambda_{1}>\dots>\lambda_{a}$.

\ \ \ \ (1) If $m=0$ or $n=0$, we have $\mathfrak{z}(\mathfrak{g}^{e})=S$
except when $n=0$, $\lambda_{2}>\lambda_{3}$ and both $\lambda_{1}$
and $\lambda_{2}$ are odd, in which case we have $\mathfrak{z}(\mathfrak{g}^{e})=S\oplus\langle\xi_{1}^{2,\lambda_{2}-1}-\xi_{2}^{1,\lambda_{1}-1}\rangle$;

\ \ \ \ (2) If $m,n\neq0$, we have $\mathfrak{z}(\mathfrak{g}^{e})=S$
except when $a\geq3$, $\ensuremath{\left|1\right|}=\ensuremath{\left|2\right|}=\bar{0}$,
or $a=2$ and $\ensuremath{\left|1\right|}\neq\ensuremath{\left|2\right|}$,
in which cases we have $\mathfrak{z}(\mathfrak{g}^{e})=S\oplus\langle\xi_{1}^{2,\lambda_{2}-1}-\xi_{2}^{1,\lambda_{1}-1}\rangle$. 
\end{singlespace}
\end{thm}

\begin{singlespace}
\noindent \begin{proof}For $m=0$ or $n=0$, then $\mathfrak{g}=\mathfrak{sp}(2n)$
or $\mathfrak{so}(m)$, the detailed proof can be found in \cite[Theorem 4]{Yakimova2009}. 

\noindent For $m,n\neq0$, it is clear that $S\subseteq\mathfrak{z}(\mathfrak{g}^{e})$.
We know that $e^{l}=\sum_{i=1}^{a}\xi_{i}^{i,l}$ and $e^{l}\in\mathfrak{g}$
for all odd $l$ and $0\leq l\leq\lambda_{1}-1$. Note that a basis
of $\mathfrak{g}^{e}$ contains elements of the form:
\[
\xi_{i}^{i,k}\text{ for all }1\leq i\leq a\text{ and odd }k\text{ with }0<k\leq\lambda_{i}-1,
\]
\[
\xi_{i}^{j,\lambda_{j}-1-k}+\varepsilon_{i}^{j,\lambda_{j}-1-k}\xi_{j}^{i,\lambda_{i}-1-k}\text{ for all }1\leq i<j\leq a,\ 0\leq k\leq\lambda_{j}-1
\]
where $\varepsilon_{i}^{j,\lambda_{j}-1-k}\in\{\pm1\}$ can be determined.
Thus an element $x\in\mathfrak{z}(\mathfrak{g}^{e})$ is of the form
\[
x=\sum_{i,k}c_{i}^{i,k}\xi_{i}^{i,k}+\sum_{i,j,k}c_{i}^{j,k}(\xi_{i}^{j,\lambda_{j}-1-k}+\varepsilon_{i}^{j,\lambda_{j}-1-k}\xi_{j}^{i,\lambda_{i}-1-k})
\]
 where $c_{i}^{j,k}\in\mathbb{C}$ are coefficients. 

\noindent Assume $a\geq3$. For $1\leq t\leq a$, we have that $\xi_{t}^{t,1}$
commutes with $\sum_{i,k}c_{i}^{i,k}\xi_{i}^{i,k}$. By taking the
commutator between $\xi_{t}^{t,1}$ and $x$ we obtain that 
\begin{align*}
[\xi_{t}^{t,1},x] & =\sum_{i<t}\sum_{k=0}^{\lambda_{t}-1}c_{i}^{t,k}(\xi_{i}^{t,\lambda_{t}-k}-\varepsilon_{i}^{t,\lambda_{t}-1-k}\xi_{t}^{i,\lambda_{i}-k})\\
 & +\sum_{t<i}\sum_{k=0}^{\lambda_{i}-1}c_{t}^{i,k}(\varepsilon_{t}^{i,\lambda_{i}-1-k}\xi_{i}^{t,\lambda_{t}-k}-\xi_{t}^{i,\lambda_{i}-k}).
\end{align*}
 This is equal to $0$ for all $t$ if and only if $c_{i}^{j,k}=0$
for all $0<k\leq\lambda_{j}-1$. Now we have that $x=\sum_{i,k}c_{i}^{i,k}\xi_{i}^{i,k}+\sum_{i,j}c_{i}^{j,0}(\xi_{i}^{j,\lambda_{j}-1}+\varepsilon_{i}^{j,\lambda_{j}-1}\xi_{j}^{i,\lambda_{i}-1})$. 

\noindent For $1\leq l<h\leq a$, by taking the commutator between
$\xi_{l}^{h,0}+\varepsilon_{l}^{h,0}\xi_{h}^{l,\lambda_{l}-\lambda_{h}}$
and $x$ we obtain that
\begin{align*}
[\xi_{l}^{h,0}+\varepsilon_{l}^{h,0}\xi_{h}^{l,\lambda_{l}-\lambda_{h}},x] & =\sum_{k}(c_{l}^{l,k}-c_{h}^{h,k})(\xi_{l}^{h,k}-\varepsilon_{l}^{h,0}\xi_{h}^{l,\lambda_{l}-\lambda_{h}+k})\\
 & +\sum_{i}c_{i}^{l,0}(\xi_{i}^{h,\lambda_{l}-1}\pm\varepsilon_{l}^{h,0}\varepsilon_{i}^{j,\lambda_{j}-1}\xi_{h}^{i,\lambda_{l}-\lambda_{h}+\lambda_{i}-1})\\
 & +\sum_{j}c_{l}^{j,0}(\varepsilon_{i}^{j,\lambda_{j}-1}\xi_{j}^{h,\lambda_{l}-1}\pm\varepsilon_{l}^{h,0}\xi_{h}^{j,\lambda_{l}-\lambda_{h}+\lambda_{j}-1})\\
 & +\sum_{i}c_{i}^{h,0}(\varepsilon_{l}^{h,0}\xi_{i}^{l,\lambda_{l}-1}\pm\varepsilon_{i}^{j,\lambda_{j}-1}\xi_{l}^{i,\lambda_{i}-1})\\
 & +\sum_{j}c_{h}^{j,0}(\varepsilon_{l}^{h,0}\varepsilon_{i}^{j,\lambda_{j}-1}\xi_{j}^{l,\lambda_{l}-1}\pm\xi_{l}^{j,\lambda_{j}-1}).
\end{align*}
 Note that $\xi_{i}^{h,\lambda_{l}-1}=\xi_{j}^{h,\lambda_{l}-1}=\xi_{h}^{i,\lambda_{l}-\lambda_{h}+\lambda_{i}-1}=\xi_{h}^{j,\lambda_{l}-\lambda_{h}+\lambda_{j}-1}=0$
since $\lambda_{l}>\lambda_{h}$. Hence, we have that 
\begin{align}
[\xi_{l}^{h,0}+\varepsilon_{l}^{h,0}\xi_{h}^{l,\lambda_{l}-\lambda_{h}},x] & =\sum_{k}(c_{l}^{l,k}-c_{h}^{h,k})(\xi_{l}^{h,k}-\varepsilon_{l}^{h,0}\xi_{h}^{l,\lambda_{l}-\lambda_{h}+k})\nonumber \\
 & +\sum_{i}c_{i}^{h,0}(\varepsilon_{l}^{h,0}\xi_{i}^{l,\lambda_{l}-1}\pm\varepsilon_{i}^{j,\lambda_{j}-1}\xi_{l}^{i,\lambda_{i}-1})\nonumber \\
 & +\sum_{j}c_{h}^{j,0}(\varepsilon_{l}^{h,0}\varepsilon_{i}^{j,\lambda_{j}-1}\xi_{j}^{l,\lambda_{l}-1}\pm\xi_{l}^{j,\lambda_{j}-1}).\label{eq:proof of mul one}
\end{align}
This is equal to $0$ if and only if $c_{l}^{l,k}=c_{h}^{h,k}$ for
all $1\leq l<h\leq a$ and $c_{i}^{j,0}=0$ for all $1\leq i<j\leq a$
except when $(i,j)=(1,2)$ and $(\ensuremath{\left|1\right|},\ensuremath{\left|2\right|})=(\bar{0},\bar{0})$.
When $\ensuremath{\left|1\right|}=\ensuremath{\left|2\right|}=\bar{0}$,
the commutator between $\xi_{l}^{h,0}+\varepsilon_{l}^{h,0}\xi_{h}^{l,\lambda_{l}-\lambda_{h}}$
and $\xi_{1}^{2,\lambda_{2}-1}-\xi_{2}^{1,\lambda_{1}-1}$ gives terms
$\varepsilon_{l}^{2,0}\xi_{1}^{l,\lambda_{l}-1}+\xi_{l}^{1,\lambda_{1}-1}-\varepsilon_{l}^{1,0}\xi_{2}^{l,\lambda_{l}-1}-\xi_{l}^{2,\lambda_{2}-1}$
and we check this is equal to zero for $l<h$. Hence, when $a\geq3$,
we have that $\mathfrak{z}(\mathfrak{g}^{e})\subseteq S$ and thus
$\mathfrak{z}(\mathfrak{g}^{e})=S$ except when $\ensuremath{\left|1\right|}=\ensuremath{\left|2\right|}=\bar{0}$,
in which case we have $\mathfrak{z}(\mathfrak{g}^{e})\subseteq S\oplus\left\langle \xi_{1}^{2,\lambda_{2}-1}-\xi_{2}^{1,\lambda_{1}-1}\right\rangle $.

\noindent Next suppose $a\geq3$ and $\ensuremath{\left|1\right|}=\ensuremath{\left|2\right|}=\bar{0}$,
we know that $\xi_{1}^{2,\lambda_{2}-1}-\xi_{2}^{1,\lambda_{1}-1}$
commutes with all basis elements in $\mathfrak{g}_{\bar{0}}^{e}$
by \cite[Theorem 4]{Yakimova2009}. Hence, it remains to check that
$\xi_{1}^{2,\lambda_{2}-1}-\xi_{2}^{1,\lambda_{1}-1}$ commutes with
basis elements $\xi_{i}^{j,\lambda_{j}-1-k}\pm\xi_{j}^{i,\lambda_{i}-1-k}$
for $\ensuremath{\left|i\right|}=\bar{0},\ensuremath{\left|j\right|}=\bar{1}$
and $0\leq k\leq\min\{\lambda_{i},\lambda_{j}\}-1$. Computing
\begin{align}
[\xi_{1}^{2,\lambda_{2}-1}-\xi_{2}^{1,\lambda_{1}-1},\xi_{i}^{j,\lambda_{j}-1-k}\pm\xi_{j}^{i,\lambda_{i}-1-k}]=\pm\xi_{j}^{2,\lambda_{2}-1+\lambda_{1}-1-k}\label{eq:osp proof 2}\\
\pm\xi_{j}^{1,\lambda_{1}-1+\lambda_{2}-1-k}-\xi_{1}^{j,\lambda_{2}-1+\lambda_{j}-1-k}+\xi_{2}^{j,\lambda_{1}-1+\lambda_{j}-1-k}.\nonumber 
\end{align}
 We have that all terms in (\ref{eq:osp proof 2}) are equal to $0$
for $0\leq k\leq\min\{\lambda_{i},\lambda_{j}\}-1$ as $\xi_{h}^{l,r}=0$
for all $h,l$ and $r>\lambda_{l}-1$. Hence, we have that $\xi_{1}^{2,\lambda_{2}-1}-\xi_{2}^{1,\lambda_{1}-1}\in\mathfrak{z}(\mathfrak{g}^{e})$
and $\mathfrak{z}(\mathfrak{g}^{e})=S\oplus\left\langle \xi_{1}^{2,\lambda_{2}-1}-\xi_{2}^{1,\lambda_{1}-1}\right\rangle $
in this case.

\noindent When $a=2$ and $\ensuremath{\left|1\right|}\neq\ensuremath{\left|2\right|}$,
i.e. the Jordan type of $e$ is $(m,2n)$ such that $m$ is odd. Assume
that $m\geq2n$, in this case a basis of $\mathfrak{g}^{e}$ only
contains elements of the form:
\[
\xi_{1}^{1,k}\text{ for }k\text{ is odd},1\leq k\leq m-1;\ \xi_{2}^{2,k}\text{ for }k\text{ is odd},\ 1\leq k\leq2n-1;
\]
\[
\text{and }\xi_{1}^{2,2n-1-k}+\varepsilon_{1}^{2,2n-1-k}\xi_{2}^{1,m-1-k}\text{ for }0\leq k\leq2n-1,\varepsilon_{1}^{2,2n-1-k}\in\{-1,1\}.
\]
By applying the similar argument to that used in the case $r+s\geq3$,
we get that $c_{1}^{1,k}=c_{2}^{2,k}$ and $c_{1}^{2,k}=0$ for all
$1\leq k\leq2n-1$. The only remaining element to check is $\xi_{1}^{2,2n-1}-\xi_{2}^{1,m-1}$.
Note that $\xi_{1}^{2,2n-1}-\xi_{2}^{1,m-1}$ commutes with $\xi_{1}^{1,k}$
and $\xi_{2}^{2,k}$ for all $k$ is odd. The element $\xi_{1}^{2,2n-1}-\xi_{2}^{1,m-1}$
also commutes with $\xi_{1}^{2,2n-1-k}+\varepsilon_{1}^{2,2n-1-k}\xi_{2}^{1,m-1-k}$
for all $k=0,1,\dots,2n-2$. Thus we only need to check that $[\xi_{1}^{2,2n-1}-\xi_{2}^{1,m-1},\xi_{1}^{2,0}+\varepsilon_{1}^{2,0}\xi_{2}^{1,m-2n}]=0$.
We calculate 
\begin{align}
[\xi_{1}^{2,2n-1}-\xi_{2}^{1,m-1},\xi_{1}^{2,0}+\varepsilon_{1}^{2,0}\xi_{2}^{1,m-2n}] & =\varepsilon_{1}^{2,0}\xi_{2}^{2,m-1}-\xi_{1}^{1,m-1}-\xi_{2}^{2,m-1}+\varepsilon_{1}^{2,0}\xi_{1}^{1,m-1}\label{eq:two element mul one}\\
 & =(\varepsilon_{1}^{2,0}-1)\xi_{1}^{1,m-1}\nonumber 
\end{align}

\noindent since $\xi_{2}^{2,m-1}=0$ as $\xi_{2}^{2,k}=0$ for $k>2n-1$. 

\noindent Next we want to know the value of $\varepsilon_{1}^{2,0}$.
Let $\mathfrak{osp}(m|2n)=\mathfrak{osp}(V)$ and $V=V_{\bar{0}}\oplus V_{\bar{1}}$.
Then we know that there exist $u_{1},u_{2}\in V$ such that $u_{1},eu_{1},\dots,e^{m-1}u_{1}$
(resp. $u_{2},eu_{2},\dots,e^{2n-1}u_{2}$) is a basis for $V_{\bar{0}}$
(resp. $V_{\bar{1}}$) according to Subsection \ref{subsec:c-osp}.
Moreover, we know that $B(u_{1},e^{m-1}u_{1})=1$ and $B(u_{2},e^{2n-1}u_{2})=1$
by equations (\ref{eq:c-osp-bilinear-form-1}) and (\ref{eq:c-osp-bilinear-form-2}).
By using equation (\ref{eq:osp(V)}), we have that 
\[
B((\xi_{1}^{2,0}+\varepsilon_{1}^{2,0}\xi_{2}^{1,m-2n})u_{1},e^{2n-1}u_{2})=(u_{2},e^{2n-1}u_{2})=1
\]
and 
\begin{align*}
B((\xi_{1}^{2,0}+\varepsilon_{1}^{2,0}\xi_{2}^{1,m-2n})u_{1},e^{2n-1}u_{2}) & =B(u_{1},(\xi_{1}^{2,0}+\varepsilon_{1}^{2,0}\xi_{2}^{1,m-2n})e^{2n-1}u_{2})\\
 & =B(u_{1},\varepsilon_{1}^{2,0}e^{m-1}u_{1})=\varepsilon_{1}^{2,0}.
\end{align*}
Therefore, we obtain that $\varepsilon_{1}^{2,0}=1$. Therefore, we
have that $[\xi_{1}^{2,2n-1}-\xi_{2}^{1,m-1},\xi_{1}^{2,0}+\varepsilon_{1}^{2,0}\xi_{2}^{1,m-2n}]=0$
and thus $\xi_{1}^{2,2n-1}-\xi_{2}^{1,m-1}\in\mathfrak{z}(\mathfrak{g}^{e})$. 

\noindent Therefore, we deduce that $\xi_{1}^{2,2n-1}-\xi_{2}^{1,m-1}\in\mathfrak{z}(\mathfrak{g}^{e})$
and $\mathfrak{z}(\mathfrak{g}^{e})=S\oplus\langle\xi_{1}^{2,\lambda_{2}-1}-\xi_{2}^{1,\lambda_{1}-1}\rangle$
as required. 

\noindent When $m<2n$, we obtain same result by applying a similar
argument. \end{proof}
\end{singlespace}

\subsection{Centre of centralizer of general nilpotent element $e\in\mathfrak{g}_{\bar{0}}$}

\begin{singlespace}
\noindent For Lie superalgebras $\mathfrak{g}=\mathfrak{osp}(m|2n)$,
we already know the construction of $\mathfrak{z}(\mathfrak{g}^{e})$
if: (1) the Jordan type of $e$ has all parts with even multiplicity;
(2) the Jordan type of $e$ has all parts with multiplicity $1$.
Now we want to use Theorems \ref{thm:osp even multiplicity} and \ref{Thm: osp multiplicity 1}
to deduce a basis of $\mathfrak{z}(\mathfrak{g}^{e})$ for a general
nilpotent element $e$. 
\end{singlespace}

\begin{singlespace}
Let $e\in\mathfrak{g}_{\bar{0}}$ be a nilpotent element with Jordan
type $\lambda$ as denoted in (\ref{eq:jordan type group}). Take
a Dynkin pyramid $\tilde{P}$ following Subsection \ref{subsec:Alternative-Dynkin-pyramid}
and $\{v_{i_{j}}\}$ in (\ref{eq:vij-alternative dynkin}) form a
basis for $V$ with respect to $\tilde{P}$. Now we write $V=V_{1}\oplus V_{2}$
where $\{v_{i_{j}}:1\leq i\leq a\}$ form a basis for $V_{1}$ and
$\{v_{i_{j}},v_{-i_{j}}:a+1\leq i\leq b\}$ form a basis for $V_{2}$.
Define $\mathfrak{g'=g}_{1}\oplus\mathfrak{g}_{2}$ where $\mathfrak{g}_{1}=\mathfrak{osp}(V_{1})$
and $\mathfrak{g}_{2}=\mathfrak{osp}(V_{2})$. Then the nilpotent
element $e\in\mathfrak{g}_{\bar{0}}$ can also be written as $e=e_{1}+e_{2}$
and $e_{i}\in\mathfrak{osp}(V_{i})$ such that the Jordan type of
$e_{1}$ has all parts with multiplicity $1$ and the Jordan type
of $e_{2}$ has all parts with even multiplicity. That is to say,
$(\lambda_{1},\dots,\lambda_{a})$ is the Jordan type of $e_{1}$
in descending order and $(\lambda_{a+1},\lambda_{-\left(a+1\right)},\dots,\lambda_{b},\lambda_{-b})$
is the Jordan type of $e_{2}$. 

Now consider a Cartan subalgebra $\mathfrak{h}$ of $\mathfrak{g}$,
we know that a basis for $\mathfrak{h}^{e}$ is $\{h_{i}=\xi_{i}^{i,0}-\xi_{-i}^{-i,0}:i=a+1,\dots,b\}$.
Define $U_{i}=\mathrm{Span}\{v_{i_{j}}:i_{j}\text{ is a box in }P\}$
for $1\leq i\leq a$ and $U_{i}=U_{i}^{+}\oplus U_{i}^{-}$ for $a+1\leq i\leq b$
where $U_{i}^{+}=\mathrm{Span}\{v_{i_{j}}:i_{j}\text{ is a box in }P\}$
and $U_{i}^{-}=\mathrm{Span}\{v_{-i_{j}}:-i_{j}\text{ is a box in }P\}$.
We also define $U_{i}^{\bot}=\left\{ v\in V:(v,u)=0\text{ for any }u\in U_{i}\right\} $.
Then we have that $\mathfrak{g}^{h_{i}}\cong\mathfrak{gl}\left(U_{i}^{+}\right)\oplus\mathfrak{osp}\left(U_{i}^{\bot}\right)$
where elements of $\mathfrak{gl}(U_{i}^{+})$ can be viewed as elements
of $\mathfrak{osp}(U_{i}^{+}\oplus U_{i}^{-})$. Let $H=\sum_{i=a+1}^{b}h_{i}$,
with the above basis of $V$ we have that $H$ is of the form 
\[
H=\begin{pmatrix}1\\
 & \ddots\\
 &  & 1\\
 &  &  & 0\\
 &  &  &  & \ddots\\
 &  &  &  &  & 0\\
 &  &  &  &  &  & -1\\
 &  &  &  &  &  &  & \ddots\\
 &  &  &  &  &  &  &  & -1
\end{pmatrix},
\]
Hence, we have that 
\[
\mathfrak{g}^{H}\cong\mathfrak{gl}(\bigoplus_{i=a+1}^{b}U_{i}^{+})\oplus\mathfrak{osp}((\bigoplus_{i=1}^{a}U_{i})\subseteq\mathfrak{g}'.
\]
Therefore, we have that $\mathfrak{z}(\mathfrak{g}^{e})\subseteq\mathfrak{g}^{H}\subseteq\mathfrak{g}'$.
Let $x\in\mathfrak{z}(\mathfrak{g}^{e})$, then $x\in(\mathfrak{g}')^{e}$
and $\left[x,y\right]=0$ for all $y\in\mathfrak{g}^{e}$. Since $(\mathfrak{g}')^{e}\subseteq\mathfrak{g}^{e}$,
we have that $[x,y]=0$ for all $y\in(\mathfrak{g}')^{e}$. Therefore,
we deduce that $\mathfrak{z}(\mathfrak{g}^{e})\subseteq\mathfrak{z}((\mathfrak{g}')^{e})=\mathfrak{z}(\mathfrak{g}_{1}^{e_{1}})\oplus\mathfrak{z}(\mathfrak{g}_{2}^{e_{2}})$.
\end{singlespace}

Now we are ready to combine Theorems \ref{thm:osp even multiplicity}
and \ref{Thm: osp multiplicity 1} to obtain a basis for $\mathfrak{z}(\mathfrak{g}^{e})$.
\begin{cor}
\label{cor:centre osp}Suppose that $\mathfrak{g}=\mathfrak{g}_{\bar{0}}\oplus\mathfrak{g}_{\bar{1}}=\mathfrak{osp}(m|2n)$
and $e=e_{1}+e_{2}$ is given by a partition $\lambda=(\lambda_{1},\dots,\lambda_{a},\lambda_{a+1},\lambda_{-(a+1)},\dots,\lambda_{b},\lambda_{-b})$.
Let $S=\langle e^{k}:k\text{ is odd and }1\leq k\leq\max\{\lambda_{1},\lambda_{a+1}\}-1\rangle$.
Then $\mathfrak{z}(\mathfrak{g}^{e})=S$ except for two special cases:

Case 1: If $a\geq3$, $\lambda_{2}>\lambda_{a+1}$, $\ensuremath{\left|1\right|}=\ensuremath{\left|2\right|}=\bar{0}$
or $a=2,\ensuremath{\left|1\right|}\neq\ensuremath{\left|2\right|}$,
then we have that $\mathfrak{z}(\mathfrak{g}^{e})=S\oplus\langle\xi_{1}^{2,\lambda_{2}-1}-\xi_{2}^{1,\lambda_{1}-1}\rangle$;

Case 2: If $\lambda_{1}<\lambda_{a+1}$, $\lambda_{a+1}>\lambda_{a+2}$,
$\ensuremath{\left|a+1\right|}=\bar{0}$ and $\lambda_{a+1}$ is odd,
we have that $\mathfrak{z}(\mathfrak{g}^{e})=S\oplus\langle\xi_{a+1}^{a+1,\mu_{a+1}-1}-\xi_{-(a+1)}^{-(a+1),\mu_{a+1}-1}\rangle$.
\end{cor}

\begin{singlespace}
\noindent \begin{proof}It is clear that $S\subseteq\mathfrak{z}(\mathfrak{g}^{e})$.
We first consider the case that $\mathfrak{z}(\mathfrak{g}_{1}^{e_{1}})=\mathrm{Span}\{e_{1}^{k}:k\text{ is odd and }1\leq k\leq\lambda_{1}-1\}$
and $\mathfrak{z}(\mathfrak{g}_{2}^{e_{2}})=\mathrm{Span}\{e_{2}^{k}:k\text{ is odd and }1\leq k\leq\lambda_{a+1}-1\}$.
According to Theorem \ref{thm:osp even multiplicity} and \ref{Thm: osp multiplicity 1},
we have that elements in a basis of $\mathfrak{z}((\mathfrak{g}')^{e})$
can be written as:
\[
\sum_{t=1}^{a}\xi_{t}^{t,k}\text{ for }k\text{ is odd and }1\leq k\leq\lambda_{1}-1;
\]
\[
\sum_{t=a+1}^{b}(\xi_{t}^{t,k}+\xi_{-t}^{-t,k})\text{ for }k\text{ is odd and }1\leq k\leq\lambda_{a+1}-1.
\]
Thus an element $x\in\mathfrak{z}(\mathfrak{g}^{e})$ is of the form
\begin{equation}
x=\sum_{k\text{ is odd;}k=1}^{\lambda_{1}-1}a_{k}\left(\sum_{t=1}^{a}\xi_{t}^{t,k}\right)+\sum_{k\text{ is odd;}k=1}^{\lambda_{a+1}-1}b_{k}\left(\sum_{t=a+1}^{b}(\xi_{t}^{t,k}+\xi_{-t}^{-t,k})\right)\label{eq:x-in-centre-osp}
\end{equation}
for coefficient $a_{k},b_{k}\in\mathbb{C}$. We assume that $\lambda_{1}\geq\lambda_{a+1}$,
then take the commutator with $\xi_{1}^{a+1,0}+\varepsilon_{1}^{a+1,0}\xi_{-(a+1)}^{1,\lambda_{1}-\lambda_{a+1}}$:
\[
[\xi_{1}^{a+1,0}+\varepsilon_{1}^{a+1,0}\xi_{-(a+1)}^{1,\lambda_{1}-\lambda_{a+1}},x]=\sum_{k\text{ is odd;}k=1}^{\lambda_{a+1}-1}\left(a_{k}-b_{k}\right)\left(\xi_{1}^{a+1,k}\pm\xi_{-(a+1)}^{1,\lambda_{1}-\lambda_{a+1}+k}\right).
\]
This is equal to $0$ if and only if $a_{k}=b_{k}$ for all $k$ is
odd and $1\leq k\leq\lambda_{a+1}-1$. If $\lambda_{1}<\lambda_{a+1}$,
then by taking the commutator between $\xi_{1}^{a+1,\lambda_{a+1}-\lambda_{1}}+\varepsilon_{1}^{a+1,\lambda_{a+1}-\lambda_{1}}\xi_{-(a+1)}^{1,0}$
and $x$, we obtain that $a_{k}=b_{k}$ for all $k$ is odd and $1\leq k\leq\mu_{1}-1$.
Hence, we have that $\mathfrak{z}(\mathfrak{g}^{e})\subseteq S$ for
this case and thus $\mathfrak{z}(\mathfrak{g}^{e})=S$. Therefore,
we have that $\dim\mathfrak{z}(\mathfrak{g}^{e})=\lceil\frac{\max\left\{ \lambda_{1},\lambda_{a+1}\right\} -1}{2}\rceil$
for this case.

\noindent Now we look at the special cases:

\noindent For the special case when $\xi_{1}^{2,\lambda_{2}-1}-\xi_{2}^{1,\lambda_{1}-1}\in\mathfrak{z}(\mathfrak{g}_{1}^{e_{1}})$
and $\xi_{a+1}^{a+1,\lambda_{a+1}-1}-\xi_{-(a+1)}^{-(a+1),\lambda_{a+1}-1}\notin\mathfrak{z}(\mathfrak{g}_{2}^{e_{2}})$,
then an element $y\in\mathfrak{z}(\mathfrak{g}^{e})$ is of the form
$y=x+c_{1}^{2,\lambda_{2}-1}(\xi_{1}^{2,\lambda_{2}-1}-\xi_{2}^{1,\lambda_{1}-1})$
where $x$ is defined in (\ref{eq:x-in-centre-osp}) and $c_{1}^{2,\lambda_{2}-1}\in\mathbb{C}$
is the coefficient. By calculating $[y,\xi_{1}^{a+1,0}+\varepsilon_{1}^{a+1,0}\xi_{-(a+1)}^{1,\lambda_{1}-\lambda_{a+1}}]=0$
for $\lambda_{1}>\lambda_{a+1}$ and $[y,\xi_{1}^{a+1,\lambda_{a+1}-\lambda_{1}}+\varepsilon_{1}^{a+1,\lambda_{a+1}-\lambda_{1}}\xi_{-(a+1)}^{1,0}]=0$
for $\lambda_{1}\leq\lambda_{a+1}$, we obtain that $a_{k}=b_{k}$
for all $k$ is odd and $c_{1}^{2,\lambda_{2}-1}=0$ when $\lambda_{1}\leq\lambda_{a+1}$.
If $\lambda_{2}\leq\lambda_{a+1}<\lambda_{1}$, then computing
\[
[x+c_{1}^{2,\lambda_{2}-1}(\xi_{1}^{2,\lambda_{2}-1}-\xi_{2}^{1,\lambda_{1}-1}),\xi_{2}^{a+1,\lambda_{a+1}-\lambda_{2}}\pm\xi_{-(a+1)}^{2,0}]=0
\]
implies that $c_{1}^{2,\lambda_{2}-1}=0$. However, if $\lambda_{2}>\lambda_{a+1}$,
we calculate the commutator between $\xi_{1}^{2,\lambda_{2}-1}-\xi_{2}^{1,\lambda_{1}-1}$
and basis elements of the form $\xi_{i}^{j,\lambda_{j}-1-k}\pm\xi_{-j}^{i,\lambda_{i}-1-k}$
for $1\leq i\leq a$, $a+1\leq j\leq b$,$0\leq k\leq\min\{\lambda_{i},\lambda_{j}\}$:
\begin{align}
[\xi_{1}^{2,\lambda_{2}-1}-\xi_{2}^{1,\lambda_{1}-1},\xi_{i}^{j,\lambda_{j}-1-k}\pm\xi_{-j}^{i,\lambda_{i}-1-k}] & =\pm\xi_{-j}^{2,\lambda_{2}-1+\lambda_{1}-1-k}\pm\xi_{-j}^{1,\lambda_{1}-1+\lambda_{2}-1-k}\label{eq:osp general}\\
 & \ \ \pm\xi_{1}^{j,\lambda_{j}-1-k+\lambda_{2}-1}\pm\xi_{2}^{j,\lambda_{j}-1-k+\lambda_{1}-1}.\nonumber 
\end{align}
We have that all terms in (\ref{eq:osp general}) are equal to $0$
for $0\leq k\leq\min\{\lambda_{i},\lambda_{j}\}$. This implies that
$\xi_{1}^{2,\lambda_{2}-1}-\xi_{2}^{1,\lambda_{1}-1}$ commutes with
all other basis elements in $\mathfrak{g}^{e}$. Therefore, we deduce
that $\mathfrak{z}(\mathfrak{g}^{e})=S\oplus\langle\xi_{1}^{2,\lambda_{2}-1}-\xi_{2}^{1,\lambda_{1}-1}\rangle$
in this case and $\dim\mathfrak{z}(\mathfrak{g}^{e})=\lceil\frac{\max\left\{ \lambda_{1},\lambda_{a+1}\right\} -1}{2}\rceil+1$.

\noindent For the special case when $\xi_{1}^{2,\lambda_{2}-1}-\xi_{2}^{1,\lambda_{1}-1}\notin\mathfrak{z}(\mathfrak{g}_{1}^{e_{1}})$
and $\xi_{a+1}^{a+1,\lambda_{a+1}-1}-\xi_{-(a+1)}^{-(a+1),\lambda_{a+1}-1}\in\mathfrak{z}(\mathfrak{g}_{2}^{e_{2}})$,
then an element $z\in\mathfrak{z}(\mathfrak{g}^{e})$ is of the form
$z=x+c_{a+1}^{a+1,\lambda_{a+1}-1}(\xi_{a+1}^{a+1,\lambda_{a+1}-1}-\xi_{-(a+1)}^{-(a+1),\lambda_{a+1}-1})$
where $x$ is defined in (\ref{eq:x-in-centre-osp}) and $c_{a+1}^{a+1,\lambda_{a+1}-1}\in\mathbb{C}$
is the coefficient. If $\lambda_{1}\geq\lambda_{a+1}$, then computing
$[x,\xi_{1}^{a+1,0}+\varepsilon_{1}^{a+1,0}\xi_{-(a+1)}^{1,\lambda_{1}-\lambda_{a+1}}]=0$
gives that $a_{k}=b_{k}$ for all $k$ is odd and $1\leq k\leq\lambda_{a+1}-1$
and $c_{a+1}^{a+1,\lambda_{a+1}-1}=0$. However, if $\lambda_{1}<\lambda_{a+1}$,
computing $[x,\xi_{1}^{a+1,\lambda_{a+1}-\lambda_{1}}+\varepsilon_{1}^{a+1,\lambda_{a+1}-\lambda_{1}}\xi_{-(a+1)}^{1,0}]=0$
gives that $a_{k}=b_{k}$ for all $k$ is odd and $1\leq k\leq\lambda_{1}-1$.
It remains to check that $\xi_{a+1}^{a+1,\lambda_{a+1}-1}-\xi_{-(a+1)}^{-(a+1),\lambda_{a+1}-1}\in\mathfrak{z}(\mathfrak{g}^{e})$.
It is obvious that $\xi_{a+1}^{a+1,\lambda_{a+1}-1}-\xi_{-(a+1)}^{-(a+1),\lambda_{a+1}-1}$
commutes with all elements of the form $\xi_{i}^{j,\lambda_{j}-1-k}\pm\xi_{-j}^{i,\lambda_{i}-1-k}$
for $1\leq i,j\leq a$ or $a+1\leq i,j\leq b$. We now calculate commutators
between $\xi_{a+1}^{a+1,\lambda_{a+1}-1}-\xi_{-(a+1)}^{-(a+1),\lambda_{a+1}-1}$
and $\xi_{i}^{j,\lambda_{j}-1-k}\pm\xi_{-j}^{i,\lambda_{i}-1-k}$
for $1\leq i\leq a$, $a+1\leq j\leq b$, $1\leq k\leq\min$\{$\lambda_{i},\lambda_{j}$\}:
\begin{align}
[\xi_{a+1}^{a+1,\lambda_{a+1}-1}-\xi_{-(a+1)}^{-(a+1),\lambda_{a+1}-1},\xi_{i}^{j,\lambda_{j}-1-k}\pm\xi_{-j}^{i,\lambda_{i}-1-k}]=\xi_{i}^{a+1,2\lambda_{a+1}-2-k}\label{eq:osp general 2}\\
\pm\xi_{-(a+1)}^{i,\lambda_{a+1}-1+\lambda_{i}-1-k}.\nonumber 
\end{align}
 We have that all terms in (\ref{eq:osp general 2}) are equal to
$0$. This implies that $\xi_{a+1}^{a+1,\lambda_{a+1}-1}-\xi_{-(a+1)}^{-(a+1),\lambda_{a+1}-1}$
commutes with all other basis elements in $\mathfrak{g}^{e}$. Therefore,
we deduce that $\xi_{a+1}^{a+1,\lambda_{a+1}-1}-\xi_{-(a+1)}^{-(a+1),\lambda_{a+1}-1}\in\mathfrak{z}(\mathfrak{g}^{e})$
in this case and $\dim\mathfrak{z}(\mathfrak{g}^{e})=\lceil\frac{\max\left\{ \lambda_{1},\lambda_{a+1}\right\} -1}{2}\rceil+1$.

\noindent Moreover, when $\xi_{1}^{2,\lambda_{2}-1}-\xi_{2}^{1,\lambda_{1}-1}\in\mathfrak{z}(\mathfrak{g}_{1}^{e_{1}})$
and $\xi_{a+1}^{a+1,\lambda_{a+1}-1}-\xi_{-(a+1)}^{-(a+1),\lambda_{a+1}-1}\in\mathfrak{z}(\mathfrak{g}_{2}^{e_{2}})$,
applying a similar argument to above we also have that $\mathfrak{z}(\mathfrak{g}^{e})=S\oplus\langle\xi_{1}^{2,\lambda_{2}-1}-\xi_{2}^{1,\lambda_{1}-1}\rangle$
for $\lambda_{2}>\lambda_{a+1}$ and $\mathfrak{z}(\mathfrak{g}^{e})=S\oplus\langle\xi_{a+1}^{a+1,\lambda_{a+1}-1}-\xi_{-(a+1)}^{-(a+1),\lambda_{a+1}-1}\rangle$
for $\lambda_{1}<\lambda_{a+1}$. \end{proof} 
\end{singlespace}

\subsection{Adjoint action on $\mathfrak{osp}(m|2n)$\label{subsec:Adjoint-action-of-osp}}

\begin{singlespace}
\noindent Recall that $G=\mathrm{O}_{m}(\mathbb{C})\times\mathrm{Sp}_{2n}(\mathbb{C})$
and $e\in\mathfrak{g}_{\bar{0}}$ be nilpotent. Recall that the adjoint
action of $G$ on $\mathfrak{g}$ is given by $g\cdot x=gxg^{-1}$
for all $g\in G$, $x\in\mathfrak{g}$. Then the centralizer $G^{e}$
of $e$ in $G$ acts on $\mathfrak{g}^{e}$ by the adjoint action.
Consider $(\mathfrak{g}^{e})^{G^{e}}$ which is defined to be $(\mathfrak{g}^{e})^{G^{e}}=\{x\in\mathfrak{g}^{e}:gxg^{-1}=x\text{ for all }g\in G^{e}\}$.
The restriction of the adjoint action of $G^{e}$ to $(\mathfrak{g}^{e})^{G^{e}}$
is the trivial action, so $\mathrm{Ad}_{\mid(\mathfrak{g}^{e})^{G^{e}}}:G^{e}\rightarrow\mathrm{GL}((\mathfrak{g}^{e})^{G^{e}})$
is the trivial map and thus $\mathrm{ad}_{\mid(\mathfrak{g}^{e})^{G^{e}}}:\mathfrak{g}^{e}\rightarrow\mathfrak{gl}((\mathfrak{g}^{e})^{G^{e}})$
is the trivial map. Hence, the adjoint action of $\mathfrak{g}^{e}$
on $(\mathfrak{g}^{e})^{G^{e}}$ is trivial. Therefore, we have that
$(\mathfrak{g}^{e})^{G^{e}}\subseteq\mathfrak{z}(\mathfrak{g}^{e})$.
Since $\mathfrak{z}(\mathfrak{g}^{e})$ is stable under any automorphism
of $\mathfrak{g}^{e}$, it is stable under the adjoint action of $G^{e}$.
Take any $x\in\mathfrak{z}(\mathfrak{g}^{e})$, $g\in G^{e}$ and
$y\in\mathfrak{g}^{e}$, then $[g\cdotp x,y]=[g\cdotp x,g\cdot(g^{-1}\cdot y)]=g\cdotp[x,g^{-1}\cdotp y]$.
Since $g^{-1}\cdotp y\in\mathfrak{g}^{e}$, we have that $[g\cdotp x,y]=0$
as required. Therefore, we deduce that $(\mathfrak{g}^{e})^{G^{e}}=\left(\mathfrak{z}(\mathfrak{g}^{e})\right)^{G^{e}}$.
\end{singlespace}

\begin{singlespace}
According to \cite[Section 3.12]{Jantzen2004}, we have that $G^{e}$
is the semidirect product of the subgroup $C^{e}$ and the connected
normal subgroup $R^{e}$, i.e. $G^{e}\cong C^{e}\ltimes R^{e}$. Denote
the connected component of $G^{e}$ (resp. $C^{e}$) containing the
identity by $(G^{e})^{\circ}$ (resp. $(C^{e})^{\circ}$), we have
that $G^{e}/(G^{e})^{\circ}\cong C^{e}/(C^{e})^{\circ}$ based on
\cite[Section 3.13]{Jantzen2004}. 
\end{singlespace}
\begin{thm}
Let $\mathfrak{g}=\mathfrak{g}_{\bar{0}}\oplus\mathfrak{g}_{\bar{1}}=\mathfrak{osp}(m|2n)$
and $G=\mathrm{O}_{m}(\mathbb{C})\times\mathrm{Sp}_{2n}(\mathbb{C})$.
Let $e=e_{1}+e_{2}\in\mathfrak{g}_{\bar{0}}$ be nilpotent with Jordan
type $\lambda$ denoted as in (\ref{eq:jordan type group}). Let $S=\langle e^{l}:l\text{ is odd and }l\leq\max\{\lambda_{1},\lambda_{a+1}\}-1\rangle$,
then $\left(\mathfrak{z}(\mathfrak{g}^{e})\right)^{G^{e}}=S$.
\end{thm}

\noindent \begin{proof}Recall that $\{v_{i_{j}}\}$ is a basis of
$V$ according to (\ref{eq:vij-alternative dynkin}) and the bilinear
form $B(\ldotp,\ldotp)$ on $V$ is given in Subsection \ref{subsec:Alternative-Dynkin-pyramid}.
We define an involution $\rho$ on the labels of the basis of $V$
by 
\[
\rho(i_{k})=\begin{cases}
i_{-k} & \text{if }i=1,\dots,a;\\
-i_{-k} & \text{if }i=\pm(a+1),\dots,\pm b.
\end{cases}
\]
 Let $e_{i_{t},d_{l}}$ be the $(i_{t},d_{l})$-matrix unit. Then
$\{e_{i_{t},d_{l}}\pm e_{\rho(d_{l}),\rho(i_{t})}\}$ is a basis for
$\mathfrak{g}$ with appropriate signs and conditions on $i_{t},d_{l}$.

\noindent Let $l=\lfloor\frac{m}{2}\rfloor$. By reordering boxes
on the right hand half of $\tilde{P}$ from $1,\ldots,l+n$ and boxes
on the left hand half of $\tilde{P}$ from $-(l+n),\ldots,-1$ (note
that there exists a box labelled by $0$ if $m$ is odd), we get a
basis $\{e_{j,k}+\gamma_{-k,-j}e_{-k,-j}\}$ for $\mathfrak{g}$ as
defined in (\ref{eq:basis for osp}). There exists an isomorphism
between basis element $e_{j,k}+\gamma_{-k,-j}e_{-k,-j}$ and $e_{i_{t},d_{l}}\pm e_{\rho(d_{l}),\rho(i_{t})}$.

\noindent It is clear that $S\subseteq\left(\mathfrak{z}(\mathfrak{g}^{e})\right)^{G^{e}}$.
When $\mathfrak{z}(\mathfrak{g}^{e})=S$, then $\left(\mathfrak{z}(\mathfrak{g}^{e})\right)^{G^{e}}\subseteq\mathfrak{z}(\mathfrak{g}^{e})=S$
and thus we obtain $\left(\mathfrak{z}(\mathfrak{g}^{e})\right)^{G^{e}}=S$.
The following part of this proof deals with two special cases in Corollary
\ref{cor:centre osp} where $\mathfrak{z}(\mathfrak{g}^{e})\neq S$.

\noindent Case 1: When $a\geq3$, $\lambda_{2}>\lambda_{a+1}$, $\ensuremath{\left|1\right|}=\ensuremath{\left|2\right|}=\bar{0}$
or $a=2$, $\ensuremath{\left|1\right|}\neq\ensuremath{\left|2\right|}$.

\noindent In this case, the extra basis element $x$ can be written
as $e_{1_{-\lambda_{1}+1},2_{\lambda_{2}-1}}-e_{2_{-\lambda_{2}+1},1_{\lambda_{1}-1}}$.
We consider a matrix $Q$ which sends each $v_{1_{k}}$ to $-v_{1_{k}}$
and all other $v_{i_{_{k}}}$ for $i\neq1$ to itself. Obviously $S\subseteq\left(\mathfrak{z}(\mathfrak{g}^{e})\right)^{G^{e}}\subseteq\left(\mathfrak{z}(\mathfrak{g}^{e})\right)^{Q}$.
Then we have that 
\[
B(Qv_{i_{k}},Qv_{j_{l}})=\begin{cases}
(-1)^{k}\delta_{i_{k},\rho(j_{l})} & \text{if }1\leq i,j\leq a;\\
\delta_{i_{k},\rho(j_{l})} & \text{if }\pm(a+1)\leq i,j\leq\pm b;\\
0 & \text{otherwise}.
\end{cases}
\]
Thus $Q$ preserves the form $B$ on $V$. Hence, we have that $Q\in G$.
For any $i_{k},j_{l}$, observe that 
\[
Qe_{i_{k},j_{l}}Q^{-1}=\begin{cases}
-e_{i_{k},j_{l}} & \text{if }i\neq j\text{ and }i=1\text{ or }j=1;\\
e_{i_{k},j_{l}} & \text{otherwise.}
\end{cases}
\]
 we have that 
\begin{align*}
Q\cdotp e_{1} & =Q\cdotp(\sum_{t;i=1}^{a}\varepsilon_{i_{t},i_{t-2}}e_{i_{t},i_{t-2}})=\sum_{t;i=1}^{a}\varepsilon_{i_{t},i_{t-2}}Qe_{i_{t},i_{t-2}}Q^{-1}\\
 & =\sum_{t}\varepsilon_{1_{t},1_{t-2}}e_{1_{t},1_{t-2}}+\sum_{t;i\neq1}\varepsilon_{i_{t},i_{t-2}}e_{i_{t},i_{t-2}}=\sum_{t;i=1}^{a}\varepsilon_{i_{t},i_{t-2}}e_{i_{t},i_{t-2}}=e_{1}.
\end{align*}
 Furthermore, based on the way we defined $Q,$ it fixes $e_{2}$.
This implies that $Q\in G^{e}$ and thus we know that $Q\cdotp e^{l}=e^{l}$
for all $l$ is odd. Next we have that 
\[
Q\cdotp x=QxQ^{-1}=-e_{1_{-\lambda_{1}+1},2_{\lambda_{2}-1}}+e_{2_{-\lambda_{2}+1},1_{\lambda_{1}-1}}=-x.
\]
Thus we deduce that $x\notin\left(\mathfrak{z}(\mathfrak{g}^{e})\right)^{Q}$
and $\left(\mathfrak{z}(\mathfrak{g}^{e})\right)^{Q}\subseteq S$.
Therefore, we have that $\left(\mathfrak{z}(\mathfrak{g}^{e})\right)^{G^{e}}=\left(\mathfrak{z}(\mathfrak{g}^{e})\right)^{Q}=S$
and $\dim\left(\mathfrak{z}(\mathfrak{g}^{e})\right)^{G^{e}}=\lceil\frac{\lambda_{1}-1}{2}\rceil$.

\noindent Case 2: When $\lambda_{1}<\lambda_{a+1}$, $\lambda_{a+1}>\lambda_{a+2}$,
$\ensuremath{\left|a+1\right|}=\bar{0}$ and $\lambda_{a+1}$ is odd.

\noindent In this case, the extra basis element $x$ can be written
as 
\[
x=e_{(a+1)_{-\lambda_{a+1}+1},(a+1)_{\lambda_{a+1}-1}}-e_{-(a+1)_{-\lambda_{a+1}+1},-(a+1)_{\lambda_{a+1}-1}}.
\]
We consider a matrix $Q$ which sends $v_{(a+1)_{k}}$ to $v_{-(a+1)_{k}}$
and all other $v_{i_{_{k}}}$ for $i\neq\pm(a+1)$ to itself. Then
we have that 
\[
B(Qv_{i_{k}},Qv_{j_{l}})=\begin{cases}
B(v_{i_{k}},v_{j_{l}})=(-1)^{k}\delta_{i_{k},\rho(j_{l})} & \text{if }1\leq i,j\leq a;\\
B(v_{-i_{k}},v_{-j_{l}})=\delta_{-i_{k},\rho(-j_{l})}=\delta_{i_{k},\rho(j_{l})} & \text{if }i,j=\pm(a+1);\\
B(v_{i_{k}},v_{j_{l}})=\delta_{i_{k},\rho(j_{l})} & \text{if }\pm(a+2)\leq i,j\leq\pm b;\\
0 & \text{otherwise.}
\end{cases}
\]
Thus $Q$ preserves the form on $V$ and $Q\in G$. For any $i_{k},j_{l}$,
observe that 
\begin{equation}
Qe_{i_{k},j_{l}}Q^{-1}=\begin{cases}
e_{i_{k},j_{l}} & \text{if }i,j\neq\pm(a+1);\\
e_{-i_{k},j_{l}} & \text{if }i=\pm(a+1),j\neq\pm(a+1);\\
e_{i_{k},-j_{l}} & \text{if }i\neq\pm(a+1),j=\pm(a+1);\\
e_{-i_{k},-j_{l}} & \text{if }i,j=\pm(a+1).
\end{cases}\label{eq:QeQ-1}
\end{equation}
Note that we can write $e_{2}$ to be $\sum_{\substack{t;i=\pm(a+1)}
}^{b}\varepsilon_{i_{t},i_{t-2}}e_{i_{t},i_{t-2}}$. Then we have that 
\begin{align*}
Q\cdotp e_{2} & =Q\cdotp(\sum_{\substack{t;i=\pm(a+1)}
}^{b}\varepsilon_{i_{t},i_{t-2}}e_{i_{t},i_{t-2}})=\sum_{\substack{t;i=\pm(a+1)}
}^{b}\varepsilon_{i_{t},i_{t-2}}Qe_{i_{t},i_{t-2}}Q^{-1}\\
 & =\sum_{t}\varepsilon_{(a+1)_{t},(a+1)_{t-2}}e_{-(a+1)_{t},-(a+1)_{t-2}}+\sum_{t}\varepsilon_{-(a+1)_{t},-(a+1)_{t-2}}e_{(a+1)_{t},(a+1)_{t-2}}\\
 & +\sum_{t;i\neq\pm(a+1)}\varepsilon_{i_{t},i_{t-2}}e_{i_{t},i_{t-2}}\\
 & =\sum_{\substack{t;i=\pm(a+1)}
}^{b}\varepsilon_{i_{t},i_{t-2}}e_{i_{t},i_{t-2}}=e_{2}.
\end{align*}
 Furthermore, based on the way we defined $Q,$ it fixes $e_{1}$.
This implies that $Q\in G^{e}$ and thus we know that $Q\cdotp e^{l}=e^{l}$
for all $l$ is odd. Next we have that 
\[
Q\cdotp x=QxQ^{-1}=e_{-(a+1)_{-\lambda_{a+1}+1},-(a+1)_{\lambda_{a+1}-1}}-e_{(a+1)_{-\lambda_{a+1}+1},(a+1)_{\lambda_{a+1}-1}}=-x.
\]
Thus we know that $x\notin\left(\mathfrak{z}(\mathfrak{g}^{e})\right)^{Q}$
and $\left(\mathfrak{z}(\mathfrak{g}^{e})\right)^{Q}\subseteq S$.
Therefore, we have that $\left(\mathfrak{z}(\mathfrak{g}^{e})\right)^{G^{e}}=\left(\mathfrak{z}(\mathfrak{g}^{e})\right)^{Q}=S$
and $\dim\left(\mathfrak{z}(\mathfrak{g}^{e})\right)^{G^{e}}=\lceil\frac{\lambda_{a+1}-1}{2}\rceil$.\end{proof}

\begin{rem}
If we choose $G'=\mathrm{SO}_{m}(\mathbb{C})\times\mathrm{Sp}_{2n}(\mathbb{C})$,
applying a similar argument we obtain that $\left(\mathfrak{z}(\mathfrak{g}^{e})\right)^{G'^{e}}=\mathfrak{z}(\mathfrak{g}^{e})$
if $\lambda_{2}>\lambda_{a+1}$ and $\lambda_{i}$ is even for $3\leq i\leq b$,
or $\lambda_{1}<\lambda_{a+1}$, $\lambda_{a+1}\neq\lambda_{a+2}$
and $\lambda_{\pm(a+1)}$ are the only odd parts for $1\leq i\leq b$.
For other cases, we have $\left(\mathfrak{z}(\mathfrak{g}^{e})\right)^{G'^{e}}=S$.
\end{rem}

\subsection{Proof of the Theorems\label{subsec:explanation of osp}}

\begin{singlespace}
\noindent Let $\lambda=(\lambda_{1},\dots,\lambda_{r+s})$ be the
Jordan type of $e$ as defined in (\ref{eq:gl(m,n)-partition 1}).
Based on Subsections \ref{subsec:cc-osp}--\ref{subsec:Adjoint-action-of-osp},
we have that $\left(\mathfrak{z}(\mathfrak{g}^{e})\right)^{G^{e}}=\lceil\frac{\lambda_{1}-1}{2}\rceil$. 
\end{singlespace}

\begin{singlespace}
In order to prove Theorem 1 for $\mathfrak{g}=\mathfrak{osp}(m|2n)$,
we calculate the number of labels in the labelled Dynkin diagram $\varDelta$
which are equal to $2$. Given a ortho-symplectic Dynkin pyramid $P$
and a partition $\lambda$ as defined in (\ref{eq:gl(m,n)-partition 2}),
let $r_{i}$ (resp. $s_{i}$) be the number of boxes with parity $\bar{0}$
(resp. $\bar{1}$) on the $i$th column. We observe that $\varDelta$
has no label equal to $1$ if and only if all parts of $\lambda$
are odd or all parts of $\lambda$ are even. Now assume that all labels
in $\varDelta$ equal to $0$ or $2$. Then based on the way that
labelled Dynkin diagram is constructed, we observe that $n_{2}(\varDelta)=\lfloor\frac{\lambda_{1}}{2}\rfloor$.
Note that when $\lambda_{1}$ is even, we have that $\lceil\frac{\lambda_{1}-1}{2}\rceil=\frac{\lambda_{1}}{2}$
and $n_{2}(\varDelta)=\lfloor\frac{\lambda_{1}}{2}\rfloor=\frac{\lambda_{1}}{2}$,
when $\lambda_{1}$ is odd, we have that $\lceil\frac{\lambda_{1}-1}{2}\rceil=\frac{\lambda_{1}-1}{2}$
and $n_{2}(\varDelta)=\frac{\lambda_{1}-1}{2}$. Therefore, we deduce
that $\left(\mathfrak{z}(\mathfrak{g}^{e})\right)^{G^{e}}=n_{2}(\varDelta)$. 

To calculate $\mathfrak{z}(\mathfrak{g}^{h})$, we first consider
the following example:
\end{singlespace}
\begin{example}
\begin{singlespace}
\noindent For a nilpotent element $e\in\mathfrak{g}_{\bar{0}}$ with
Jordan type $(5,3,1|3,3)$, the corresponding Dynkin pyramid is shown
in Example \ref{exa:(5,3,1|3,3)=000026(3,3|4)}. Hence, we have $n_{2}(\varDelta)=\lfloor\frac{5}{2}\rfloor=2$.
We can calculate that $\mathfrak{g}^{h}=\mathfrak{gl}(1|0)\oplus\mathfrak{gl}(2|2)\oplus\mathfrak{osp}(3|2)$.
Note that $\mathfrak{z}(\mathfrak{gl}(1|0))=I_{1}$, $\mathfrak{z}(\mathfrak{gl}(2|2))=I_{4}$
and $\mathfrak{z}(\mathfrak{osp}(3|2))=0$. Therefore, we have that
$\dim\mathfrak{z}(\mathfrak{g}^{h})=2$.
\end{singlespace}
\end{example}

\begin{singlespace}
Assume that there is no label equal to $1$ in $\varDelta$. When
all parts of $\lambda$ are odd, we observe that 
\[
\mathfrak{g}^{h}=\bigoplus_{i>0}\mathfrak{gl}(r_{i}|s_{i})\oplus\mathfrak{osp}(r_{0}|s_{0}).
\]
 When all parts of $\lambda$ are even , then $\mathfrak{g}^{h}$
is just the direct sum of $\mathfrak{gl}(r_{i}|s_{i})$ for $i>0$.
Note that $\mathfrak{z}(\mathfrak{gl}(r_{i}|s_{i}))=I_{r_{i}+s_{i}}$
and $\mathfrak{z}(\mathfrak{osp}(r_{0}|s_{0}))=0$. Hence, we have
that $\mathfrak{z}(\mathfrak{g}^{h})$ is the direct sum of $I_{r_{i}+s_{i}}$
for $i>0$. Since there are in total $\lambda_{1}$ columns in $P$
with non-zero boxes, we deduce that $\dim\mathfrak{z}(\mathfrak{g}^{h})=\lfloor\frac{\lambda_{1}}{2}\rfloor=n_{2}(\varDelta)$. 

Next we describe the relation between $\dim\left(\mathfrak{z}(\mathfrak{g}^{e})\right)^{G^{e}}$
and the sum of labels $\sum a_{i}$ in $\varDelta$. Note that $a_{i}=\mathrm{col}(i+1)-\mathrm{col}(i)$
for $i=1,\dots,l+n-1$. When $m$ is even, we have 
\[
a_{l+n}=\begin{cases}
-2\mathrm{col}(l+n) & \text{if }\left|l+n\right|=\bar{1};\\
-\mathrm{col}(l+n)-\mathrm{col}(l+n-1) & \text{if }\left|l+n\right|=\bar{0}.
\end{cases}
\]

\end{singlespace}

\begin{singlespace}
\noindent Then 
\begin{align*}
\sum a_{i} & =\sum_{i=1}^{l+n-1}\left(\mathrm{col}(i+1)-\mathrm{col}(i)\right)+\begin{cases}
-2\mathrm{col}(l+n) & \text{if }\left|l+n\right|=\bar{1};\\
(\mathrm{col}(-l-n)-\mathrm{col}(l+n-1)) & \text{if }\left|l+n\right|=\bar{0}
\end{cases}\\
 & =\begin{cases}
\mathrm{col}(-l-n)-\mathrm{col}(1) & \text{if }\left|l+n\right|=\bar{1};\\
-\mathrm{col}(l+n-1)-\mathrm{col}(1) & \text{if }\left|l+n\right|=\bar{0}.
\end{cases}
\end{align*}
 Hence, we have that $\sum a_{i}=\lambda_{1}-1$ or $\lambda_{1}$.
\end{singlespace}

\begin{singlespace}
When $m$ is odd, then there exists a box that is labelled by $0$,
we have $a_{l+n}=\mathrm{col}(0)-\mathrm{col}(l+n)$. Then
\begin{align*}
\sum a_{i} & =\sum_{i=1}^{l+n-1}\left(\mathrm{col}(i+1)-\mathrm{col}(i)\right)+(\mathrm{col}(0)-\mathrm{col}(l+n))\\
 & =\mathrm{col}(0)-\mathrm{col}(1)=\lambda_{1}-1.
\end{align*}
 Hence, we deduce that 
\[
\lceil\frac{1}{2}\sum a_{i}\rceil=\begin{cases}
\lceil\frac{\lambda_{1}}{2}\rceil & \text{\text{if} all parts of the Jordan type of }e\text{ are even;}\\
\lceil\frac{\lambda_{1}-1}{2}\rceil & \text{otherwise.}
\end{cases}
\]
 
\end{singlespace}

\begin{singlespace}
\noindent Therefore, we deduce that $\dim\left(\mathfrak{z}(\mathfrak{g}^{e})\right)^{G^{e}}=\lceil\frac{1}{2}\sum a_{i}\rceil$.
\end{singlespace}

\begin{singlespace}
In order to prove Theorem 3 for $\mathfrak{g}$, we consider two general
cases below.

\textbf{Case 1:} When $\varDelta$ has no label equal to $1$, i.e.
all labels are equal to $0$ or $2$. In this case, we know that either
all parts of $\lambda$ are odd or all parts of $\lambda$ are even.
Note that $e_{0}=0$ since $\varDelta_{0}$ has all labels equal to
$0$. Thus we have that 
\[
\mathfrak{g}_{0}^{e_{0}}=\mathfrak{g}_{0}=\left(\bigoplus_{i>0}\mathfrak{sl}(r_{i}|s_{i})\right)\oplus\mathfrak{osp}(r_{0}|s_{0})
\]
We denote $c_{i}=r_{i}+s_{i}$. Then
\[
\dim\mathfrak{g}_{0}^{e_{0}}=\dim\mathfrak{g}_{0}=\sum_{i>0}\dim\mathfrak{sl}(r_{i}|s_{i})+\dim\mathfrak{osp}(r_{0}|s_{0}).
\]
 Now if all parts of $\lambda$ are even, then we have $\dim\mathfrak{osp}(r_{0}|s_{0})=0$.
This implies that $\dim\mathfrak{g}_{0}^{e_{0}}=\sum_{i>0}(c_{i}^{2}-1).$
We also have that $\arrowvert\{i:\lambda_{i}\text{ is odd},\ensuremath{\left|i\right|}=\bar{0}\}\mid=\arrowvert\{i:\lambda_{i}\text{ is odd},\ensuremath{\left|i\right|}=\bar{1}\}\mid=0$.
Thus $\dim\mathfrak{g}^{e}=\frac{1}{2}\dim\mathfrak{gl}(m\vert2n)^{e}=\frac{1}{2}(\sum_{i\in\mathbb{Z}}c_{i}^{2})$
by Subsection \ref{subsec:c-osp}. Hence, we have that 
\[
\dim\mathfrak{g}^{e}-\dim\mathfrak{g}_{0}^{e_{0}}=\frac{1}{2}(\sum_{i\in\mathbb{Z}}c_{i}^{2})-\sum_{i>0}(c_{i}^{2}-1)=\sum_{i>0:c_{i}>0}1=n_{2}(\varDelta).
\]
Next we consider when all parts of $\lambda$ are odd, we have that
\[
\dim\mathfrak{osp}(r_{0}|s_{0})=\frac{(r_{0}+s_{0})^{2}-(r_{0}+s_{0})}{2}+s_{0}=\frac{c_{0}^{2}-r_{0}+s_{0}}{2}.
\]
Note that in this case $\dim\mathfrak{g}^{e}=\frac{1}{2}\dim\mathfrak{gl}(m|n)^{e}-\frac{r}{2}+\frac{s}{2}$
by Subsection \ref{subsec:c-osp} where $r$ (resp. $s$) is the total
number of $\lambda_{i},\ensuremath{\left|i\right|}=\bar{0}$ (resp.
$\lambda_{i},\ensuremath{\left|i\right|}=\bar{1}$). Since all parts
of $\lambda$ are odd, then each row which corresponds to a certain
part of $\lambda$ has a box in the $0$th column. Hence, the number
of $\lambda_{i}$ with $\ensuremath{\left|i\right|}=\bar{0}$ (resp.
$\ensuremath{\left|i\right|}=\bar{1}$) is equal to the number of
even (resp. odd) boxes in the $0$th column, i.e. $r=r_{0}$ and $s=s_{0}$.
Thus we can calculate
\[
\dim\mathfrak{g}^{e}=\frac{1}{2}(\sum_{i\in\mathbb{Z}}c_{i}^{2}+c_{0}^{2})-\frac{r_{0}}{2}+\frac{s_{0}}{2}.
\]
Therefore, we have that 
\begin{align*}
\dim\mathfrak{g}^{e}-\dim & \mathfrak{g}_{0}^{e_{0}}=\left(\frac{1}{2}(\sum_{i\in\mathbb{Z}}c_{i}^{2}+c_{0}^{2})-\frac{r_{0}}{2}+\frac{s_{0}}{2}\right)-\left(\sum_{i>0}(c_{i}^{2}-1)+\frac{c_{0}^{2}-r_{0}+s_{0}}{2}\right)\\
 & =\sum_{i>0:c_{i}>0}1=n_{2}(\varDelta).
\end{align*}

\end{singlespace}

\begin{singlespace}
\noindent When $r_{i}\neq s_{i}$ for all $i>0$, we have that $\dim\mathfrak{z}(\mathfrak{g}_{0}^{e_{0}})=\dim\mathfrak{z}(\mathfrak{g}_{0})=0$
because $\mathfrak{z}(\mathfrak{sl}(r_{i}|s_{i}))=0$ for all $i$
and $\mathfrak{z}(\mathfrak{osp}(r_{0}|s_{0}))=0$. However, if there
exists $r_{i}=s_{i}$ for some $i$, we have that $I_{r_{i}\mid r_{i}}\in\mathfrak{z}(\mathfrak{sl}(r_{i}|s_{i}))$
and thus $\dim\mathfrak{z}(\mathfrak{g}_{0}^{e_{0}})=\tau$ where
$\tau$ is the number of $i>0$ for which $r_{i}=s_{i}$. Hence, we
have that $\dim\left(\mathfrak{z}(\mathfrak{g}^{e})\right)^{G^{e}}-\dim\left(\mathfrak{z}(\mathfrak{g}_{0}^{e_{0}})\right)^{G_{0}^{e_{0}}}=n_{2}(\varDelta)-\tau$.
\end{singlespace}

\begin{singlespace}
\textbf{Case 2:} When there exist some labels equal to $1$ in $\varDelta$.
Note that there are in total $2\lambda_{1}+1$ columns in $P$ and
let $n_{2}(\varDelta)=t$. Observe that $\varDelta$ has some labels
equal to $2$ if there exist some $c_{k_{j}}=0$ for $j=1,\dots,t$
and $t=n_{2}(\varDelta)$. Let $k>0$ be the minimal column number
such that $c_{k}=0$ and $k+2t=\lambda_{1}$. Note that once a label
equal to $1$ occurs, say $a_{k}$, then there is no label equal to
$2$ for all $a_{h}$ with $h>k$. Then we have that 
\[
\mathfrak{g}_{0}\cong\left(\bigoplus_{i=1}^{t}\mathfrak{sl}(r_{k+2i-1}|s_{k+2i-1})\right)\oplus\mathfrak{osp}(\sum_{i=1-k}^{k-1}r_{i}\mid\sum_{i=1-k}^{k-1}s_{i}).
\]
Hence,
\[
\dim\mathfrak{g}_{0}^{e_{0}}=\sum_{i=k+1}^{\lambda_{1}}\left(c_{i}^{2}+c_{i}c_{i+1}-1\right)+\dim\mathfrak{osp}(\sum_{i=1-k}^{k-1}r_{i}\mid\sum_{i=1-k}^{k-1}s_{i}).
\]
Denote the Jordan type of $e_{0}$ to be $\lambda^{0}$ which is defined
the similar way to (\ref{eq:gl(m,n)-partition 1}). We also observe
that $\ensuremath{\left|\{i:\lambda_{i}\text{ is odd},\ensuremath{\left|i\right|}=\bar{0}\}\right|}=\ensuremath{\left|\{i:\lambda_{i}^{0}\text{ is odd},\ensuremath{\left|i\right|}=\bar{0}\}\right|}$
and $\ensuremath{\left|\{i:\lambda_{i}\text{ is odd},\ensuremath{\left|i\right|}=\bar{1}\}\right|}=\left|\{i:\lambda_{i}^{0}\text{ is odd},\ensuremath{\left|i\right|}=\bar{1}\}\right|$.
Note that 
\begin{align*}
\dim\mathfrak{osp}(\sum_{i=1-k}^{k-1}r_{i}\vert\sum_{i=1-k}^{k-1}s_{i}) & =\frac{1}{2}(\sum_{i=1-k}^{k-1}(c_{i}^{2}+c_{i}c_{i+1}))\\
 & -\frac{1}{2}\left|\{i:\lambda_{i}^{0}\text{ is odd},\ensuremath{\left|i\right|}=\bar{0}\}\right|\\
 & +\frac{1}{2}\left|\{i:\lambda_{i}^{0}\text{ is odd},\ensuremath{\left|i\right|}=\bar{1}\}\right|.
\end{align*}
 Therefore, we have that $\dim\mathfrak{g}^{e}-\dim\mathfrak{g}_{0}^{e_{0}}=\sum_{i=1}^{t}1=t=n_{2}(\varDelta)$.
Moreover, let $\mathfrak{\hat{g}}_{0}=\mathfrak{osp}(\sum_{i=1-k}^{k-1}r_{i}\mid\sum_{i=1-k}^{k-1}s_{i})$.
We observe that the projection of $e_{0}$ in each $\mathfrak{sl}(r_{k+2i-1}|s_{k+2i-1})$
is $0$ and so $e_{0}\in\mathfrak{\hat{g}}_{0}$. Thus 
\[
\dim\mathfrak{z}(\mathfrak{g}_{0}^{e_{0}})=\sum_{i=1}^{t}\dim\mathfrak{z}(\mathfrak{sl}(r_{k+2i-1}|s_{k+2i-1}))+\text{dim}\mathfrak{z}(\mathfrak{\hat{g}}_{0}^{e_{0}}).
\]
 Similar to Case 1, $\mathfrak{z}(\mathfrak{sl}(r_{k+2i-1}|s_{k+2i-1}))=0$
when $r_{k+2i-1}\neq s_{k+2i-1}$ for all $i$. If there exist some
$i>0$ such that $r_{k+2i-1}=s_{k+2i-1}$, then $I_{r_{k+2i-1}\vert r_{k+2i-1}}\in\mathfrak{z}(\mathfrak{sl}(r_{k+2i-1}|s_{k+2i-1}))$.
We know that $\dim\left(\mathfrak{z}(\mathfrak{g}_{0}^{e_{0}})\right)^{G_{0}^{e_{0}}}=\lceil\frac{k_{1}-1}{2}\rceil+\tau$
where $\tau$ is the number of positive column number $i$ such that
$r_{i}=s_{i}$. Hence, we deduce that $\dim\left(\mathfrak{z}(\mathfrak{g}^{e})\right)^{G^{e}}-\dim\left(\mathfrak{z}(\mathfrak{g}_{0}^{e_{0}})\right)^{G_{0}^{e_{0}}}=n_{2}(\varDelta)-\tau$
in this case.
\end{singlespace}

\end{document}